\documentclass[a4paper,11pt,reqno,noindent]{amsart}
\usepackage[centertags]{amsmath}
\usepackage{amsfonts,amssymb,amsthm,dsfont,cases,amscd,esint,enumerate}
\usepackage[T1]{fontenc}
\usepackage[english]{babel}
\usepackage[applemac]{inputenc}
\usepackage{newlfont}
\usepackage{color}
\usepackage[body={15cm,21.5cm},centering]{geometry} 
\usepackage{fancyhdr}
\usepackage{enumerate}

\theoremstyle{plain}
\newtheorem{theor10}{Theorem}
\newenvironment{theor1}
  {\pushQED{\qed}\begin{theor10}}
  {\popQED\end{theor10}}
\newtheorem{lem10}[theor10]{Lemma}
\newenvironment{lem1}
  {\pushQED{\qed}\begin{lem10}}
  {\popQED\end{lem10}}
\newtheorem{theor0}{Theorem}[section]
\newenvironment{theor}
  {\pushQED{\qed}\begin{theor0}}
  {\popQED\end{theor0}}
\newtheorem{lem0}[theor0]{Lemma}
\newenvironment{lem}
  {\pushQED{\qed}\begin{lem0}}
  {\popQED\end{lem0}}
\newtheorem{prop0}[theor0]{Proposition}
\newenvironment{prop}
  {\pushQED{\qed}\begin{prop0}}
  {\popQED\end{prop0}}
\newtheorem{cor0}[theor0]{Corollary}

\theoremstyle{definition}
\newtheorem{rem0}[theor0]{Remark}
\newenvironment{rem}
  {\pushQED{\qed}\begin{rem0}}
  {\popQED\end{rem0}}

\theoremstyle{plain}
\newtheorem{as0}[theor0]{Assumption}

\newtheorem*{asn0*}{\assumptionnumber}
  \providecommand{\assumptionnumber}{}
  \makeatletter
  \newenvironment{asn0}[2]
   {\renewcommand{\assumptionnumber}{Assumption \!(#1) {\normalfont--- #2}}
    \begin{asn0*}
    \protected@edef\@currentlabel{{\normalfont(#1)}}}
   {\end{asn0*}}
  \makeatother
\newenvironment{asn}
  {\pushQED{\qed}\begin{asn0}}
  {\popQED\end{asn0}}

\numberwithin{equation}{section}

\newcommand{\e}{\varepsilon}
\newcommand{\Exc}{\operatorname{Exc}}
\newcommand{\Pc}{\mathcal{P}}
\newcommand{\Lc}{\mathcal{L}}
\newcommand{\Ec}{\mathcal E}
\newcommand{\dist}{\operatorname{dist}}
\newcommand{\R}{\mathbb R}
\newcommand{\Z}{\mathbb Z}
\newcommand{\Ic}{\mathcal I}
\newcommand{\Md}{\mathbb M}
\newcommand{\Nc}{\mathcal N}
\newcommand{\cvf}{\rightharpoonup}

\newcommand{\loc}{{\operatorname{loc}}}
\newcommand{\Id}{\operatorname{Id}}
\newcommand{\D}{\operatorname{D}}
\newcommand{\ee}{e}
\newcommand{\supess}{\operatorname{sup\,ess}}
\newcommand{\infess}{\operatorname{inf\,ess}}

\newcommand{\bb}{\bar{\boldsymbol b}}
\newcommand{\Bb}{\bar{\boldsymbol B}}

\newcommand{\Ld}{\operatorname{L}}
\newcommand{\Div}{{\operatorname{div}}}
\newcommand{\Sym}{{\operatorname{sym}}}
\newcommand{\Skew}{{\operatorname{skew}}}
\newcommand{\Tr}{{\operatorname{tr}}}
\newcommand{\step}[1]{\noindent \textit{Step} #1.}
\newcommand{\substep}[1]{\noindent \textit{Substep} #1.}
\newcommand{\Pm}{\mathbb{P}}

\newcommand{\E}{\mathbb{E}}
\newcommand{\expec}[1]{\mathbb{E}\left[ #1 \right]}
\newcommand{\expecm}[1]{\mathbb{E}\big[ #1 \big]}
\newcommand{\expecmm}[1]{\mathbb{E}\,[ #1 ]}
\newcommand{\expecM}[1]{\mathbb{E}\bigg[ #1 \bigg]}
\newcommand{\var}[1]{\mathrm{Var}\left[#1\right]}

\usepackage[colorlinks,citecolor=black,urlcolor=black]{hyperref}

\title[Quantitative homogenization of random suspensions]{Quantitative homogenization theory\\for random suspensions in steady Stokes flow}

\author[M. Duerinckx]{Mitia Duerinckx}
\address[Mitia Duerinckx]{Universit\'e Paris-Saclay, CNRS, Laboratoire de Math\'ematiques d'Orsay, 91400~Orsay, France
\& University of California, Los Angeles, Departement of Mathematics, CA 90095, USA
\& Universit\'e Libre de Bruxelles, D\'epartement de Math\'ematique, 1050~Brussels, Belgium}
\email{mitia.duerinckx@u-psud.fr}
\author[A. Gloria]{Antoine Gloria}
\address[Antoine Gloria]{Sorbonne Universit\'e, CNRS, Universit\'e de Paris, Laboratoire Jacques-Louis Lions, 75005~Paris, France \& Institut Universitaire de France \& Universit\'e Libre de Bruxelles, D\'epartement de Math\'ematique, 1050~Brussels, Belgium}
\email{gloria@ljll.math.upmc.fr}

\begin{document}
\selectlanguage{english}

\begin{abstract}
This work develops a quantitative homogenization theory for random suspensions of rigid particles in a steady Stokes flow, and completes recent qualitative results. More precisely, we establish a large-scale regularity theory for this Stokes problem, and we prove moment bounds for the associated correctors and optimal estimates on the homogenization error; the latter further requires a quantitative ergodicity assumption on the random suspension. Compared to the corresponding quantitative homogenization theory for divergence-form linear elliptic equations, substantial difficulties arise from the analysis of the fluid incompressibility and the particle rigidity constraints.
Our analysis further applies to the problem of stiff inclusions in (compressible or incompressible) linear elasticity and in electrostatics; it is also new in those cases, even in the periodic setting.

\bigskip\noindent
{\sc MSC-class:}
35R60; 76M50; 35Q35; 76D03; 76D07.
\end{abstract}
\maketitle

\setcounter{tocdepth}{1}
\tableofcontents

\section{Introduction}

We start with the formulation of the steady Stokes model describing a viscous fluid in presence of a random suspension of small rigid particles, see e.g.~\cite{DG-19}.
Throughout, we denote by $d\ge 2$ the space dimension, we consider a given random set $\Ic=\bigcup_nI_n\subset\R^d$, where $\{I_n\}_n$ stands for the different particles, and we denote by $x_n$ the barycenter of $I_n$. Ergodicity, hardcore, and regularity assumptions are listed in Section~\ref{sec:main-res}.
To model a dense suspension of small particles, we rescale the random set~$\Ic$ by a small parameter $\e>0$ and consider $\e\Ic=\bigcup_n\e I_n$.
Next, we view these small particles $\{\e I_n\}_n$ as suspended in a solvent described by the steady Stokes equation: in a reference domain $U\subset\R^d$, given an internal forcing $f\in \Ld^2(U)^d$, the fluid velocity $u_\e$ satisfies
\begin{equation}\label{e.intro1}
-\triangle u_\e+\nabla P_\e=f,\qquad\Div( u_\e)=0,\qquad\text{in $U\setminus\e\Ic$},
\end{equation}
with $u_\e=0$ on $\partial U$.
(Assume for the moment that no particle intersects the boundary~$\partial U$.)
No-slip conditions are imposed at particle boundaries: since particles are constrained to have rigid motions, this amounts to letting the velocity field~$u_\e$ be extended inside particles, with the rigidity constraint
\begin{equation}\label{e.intro+2}
\D(u_\e)=0,\qquad\text{in $\e\Ic$},
\end{equation}
where $\D(u_\e)$ denotes the symmetrized gradient of $u_\e$; in other words, this condition means that the velocity field $u_\e$ coincides with a rigid motion $V_{\e,n}+\Theta_{\e,n}(x-\e x_n)$ inside each particle $\e I_n$ (centered at $\e x_n$), for some $V_{\e,n}\in\R^d$ and some skew-symmetric matrix $\Theta_{\e,n}\in\R^{d\times d}$.
Finally, assuming that the particles have the same mass density as the fluid, buoyancy forces vanish, and
the force and torque balances on each particle take the form
\begin{align}
&\int_{\e\partial I_n}\sigma(u_\e,P_\e)\nu=0,\label{eq:force-bal}\\
&\int_{\e\partial I_n}\Theta(x-\e x_n)\cdot\sigma(u_\e,P_\e)\nu=0,\quad\text{for all skew-symmetric $\Theta\in\R^{d\times d}$},\label{e.intro+3}
\end{align}
in terms of the Cauchy stress tensor
\begin{equation}\label{eq:Cauchy}
\sigma(u_\e,P_\e)=2\D(u_\e)-P_\e\Id,
\end{equation}
where $\nu$ stands for the outward unit normal vector at the particle boundaries.
In the physically relevant 3D case, skew-symmetric matrices $\Theta\in\R^{3\times 3}$ are equivalent to cross products $\theta\times$ with $\theta\in\R^3$, and equations recover their usual form.

\medskip
In the companion article~\cite{DG-19}, we proved that in the macroscopic limit $\e \downarrow 0$ the velocity and pressure fields $(u_\e,P_\e)$ converge weakly to $(\bar u, \bar P+\bb:\D(\bar u))$, where $(\bar u,\bar P)$ solves the homogenized equation 
\begin{equation}\label{e.intro2} 
\quad\left\{\begin{array}{ll}
-\Div(2\Bb \D(\bar u))+\nabla\bar P=(1-\lambda) f ,&\text{in $U$},\\
\Div(\bar u)=0,&\text{in $U$},\\
\bar u=0,&\text{on $\partial U$},
\end{array}\right.
\end{equation}
for some effective viscosity tensor $\Bb$ and some effective matrix $\bb$, where $\lambda=\expec{\mathds1_\Ic}$ denotes the volume fraction of the suspension.
The aim of the present contribution is twofold:
\begin{enumerate}[(I)]
\item We make this qualitative convergence result quantitative by optimally estimating the error between $(u_\e,P_\e)$ and a two-scale expansion based on $(\bar u, \bar P+\bb:\nabla \bar u)$ in terms of suitable correctors, cf.~Theorem~\ref{th:hom-quant} below.
\smallskip\item We develop a large-scale regularity theory for the Stokes problem~\eqref{e.intro1}--\eqref{e.intro+3}, which ensures that on large scales the solution~$u_\e$ has the same 
regularity properties as the solution $\bar u$ of the limiting equation \eqref{e.intro2} (both in terms of $C^{1,1-}$ Schauder theory and in terms of $\Ld^p$ regularity), cf.~Theorems~\ref{th:schauder},~\ref{th:lp-reg}, and~\ref{th:ann-reg} below.
\end{enumerate}
On the one hand, part~(I) provides the optimal quantitative version of~\cite{DG-19}, estimating the error in the homogenization process.
This is proved under a strong mixing assumption on the random suspension~$\Ic$, which is conveniently formulated in form of a multiscale variance inequality in the spirit of~\cite{DG1,DG2}.
On the other hand, part~(II)
makes precise the intuitive idea that the Stokes problem~\eqref{e.intro1}--\eqref{e.intro+3} should inherit the regularity properties of the limiting equation~\eqref{e.intro2} on sufficiently large scales, which is expressed intrinsically in terms of the growth of correctors.
This is established under a mere ergodicity assumption, but the results are only of practical use provided that a strong control is available on the minimal necessary large scales: this requires a quantitative control on the growth of correctors and is established here for convenience again under a strong mixing assumption in form of a multiscale variance inequality.

\medskip
Our main motivation to develop a large-scale regularity theory for~\eqref{e.intro1}--\eqref{e.intro+3} stems from the sedimentation problem for a random suspension in a Stokes flow under a constant gravity field $\ee\in\R^d$,  in which case the force balance~\eqref{eq:force-bal} is  replaced by
\[\fint_{\e\partial I_n}\sigma(u_\e,P_\e)\nu\,+\ee\,=\,0.\]
Since energy is then pumped into the system, na\"ive energy estimates blow up, and the analysis crucially relies on stochastic cancellations.
Annealed $\Ld^p$ regularity in form of Theorem~\ref{th:ann-reg} below constitutes the main technical input in~\cite{DG-20} for our analysis of this sedimentation problem.
More precisely, in a general non-dilute regime, this allows us to obtain the first rigorous proof of the celebrated predictions by Batchelor~\cite{Batchelor-72} and by Caflisch and Luke~\cite{Caflisch-Luke} on the effective sedimentation speed and on individual velocity fluctuations, thus significantly extending the perturbative results of~\cite{Gloria-19} (see also~\cite{Hofer-19}).

\medskip
Although the present contribution primarily focusses on random suspensions of rigid particles in a steady Stokes flow, we point out 
that our arguments apply more generally to homogenization problems with stiff inclusions.
First note that equation~\eqref{e.intro1} can be written in the equivalent form
\begin{equation}\label{eq:rewr-eqn-sigma}
-\Div\big(\sigma(u_\e,P_\e)\big)=f,\qquad\Div (u_\e)=0,\qquad\text{in $U\setminus \e\Ic$},
\end{equation}
with $u_\e=0$ on $\partial U$, where we recall that $\sigma(u_\e,P_\e)$ denotes the Cauchy stress tensor~\eqref{eq:Cauchy},
and the equation is completed by the rigidity constraint $\D(u_\e)=0$ on the inclusions $\e\Ic$
and by the boundary conditions \eqref{eq:force-bal}--\eqref{e.intro+3}.
We mention a few physical models that can be obtained as a slight modification of the above:
\begin{enumerate}[\quad$\bullet$]
\item \emph{Incompressible linear elasticity with stiff inclusions} takes the same form, with the Cauchy stress tensor replaced by $\sigma(u_\e,P_\e)=K\D(u_\e)-P_\e\Id$,
in terms of the constant stiffness tensor $K$ of the background material (satisfying the Legendre--Hadamard condition).
Surprisingly, the qualitative homogenization of this problem is quite recent and follows from~\cite{DG-19}.\footnote{In this problem it might make more sense to include the internal forcing $f$ in the boundary conditions, replacing~\eqref{eq:force-bal} by $\int_{\e\partial I_n}\sigma(u_\e,P_\e)\nu+\int_{\e I_n} f=0$. In that case, the forcing term in the homogenized problem \eqref{e.intro2} is~$f$ rather than $(1-\lambda)f$; this is only a minor change in the analysis.}
\smallskip\item \emph{Compressible linear elasticity with stiff inclusions} is obtained by dropping the incompressibility constraint $\Div (u_\e)=0$ in~\eqref{eq:rewr-eqn-sigma}, and replacing the Cauchy stress tensor by $\sigma(u_\e)=K\D(u_\e)$, in terms of the constant stiffness tensor $K$ of the background material.
In this case, qualitative homogenization follows from~\cite[Chapter~3]{JKO94};
see also \cite{Braides-Garroni-95} for a compactness result in a corresponding nonlinear setting.
\smallskip\item \emph{Linear electrostatics with stiff inclusions} amounts to taking $u_\e$ scalar-valued, dropping the incompressibility constraint in~\eqref{eq:rewr-eqn-sigma}, and replacing the Cauchy stress-tensor by $\sigma(u_\e)=K\nabla u_\e$, in terms of the constant conductivity matrix $K$ of the background material. We refer to~\cite[Chapter~3]{JKO94} for the qualitative homogenization of this problem (under weaker hardcore conditions).
\end{enumerate}
Our present \emph{quantitative} analysis covers all these other models for the first time.
Our results are new even in the periodic setting (that is, when $\Ic$ is a periodic set), in which case Theorems~\ref{th:cor} and~\ref{th:hom-quant} below hold with $\mu_d\equiv 1$.

\bigskip
Before turning to precise statements of our results, we discuss the context.
The present contribution constitutes a natural extension
to the steady Stokes problem~\eqref{e.intro1}--\eqref{e.intro+3}
of the by-now well-developed quantitative homogenization
theory for the model case of divergence-form linear elliptic equations with random coefficients. This theory was started in~\cite{GO1,GO2,GNO1,GNO2}, with statements close to Theorem~\ref{th:hom-quant} below under
similar mixing conditions,
while a more mature theory was later initiated in~\cite{AS} based on large-scale regularity.
For recent developments, we refer the reader to the recent monograph~\cite{AKM-book}, based on~\cite{AM-16,AKM1,AKM2}, and to the alternative series of works~\cite{GNO-reg,GNO-quant,GO4,DO1,JO}.
In the present contribution,
we consider for simplicity a strong mixing assumption in form of a multiscale variance inequality~\cite{DG1,DG2},
and large-scale regularity is established
by following the approach in~\cite{GNO-reg,GNO-quant,DO1}, but we believe that the one in~\cite{AKM-book} could be used as well (see~\cite[Appendix]{DG-20c}).
We further focus for simplicity on the weakly correlated setting: as inspired by~\cite{JO}, this allows to bypass part of the argument in~\cite{GNO-reg,GNO-quant} by appealing to deterministic regularity (in form of Meyers' perturbative estimates) and to a buckling argument based on stochastic cancellations, which makes the proof particularly short and efficient.
The strongly correlated setting could be treated by following~\cite{GNO-reg}, but it would substantially increase both the technicality and the length of the argument.

\medskip
Compared to the model case of divergence-form linear elliptic equations with random coefficients, we face three main additional difficulties in this work:
\begin{enumerate}[\quad---]
\item
the rigidity constraint on the particles makes the canonical structure of fluxes and flux correctors less obvious: as in~\cite{D-20}, fluxes are constructed via a nontrivial extension procedure, which is crucial to obtain optimal convergence rates;
\smallskip\item
naïve two-scale expansions are incompatible with the rigidity constraint on the particles, thus requiring some local surgery;
\smallskip\item
the incompressibility of the fluid gives rise to the pressure in the equation and makes many estimates more involved.
\end{enumerate}

\subsection*{Notation}
\begin{enumerate}[\quad$\bullet$]
\item For vector fields $u,u'$ and matrix fields $T,T'$, we set $(\nabla u)_{ij}=\partial_ju_i$, $\Div( T)_i=\partial_jT_{ij}$, $T\!\!:\!\!T'=T_{ij}T'_{ij}$, $(u\otimes u')_{ij}=u_iu'_j$, where we systematically use Einstein's summation convention on repeated indices. We also write {$\partial_Eu=E:\nabla u$} for any matrix $E$.
\smallskip\item For a vector field~$u$ and scalar field $P$, we denote by $(\D(u))_{ij}=\frac12(\partial_ju_i+\partial_iu_j)$ the symmetrized gradient and we recall the notation $\sigma(u,P)=2\D(u)-P\Id$ for the Cauchy stress tensor. We also recall that $\nu$ stands for the outward unit normal vector at particle boundaries.
\smallskip\item We denote by $\Md_0\subset\R^{d\times d}$ the subset of trace-free matrices, by $\Md_0^\Sym$ the subset of symmetric trace-free matrices, and by $\Md^\Skew$ the subset of skew-symmetric matrices.
\smallskip\item We denote by $C\ge1$ any constant that only depends on dimension $d$, on the constant $\delta>0$ in Assumption~\ref{Hd} below, on the weight $\pi$ in Assumption~\ref{Mix+} if applicable, and on the reference domain~$U$. We use the notation $\lesssim$ (resp.~$\gtrsim$) for $\le C\times$ (resp.~\mbox{$\ge\frac1C\times$}) up to such a multiplicative constant $C$. We write $\ll$ (resp.~$\gg$) for $\le \frac1C\times$ (resp.~\mbox{$\ge C\times$}) up to a sufficiently large multiplicative constant $C$.
We add subscripts to $C,\lesssim,\gtrsim,\ll,\gg$ in order to indicate dependence on other parameters.
\smallskip\item The ball centered at $x$ of radius $r$ in $\R^d$ is denoted by $B_r(x)$, and we simply write $B(x)=B_1(x)$, $B_r=B_r(0)$, and $B=B_1(0)$.
\smallskip\item For a function $f$, we write $[f]_2(x):=(\fint_{B(x)}|f|^2)^{1/2}$ for the local moving quadratic averages at the unit scale.
\smallskip\item We set $\langle r \rangle =(1+r^2)^{1/2}$ for $r \ge 0$, we set $\langle x\rangle=(1+|x|^2)^{1/2}$ for $x\in\R^d$, and we similarly write $\langle\nabla\rangle=(1-\triangle)^{1/2}$.
\end{enumerate}

\section{Main results}\label{sec:main-res}

\subsection{Assumptions}

Given an underlying probability space~$(\Omega,\Pm)$,
let $\Pc=\{x_n\}_n$ be a random point process on $\R^d$, and consider a collection of random shapes $\{I_n^\circ\}_n$, where each~$I_n^\circ$ is a connected random Borel subset of the unit ball $B$ and is centered at $0$ in the sense of $\int_{I_n^\circ}x\,dx=0$. We then define the corresponding inclusions $I_n:=x_n+I_n^\circ$ centered at the points of $\Pc$,
and we consider the random set $\Ic:=\bigcup_nI_n$.
We also denote by $I_n^+$ the convex hull of $I_n$, hence $I_n\subset I_n^+\subset B(x_n)$.
Throughout, we make the following general assumptions, for some fixed deterministic constant $\delta>0$.

\begin{asn}{H$_\delta$}{General conditions}\label{Hd}$ $
\begin{enumerate}[\quad$\bullet$]
\item \emph{Stationarity and ergodicity:} The random set $\Ic$ is stationary and ergodic.
\smallskip\item \emph{Uniform $C^2$ regularity:} The random shapes $\{I_n^\circ\}_n$ satisfy interior and exterior ball conditions with radius $\delta$ almost surely.
\smallskip\item \emph{Uniform hardcore condition:} There holds $(I_n^++\delta B)\cap(I_m^++\delta B)=\varnothing$ almost surely for all $n\ne m$.
\qedhere
\end{enumerate}
\end{asn}

In view of quantitative homogenization results, we need to further consider quantitative ergodicity assumptions, which we make here for simplicity in form of a multiscale variance inequality as we introduced in~\cite{DG1,DG2}.

\begin{asn}{Mix$^+$}{Quantitative mixing condition}\label{Mix+}$ $\\
There exists a non-increasing weight function $\pi:\R^+\to\R^+$ with superalgebraic decay (that is, $\pi(\ell)\le C_p\langle\ell\rangle^{-p}$ for all $p<\infty$) such that
the random set $\Ic$ satisfies, for all $\sigma(\Ic)$-measurable random variables~$Y(\Ic)$,
\begin{equation}\label{eq:SGL}
\var{Y(\Ic)}\,\le\,\expec{\int_0^\infty\int_{\R^d}\Big(\partial^{\operatorname{osc}}_{\Ic,B_\ell(x)}Y(\Ic)\Big)^2dx\,\langle\ell\rangle^{-d}\pi(\ell)\,d\ell},
\end{equation}
where the ``oscillation'' $\partial^{\operatorname{osc}}$ of the random variable $Y(\Ic)$ is defined by
\begin{align*}
&\partial^{\operatorname{osc}}_{\Ic,B_\ell(x)}Y(\Ic)\,:=\,\supess\Big\{Y(\Ic'):\Ic'\cap(\R^d\setminus B_\ell(x))=\Ic\cap(\R^d\setminus B_\ell(x))\Big\}\\
&\hspace{4cm}-\infess\Big\{Y(\Ic'):\Ic'\cap(\R^d\setminus B_\ell(x))=\Ic\cap(\R^d\setminus B_\ell(x))\Big\}.\qedhere
\end{align*}
\end{asn}

\subsection{Corrector estimates}

We first recall the suitable definition of correctors associated with the steady Stokes problem~\eqref{e.intro1}--\eqref{e.intro+3}, as introduced in the companion work~\cite[Proposition~2.1]{DG-19}.

\begin{lem1}[Correctors; \cite{DG-19}]\label{lem:cor}
Under Assumption~\ref{Hd}, for all $E\in\Md_0$, there exists a unique solution $(\psi_E,\Sigma_E)$ to the following infinite-volume corrector problem:
\begin{enumerate}[\quad$\bullet$]
\item Almost surely, $(\psi_E,\Sigma_E)$ belongs to $H^1_\loc(\R^d)^d\times\Ld^2(\R^d\setminus\Ic)$ and satisfies in the strong sense,
\begin{equation}\label{eq:corr-Stokes}
\qquad\qquad\left\{\begin{array}{ll}
-\triangle\psi_E+\nabla\Sigma_E=0,&\text{in $\R^d\setminus\Ic$},\\
\Div(\psi_E)=0,&\text{in $\R^d\setminus\Ic$},\\
\D(\psi_E+Ex)=0,
&\text{in $\Ic$},\\
\int_{\partial I_n}\sigma\big(\psi_E+E(x-x_n),\Sigma_E\big)\nu=0,&\forall n,\\
\int_{\partial I_n}\Theta(x-x_n)\cdot\sigma\big(\psi_E+E(x-x_n),\Sigma_E\big)\nu=0,&\forall n,\,\forall\Theta\in\Md^\Skew.
\end{array}\right.
\end{equation}
\item The gradient field $\nabla\psi_E$ and the pressure field $\Sigma_E\mathds1_{\R^d\setminus\Ic}$ are stationary, they have vanishing expectation, they have finite second moments, and $\psi_E$ satisfies the anchoring condition $\fint_B \psi_E=0$ almost surely.
\end{enumerate}
In addition, the corrector $\psi_E$ is sublinear at infinity, that is, $\e\psi_E(\frac\cdot\e)\cvf{}0$ in $H^1_\loc(\R^d)^d$ almost surely as $\e\downarrow0$. Note that $(\psi_E,\Sigma_E)=(\psi_{E^\Sym},\Sigma_{E^\Sym})$ where $E^\Sym$ denotes the symmetric part of $E$.
\end{lem1}

As a key tool for quantitative homogenization, we establish the following moment bounds on correctors.
Inspired by the corresponding corrector estimates for divergence-form linear elliptic equations in~\cite{JO}, the proof is based on the analysis of stochastic cancellations for large-scale averages of the corrector gradient, together with perturbative annealed $\Ld^p$ regularity and a buckling argument.
If the weight $\pi$ in Assumption~\ref{Mix+} has some stretched exponential decay, then the moment bounds below can be upgraded to corresponding stretched exponential moments.

\begin{theor1}[Corrector estimates]\label{th:cor}
Under Assumptions~\ref{Hd} and~\ref{Mix+}, for all~$E\in\Md_0$ and $q<\infty$, we have
\begin{equation}\label{eq:bnd-grad-cor}
\|[(\nabla\psi_E,\Sigma_E\mathds1_{\R^d\setminus\Ic})]_2\|_{\Ld^q(\Omega)}\,\lesssim_q\, |E|,
\end{equation}
and
\begin{equation}\label{eq:bnd-cor}
\|[\psi_E]_2(x)\|_{\Ld^q(\Omega)}\,\lesssim_q\, |E|\,\mu_d(|x|),\qquad\mu_d(r):=\left\{\begin{array}{lll}
1&:&d>2,\\
\log(2+r)^\frac12&:&d=2,\\
\langle r\rangle^\frac12&:&d=1.
\end{array}\right.
\end{equation}
In particular, in dimension $d>2$, up to relaxing the anchoring condition, the solution $\psi_E$ of the infinite-volume problem~\eqref{eq:corr-Stokes} can be uniquely constructed itself as a stationary field with vanishing expectation.
\end{theor1}
\begin{rem}
We include the case $d=1$ in the statements for completeness, in which case the problem is scalar without incompressibility constraint.
\end{rem}

\subsection{Large-scale regularity}\label{sec:reg0}
Given a random forcing $g\in C^\infty_c(\R^d;\Ld^\infty(\Omega)^{d\times d})$, we consider the unique solution $(\nabla u_g,P_g)\in\Ld^\infty(\Omega;\Ld^2(\R^d)^{d\times d}\times\Ld^2(\R^d\setminus\Ic))$ of the following steady Stokes problem,
\begin{equation}\label{eq:test-v0}
\left\{\begin{array}{ll}
-\triangle u_g+\nabla P_g=\Div (g),&\text{in $\R^d\setminus\Ic$},\\
\Div(u_g)=0,&\text{in $\R^d\setminus\Ic$},\\
\D(u_g)=0,&\text{in $\Ic$},\\
\int_{\partial I_{n}}\big(g+\sigma(u_g,P_g)\big)\nu=0,&\forall n,\\
\int_{\partial I_{n}}\Theta(x-x_n)\cdot\big(g+\sigma(u_g,P_g)\big)\nu=0,&\forall n,\,\forall\Theta\in\Md^\Skew.
\end{array}\right.
\end{equation}
The energy inequality yields, almost surely,
\begin{equation}\label{eq:energy}
\|\nabla u_g\|_{\Ld^2(\R^d)}\le\|g\|_{\Ld^2(\R^d\setminus\Ic)}.
\end{equation}
Aside from Meyers' perturbative improvements of this energy inequality, cf.~Section~\ref{sec:Meyers}, and aside from local regularity theory, no other regularity estimates are expected to hold in general in a deterministic form due to the presence of the rigidity constraints on the random set of particles --- except in a dilute regime when particles are sufficiently far apart, cf.~Remark~\ref{rem:reg-hofer}.
However, in view of homogenization, the heterogeneous Stokes problem~\eqref{eq:test-v0} can be replaced on large scales by a homogenized one in form of~\eqref{e.intro2}.
Since standard constant-coefficient regularity theory is available for this large-scale approximation,  the solution to~\eqref{eq:test-v0} should enjoy improved regularity properties on large scales.
This type of result was pioneered by Avellaneda and Lin~\cite{Avellaneda-Lin-87,Avellaneda-Lin-91} in the context of periodic homogenization in the model setting of divergence-form linear elliptic equations.
In the stochastic case, while early contributions in form of annealed Green's function estimates appeared in~\cite{Delmotte-Deuschel-05,MaO}, a quenched large-scale regularity theory was first outlined by Armstrong and Smart~\cite{AS}, and later fully developed in~\cite{AM-16,AKM1,AKM2,AKM-book} and in~\cite{GNO-reg,GNO-quant,GO4}.
We also mention the useful reformulation in form of annealed regularity in~\cite{DO1}.
Based on these ideas, we develop corresponding quenched large-scale and annealed regularity theories for the steady Stokes problem~\eqref{eq:test-v0}, which constitute the key technical ingredient in our work~\cite{DG-20} on sedimentation.

\medskip
We start with a quenched large-scale Schauder theory.
Hölder norms are reformulated à la Campanato in terms of the growth of local integrals, and the latter are restricted to scales $\ge r_*$ for some (well-controlled) random minimal radius~$r_*$.
As is natural, note that Hölder regularity is measured by replacing Euclidean coordinates $x\mapsto Ex$ by their heterogeneous versions $x\mapsto \psi_E(x)+Ex$ in terms of the corrector~$\psi_E$.

\begin{theor1}[Quenched large-scale Schauder theory]\label{th:schauder}
Under Assumption~\ref{Hd}, given \mbox{$\alpha\in(0,1)$}, there exists an almost surely finite stationary random field $r_*\ge1$ on $\R^d$, see~\eqref{e.def-r*}, such that the following holds: For all $g\in C^\infty_c(\R^d;\Ld^\infty(\Omega)^{d\times d})$ and $R\ge r_*(0)$, if~$\nabla u_g$ is a solution of the steady Stokes problem~\eqref{eq:test-v0} in $B_R$, then the following large-scale Lipschitz estimate holds on scales $\ge r_*(0)$,
\begin{equation}\label{eq:Lip}
\sup_{r_*(0)\le r\le R} \fint_{B_r} |\nabla u_g|^2 \,\lesssim\, \fint_{B_R} |\nabla u_g|^2+ \sup_{r_*(0)\le r\le R} \Big(\frac R r\Big)^{2\alpha}\fint_{B_r} \Big|g-\fint_{B_r}g\Big|^2,
\end{equation}
as well as the following large-scale $C^{1,\alpha}$ estimate,
\begin{equation}\label{eq:C1a}
\sup_{r_*(0)\le r\le R}\frac1{r^{2\alpha}}\Exc(\nabla u_g;B_r) \,\lesssim\, \frac1{R^{2\alpha}} \Exc(\nabla u_g;B_R)+\sup_{r_*(0)\le r\le R} \frac1{r^{2\alpha}} \fint_{B_r} \Big|g-\fint_{B_r}g\Big|^2,
\end{equation}
where the excess is defined by
\begin{equation}\label{eq:def-exc}
\Exc(h;D)\,:=\,\inf_{E\in\Md_0}\fint_D|h-(\nabla\psi_E+E)|^2.
\end{equation}
Under Assumption~\ref{Mix+}, the so-called minimal radius $r_*(0)$ satisfies \mbox{$\expec{r_*(0)^q}<\infty$} for all $q<\infty$.
\end{theor1}
As in~\cite{Armstrong-Daniel-16},~\cite[Section~7]{AKM-book}, \cite[Corollary~4]{GNO-reg}, or~\cite[Proposition~6.4]{DO1}, the above large-scale Lipschitz regularity~\eqref{eq:Lip} can be exploited together with a Calder\'on--Zygmund argument to deduce the following $\Ld^p$ regularity estimate on scales $\ge r_*$.
\begin{theor1}[Quenched large-scale $\Ld^p$ regularity]\label{th:lp-reg}
Under Assumption~\ref{Hd}, there exists an almost surely finite stationary random field $r_*\ge1$ on $\R^d$ as in Theorem~\ref{th:schauder} such that the following holds: For all $g\in C^\infty_c(\R^d;\Ld^\infty(\Omega)^{d\times d})$ and $1<p<\infty$, the solution $(\nabla u_g,P_g)$ of the steady Stokes problem~\eqref{eq:test-v0} satisfies
\[\bigg(\int_{\R^d}\Big(\fint_{B_*(x)}|\nabla u_g|^2\Big)^\frac p2dx\bigg)^\frac1p\,\lesssim_p\,\bigg(\int_{\R^d}\Big(\fint_{B_*(x)}|g|^2\Big)^\frac p2dx\bigg)^\frac1p,\]
where we use the short-hand notation $B_*(x):=B_{r_*(x)}(x)$.
\end{theor1}

As in~\cite{DG-20}, we establish the following annealed version of the above quenched large-scale~$\Ld^p$~regularity statement.
The main merit of this estimate is that a stochastic $\Ld^q(\Omega)$ norm appears inside the spatial $\Ld^p(\R^d)$ norm and allows to remove local quadratic averages on the random minimal scale~$r_*$ (up to a tiny loss of stochastic integrability), which is particularly convenient for applications.

\begin{theor1}[Annealed $\Ld^p$ regularity]\label{th:ann-reg}
Under Assumptions~\ref{Hd} and~\ref{Mix+}, for all $g\in C^\infty_c(\R^d;\Ld^\infty(\Omega)^{d\times d})$, $1<p,q<\infty$, and $\eta>0$, the solution $(\nabla u_g,P_g)$ of the steady Stokes problem~\eqref{eq:test-v0} satisfies
\begin{equation}\label{eq:ann-reg-st}
\|[\nabla u_g]_2\|_{\Ld^p(\R^d;\Ld^q(\Omega))}\,\lesssim_{p,q,\eta}\,\|[g]_2\|_{\Ld^p(\R^d;\Ld^{q+\eta}(\Omega))}.
\end{equation}
In addition, under only Assumption~\ref{Hd} (without need of Assumption~\ref{Mix+}), a Meyers' perturbative result holds without loss of stochastic integrability: there exists $C_0\simeq1$ such that the same estimate~\eqref{eq:ann-reg-st} holds with $\eta=0$ provided $|p-2|,|q-2|<\frac1{C_0}$.
\end{theor1}

\begin{rem}[Deterministic $\Ld^p$ regularity in dilute regime]\label{rem:reg-hofer}
In the dilute regime, the recent work of Höfer~\cite{Hofer-19} on the reflection method
easily yields the following version of the above; the proof is a direct adaptation of~\cite{Hofer-19} and is omitted.
This also constitutes a variant of the dilute Green's function estimates in~\cite[Lemma~2.7]{Gloria-19}.\\
{\it Under assumption~\ref{Hd}, we denote by $\delta(\Ic)\ge2\delta$ the minimal interparticle distance in $\Ic$.
For all $1<p,q<\infty$, there exists a constant $\delta_{p}>0$ (only depending on $d,p$) such that, provided $\Ic$ is dilute enough in the sense of $\delta(\Ic)\ge\delta_{p}$, the following holds: Given a random forcing $g\in\Ld^\infty(\Omega;C^\infty_c(\R^d)^{d\times d})$, the solution $(\nabla u_g,P_g)$ of the steady Stokes problem~\eqref{eq:test-v0} satisfies
\[\|\nabla u_g\|_{\Ld^p(\R^d;\Ld^q(\Omega))}\,\lesssim_{p,q}\,\|g\|_{\Ld^p(\R^d;\Ld^q(\Omega))},\]
as well as the following deterministic estimate, almost surely,
\[\|\nabla u_g\|_{\Ld^p(\R^d)}\,\lesssim_{p}\,\|g\|_{\Ld^p(\R^d)}.\qedhere\]
}
\end{rem}

\subsection{Quantitative homogenization result}\label{sec:homog}
We consider a steady Stokes fluid in a domain $U\subset\R^d$ with some internal forcing and with a dense suspension of small particles, cf.~\eqref{e.intro1}--\eqref{e.intro+3}, and we analyze the fluid velocity in the non-dilute homogenization regime with vanishing particle size but fixed volume fraction. Suspended particles in the fluid act as obstacles and hinder the fluid flow, thus increasing the flow resistance, that is, the viscosity. The system is then expected to behave approximately like an homogeneous Stokes fluid with some effective viscosity, cf.~\eqref{e.intro2}. This was the basis of Perrin's celebrated experiment to estimate the Avogadro number as inspired by Einstein's PhD thesis~\cite{Einstein-06}.

\medskip\noindent
Before stating the homogenization result, given a reference domain $U$, the set of particles must be modified to avoid particles intersecting the boundary: we consider the random set~$\Nc_\e(U)$ of all indices $n$ 
such that $\e(I_n^++\delta B)\subset U$, and we define
\[\Ic_\e(U):=\bigcup_{n\in\Nc_\e(U)}\e I_n.\]
Particles in this collection are of size $O(\e)$ and are at distance at least $\e\delta$ from the boundary~$\partial U$ and from one another, cf.~\ref{Hd}. We may now turn to the statement of the quantitative homogenization result, which provides an optimal quantitative version of the homogenization result in the companion article~\cite{DG-19}.

\begin{theor1}[Quantitative homogenization result]\label{th:hom-quant}
Under Assumptions~\emph{\ref{Hd}} and~\emph{\ref{Mix+}},
given a smooth bounded domain $U\subset\R^d$ and a forcing $f\in W^{1+\alpha,\infty}(U)^d$ for some $\alpha>0$, consider for all $\e>0$ the unique solution \mbox{$(u_\e,P_\e)\in\Ld^\infty(\Omega;H^1_0(U)^d\times\Ld^2(U\setminus\Ic_\e(U)))$} of the steady Stokes problem
\begin{equation}\label{eq:Stokes}
\left\{\begin{array}{ll}
-\triangle u_\e+\nabla P_\e=f,&\text{in $U\setminus\Ic_\e(U)$},\\
\Div (u_\e)=0,&\text{in $U\setminus\Ic_\e(U)$},\\
u_\e=0,&\text{on $\partial U$},\\
\D(u_\e)=0,&\text{in $\Ic_\e(U)$},\\
\int_{\e\partial I_n}\sigma(u_\e,P_\e)\nu=0,&\forall n\in\Nc_\e(U),\\
\int_{\e\partial I_n}\Theta(x-x_n)\cdot\sigma(u_\e,P_\e)\nu=0,&\forall n\in\Nc_\e(U),\,\forall\Theta\in\Md^\Skew.
\end{array}\right.
\end{equation}
Also consider the unique solution $(\bar u,\bar P)\in H^1_0(U)^d\times\Ld^2(U)$ of the corresponding homogenized Stokes problem
\begin{equation}\label{eq:Stokes-hom}
\quad\left\{\begin{array}{ll}
-\Div(2\Bb \D(\bar u))+\nabla\bar P=(1-\lambda) f ,&\text{in $U$},\\
\Div(\bar u)=0,&\text{in $U$},\\
\bar u=0,&\text{on $\partial U$},
\end{array}\right.
\end{equation}
where $\lambda:=\expec{\mathds 1_{\Ic}}$ denotes the volume fraction of the suspension
and where the effective viscosity tensor $\Bb$ is positive definite on $\Md_0^\Sym$ and is given by
\begin{equation}\label{eq:def-B}
\Bb:= \sum_{E,E'\in\Ec}(E'\otimes E)\,\expec{(\D(\psi_{E'})+E'):(\D(\psi_{E})+E)},
\end{equation}
where the sum runs over an orthonormal basis $\Ec$ of $\Md_0^\Sym$ and where we recall that the corrector $(\psi_E,\Sigma_E)$ is defined in Lemma~\ref{lem:cor}.
Then,
the following quantitative corrector result holds for all $q<\infty$,
\begin{align}\label{eq:hom-quant0}
&\bigg\|u_\e-\bar u-\e\sum_{E\in\Ec}\psi_E(\tfrac\cdot\e)\partial_E\bar u\bigg\|_{\Ld^q(\Omega;H^1(U))}\nonumber\\
+&\inf_{\kappa\in\R}~\bigg\|P_\e-\bar P-\bb:\D(\bar u)
-\sum_{E\in\Ec}(\Sigma_E\mathds1_{\R^d\setminus\Ic})(\tfrac\cdot\e)\partial_E\bar u-\kappa\bigg\|_{\Ld^q(\Omega;\Ld^2(U\setminus\Ic_\e(U)))}\nonumber\\
&\hspace{8cm}\,\lesssim_{\alpha,q}\,
\big(\e\mu_d(\tfrac1\e)\big)^\frac12\,\|f\|_{W^{1+\alpha,\infty}(U)},
\end{align}
where we recall that $\mu_d$ is defined in~\eqref{eq:bnd-cor}
and where the effective matrix $\bb\in\Md_0^\Sym$ is given~by
\begin{equation}\label{eq:def-b}
\bb:E\,:=\, \frac1d\,\expecM{\sum_n\frac{\mathds1_{I_n}}{|I_n|}\int_{\partial I_n}(x-x_n)\cdot\sigma(\psi_E+Ex,\Sigma_E)\nu}.
\end{equation}
In addition, if $f$ and $\bar u$ are compactly supported in $U$, then boundary layers disappear and the bound~\eqref{eq:hom-quant0} holds with the optimal convergence rate $\e\mu_d(\frac1\e)$.
\end{theor1}

\section{Perturbative annealed regularity}\label{sec:Meyers}

This section is devoted to the proof of the  Meyers-type perturbative result stated in Theorem~\ref{th:ann-reg}.

\begin{samepage}
\begin{theor}[Perturbative annealed $\Ld^p$ regularity]\label{th:meyers-ann}
Under Assumption~\ref{Hd}, there exists a constant $C_0\simeq1$ such that the following holds:
For all $g\in C^\infty_c(\R^d;\Ld^\infty(\Omega)^{d\times d})$, the solution $(\nabla u_g,P_g)$ of the Stokes problem~\eqref{eq:test-v0} satisfies for all~$p,q$ with $|p-2|,|q-2|\le\frac1{C_0}$,
\[\|[\nabla u_g]_2\|_{\Ld^p(\R^d;\Ld^q(\Omega))}\,\lesssim\,\|[g]_2\|_{\Ld^p(\R^d;\Ld^q(\Omega))}.\qedhere\]
\end{theor}
\end{samepage}

\subsection{Preliminary}
We start with a number of PDE tools that are useful in the proof.

\vspace{-0.3cm}
\subsubsection{Whole-space weak formulations}
The steady Stokes problem~\eqref{eq:test-v0} can be reformulated as an equation on the whole space, where particles generate source terms concentrated at their boundaries. This reformulation is particularly convenient for our computations.

\begin{lem}\label{lem:reform-whole}
The solution $(\nabla u_g,P_g)$ of the steady Stokes problem~\eqref{eq:test-v0} satisfies in the weak sense in the whole space~$\R^d$,
\begin{equation}\label{e:up-whole0}
-\triangle u_g+\nabla (P_g\mathds1_{\R^d\setminus\Ic})\,=\,\Div(g\mathds1_{\R^d\setminus\Ic})-\sum_{n} \delta_{\partial I_n}\big(g+\sigma(u_g,P_g)\big)\nu.
\end{equation}
Likewise, the corrector~$(\psi_E,\Sigma_E)$ in Lemma~\ref{lem:cor} satisfies in the weak sense in $\R^d$,
\begin{equation}\label{eq:psiE-ref0}
-\triangle\psi_E+\nabla(\Sigma_E\mathds1_{{\R^d\setminus\Ic}})=-\sum_n\delta_{\partial I_n}\,\sigma(\psi_E+Ex,\Sigma_E)\nu.
\qedhere
\end{equation}
\end{lem}

\begin{proof}
We focus on the proof of~\eqref{e:up-whole0}, while the argument for~\eqref{eq:psiE-ref0} is similar.
Given $\zeta\in C^\infty_c(\R^d)^d$, testing equation~\eqref{eq:test-v0} with $\zeta$ and integrating by parts on ${\R^d\setminus\Ic}$, we find
\begin{equation}\label{eq:reform-prepre}
\int_{{\R^d\setminus\Ic}}\nabla\zeta:\nabla u - \int_{{\R^d\setminus\Ic}}\Div(\zeta)\,P
\,=\,-\int_{\R^d\setminus\Ic}\nabla\zeta: g-\sum_{n}\int_{\partial I_{n}}(\zeta\otimes\nu):(g+\nabla u-P\Id).
\end{equation}
The claim~\eqref{e:up-whole0} follows provided we prove that
\begin{equation}\label{ag+1}
\int_{\Ic}\nabla\zeta:\nabla u \,=\,-\sum_{n}\int_{\partial I_{n}}(\nu \otimes \zeta):\nabla u.
\end{equation}
Indeed, adding the latter to~\eqref{eq:reform-prepre} yields the claim~\eqref{e:up-whole0}, in view of
\[\int_{\partial I_{n}}(\nu \otimes \zeta+\zeta\otimes\nu):\nabla u=\int_{\partial I_{n}}\zeta\otimes\nu:2\D(u).\]
We turn to the proof of~\eqref{ag+1}.
Since $u$ is affine in $I_{n}$, Stokes' formula yields
\[\int_{\partial I_{n}}(\nu\otimes\zeta):\nabla u=\int_{\partial I_{n}}\zeta_i\nu\cdot\partial_iu=\int_{I_{n}}\Div(\zeta_i\partial_iu)=\int_{I_{n}}\nabla\zeta_i\cdot\partial_iu.\]
The relation $\D(u)=0$ on $I_{n}$ entails that $\nabla u$ is skew-symmetric in $I_n$, so that the above becomes
\[\int_{\partial I_{n}}(\nu\otimes\zeta):\nabla u=-\int_{I_{n}}\nabla\zeta_i\cdot\nabla u_i=-\int_{I_{n}}\nabla\zeta:\nabla u,\]
and the claim~\eqref{ag+1} follows.
\end{proof}

\subsubsection{Localized pressure estimates}
We establish the following localized pressure estimate for the steady Stokes problem~\eqref{eq:test-v0}.
It follows from standard pressure estimates in~\cite{Galdi}, but as in~\cite[Proof of Proposition~2.1]{DG-19} some additional care is needed to make it uniform with respect to the size of $D$ although~$\Ic$ consists of an unbounded number of components; a short proof is included for convenience.

\begin{lem}[\cite{Galdi,DG-19}]\label{lem:pres}
Given a deterministic point set $\{x_{n}\}_n$ satisfying the hardcore and regularity conditions in~\ref{Hd}, for all $g\in \Ld^2_\loc(\R^d)^{d\times d}$ and all balls $D\subset\R^d$, any solution $(u_g,P_g)$ of the steady Stokes problem~\eqref{eq:test-v0} in $D$ satisfies for all $1<p<\infty$,
\begin{equation*}
\Big\|P_g-\fint_{D\setminus\Ic}P_g\Big\|_{\Ld^p(D\setminus\Ic)}\,\lesssim_p\, \|(\nabla u_g,g)\|_{\Ld^p(D\setminus\Ic)}.\qedhere
\end{equation*}
\end{lem}

\begin{proof}
We split the proof into two steps.

\medskip
\step1 Preliminary: There is a vector field $S\in W^{1,p'}_0(D)^d$ such that $S|_{I_{n}}$ is constant for all $n$ and such that
\begin{gather*}
\Div(S)\,=\,\Big(Q|Q|^{p-2}-\fint_{D\setminus\Ic}Q|Q|^{p-2}\Big)\mathds1_{D\setminus\Ic},
\qquad Q:=P_g-\fint_{D\setminus\Ic}P_g,\\
\|\nabla S\|_{\Ld^{p'}(D)}\,\lesssim_p\,\|Q|Q|^{p-2}\|_{\Ld^{p'}(D\setminus\Ic)}\,=\,\|Q\|_{\Ld^p(D\setminus\Ic)}^{p-1},
\end{gather*}
where we emphasize that the prefactor in the last estimate is uniformly bounded independently of $D$.

\medskip\noindent
By a standard use of the Bogovskii operator in form of~\cite[Theorem~III.3.1]{Galdi}, there exists a vector field~$S^\circ\in W^{1,p'}_0(D)^d$ such that
\begin{gather*}
\Div(S^\circ)\,=\,\Big(Q|Q|^{p-2}-\fint_{D\setminus\Ic}Q|Q|^{p-2}\Big)\mathds1_{D\setminus\Ic},\\
\|\nabla S^\circ\|_{\Ld^{p'}(D)}\,\lesssim_p\,\|Q|Q|^{p-2}\|_{\Ld^{p'}(D\setminus\Ic)}.
\end{gather*}
We need to modify $S^\circ$ to make it constant in $I_n$ while keeping the divergence-free constraint and without increasing the norm.
For all $n$ such that $I_n+\frac\delta2 B \subset D$, choose an extension $\tilde S^\circ_n\in W_0^{1,p'}(I_n+\frac\delta2 B)^d$ such that $\tilde S^\circ_n=-S^\circ+\fint_{I_n}S^\circ$ in $I_n$ and
\[\|\tilde S^\circ_n\|_{W^{1,p'}(I_n+\frac\delta2 B)}\,\lesssim\,\Big\|S^\circ-\fint_{I_n}S^\circ\Big\|_{\Ld^{p'}(I_n)}.\]
So defined, $S^\circ+\tilde S^\circ_n$ is constant on $I_n$ but not divergence-free.
By a standard use of the Bogovskii operator in form of~\cite[Theorem~III.3.1]{Galdi}, there exists a vector field $\tilde S^n\in W^{1,p'}_0((I_n+\frac\delta2 B)\setminus I_n)^d$ such that
\begin{gather*}
\Div(\tilde S^n)\,=\,-\Div(\tilde S_n^\circ),\quad \text{in }I_n+\tfrac\delta2 B,\\
\|\nabla\tilde S^n\|_{\Ld^{p'}((I_n+\frac\delta2 B)\setminus I_n)}\,\lesssim_p\,\|\nabla \tilde S_n^\circ\|_{\Ld^{p'}(I_n+\frac\delta2 B)}.
\end{gather*}
We then define $S^n:=\tilde S^\circ_n+\tilde S^n\in W^{1,p'}_0(I_n+\frac\delta2 B)^d$, which satisfies $S^n=\tilde S^\circ_n=-S^\circ+\fint_{I_n}S^\circ$ in $I_n$ and in addition, combining the above with Poincaré's inequality,
\begin{gather}
\Div( S^n) = 0,\nonumber\\
\|\nabla S^n\|_{\Ld^{p'}(I_n+\frac\delta2 B)}\,\lesssim_p\,\|\nabla S^\circ\|_{\Ld^{p'}(I_n)}.\label{eq:ext-Bogo}
\end{gather}
For all $n$ such that $(I_n+\frac\delta2 B) \cap \partial D \ne \varnothing$,
we proceed to a similar construction, replacing $I_n+\frac\delta2 B$ by $(I_n+\frac\delta2 B)\cap D$, and $\fint_{I_n} S^\circ$ by zero. Using Poincar\'e's inequality on \mbox{$(I_n+\delta B) \cap D$}, rather than Poincar\'e's inequality with vanishing average on $I_n+\frac\delta2 B$, this provides a vector field $S^n\in W^{1,p'}_0((I_n+\frac\delta2B)\cap D)^d$, which satisfies $S^n=-S^\circ$ in $I_n\cap D$ and
\begin{gather*}
\Div( S^n) = 0,\\
\|\nabla S^n\|_{\Ld^{p'}((I_n+\frac\delta2B)\cap D)}\,\lesssim_p\,\|\nabla S^\circ\|_{\Ld^{p'}((I_n+\delta B)\cap D)}.
\end{gather*}
Since the fattened inclusions $\{(I_n+\delta B)\cap D\}_n$ are all disjoint, cf.~\ref{Hd}, implicitly extending~$S^n$ by~$0$ outside its domain of definition, the vector field $S:=S^\circ+\sum_nS^n$ satisfies all the required properties.

\medskip
\step2 Conclusion.\\
Testing equation~\eqref{e:up-whole0} with $S$, using that $S$ is constant inside particles, and recalling the boundary conditions for $u_g$, cf.~\eqref{eq:test-v0}, we are led to
\begin{equation*}
\int_{D\setminus\Ic}\Div(S)\,P_g\,=\,\int_{D}\nabla S:\nabla u_g-\int_{D\setminus \Ic}\nabla S:g.
\end{equation*}
Inserting the definition of $\Div(S)$, recalling that $\nabla S$ vanishes in $\Ic$, and using Hölder's inequality, we find
\begin{equation*}
\|Q\|_{\Ld^p(D\setminus\Ic)}^p\,\lesssim_p\,\|\nabla S\|_{\Ld^{p'}(D)}\|(\nabla u_g,g)\|_{\Ld^{p}(D\setminus\Ic)},
\end{equation*}
and the claim follows from the bound on the norm of $\nabla S$ in Step~1.
\end{proof}

\subsubsection{Dual Calder\'on--Zygmund lemma}
As in~\cite{DO1}, we shall appeal to the following dual version of the Calder\'on--Zygmund lemma due to Shen~\cite[Theorem~3.2]{Shen-07}.
For a ball $D\subset\R^d$, we henceforth set $D=B_{r_D}(x_D)$ and use the abusive short-hand notation  $k D:=B_{k r_D}(x_D)$ for dilations centered at the same point.

\begin{lem}[\cite{Shen-07}]\label{lem:Shen}
Given $1\le p_0<p_1\le\infty$, $F,G\in\Ld^{p_0}\cap\Ld^{p_1}(\R^d)$, and $C_0>0$, assume that for all balls $D\subset \R^d$ there exist measurable functions $F_{D,0}$ and $F_{D,1}$ such that \mbox{$|F|\le |F_{D,0}|+|F_{D,1}|$} and $|F_{D,1}|\le |F|+|F_{D,0}|$ on $D$, and such that
\begin{eqnarray*}
\Big(\fint_{D}|F_{D,0}|^{p_0}\Big)^\frac1{p_0}&\le&C_0\Big(\fint_{C_0D}|G|^{p_0}\Big)^\frac1{p_0},\\
\Big(\fint_{\frac1{C_0}D}|F_{D,1}|^{p_1}\Big)^\frac1{p_1}&\le&C_0\Big(\fint_{D}|F_{D,1}|^{p_0}\Big)^\frac1{p_0}.
\end{eqnarray*}
Then,
for all $p_0<p<p_1$,
\[\Big(\int_{\R^d}|F|^p\Big)^\frac1p\,\lesssim_{C_0,p_0,p,p_1}\,\Big(\int_{\R^d}|G|^p\Big)^\frac1p.\qedhere\]
\end{lem}

\subsubsection{Gehring's lemma}
We shall appeal to the following version of Gehring's lemma, which is a mild reformulation of~\cite[Proposition~5.1]{GiaMo}.

\begin{lem}[\cite{Gehring,GiaMo}]\label{lem:Gehring}
Given $1<q<s$ and a reference cube $Q_0\subset\R^d$, let $G\in\Ld^q(Q_0)$ and $F\in\Ld^s(Q_0)$ be nonnegative functions.
There exist $\theta_0>0$ (only depending on $d,q,s$) with the following property:
Given $\theta\le\theta_0$, if for some $C_0\ge1$ the following condition holds for all cubes $Q\subset Q_0$,
\[\Big(\fint_{\frac1{C_0}Q}G^q\Big)^\frac1q\le C_0\fint_{Q}G+C_0\Big(\fint_{Q}F^q\Big)^\frac1q+\theta\Big(\fint_{Q}G^q\Big)^\frac1q,\]
then there exists $\eta_0>0$ (only depending on $C_0,d,q,s$) such that for all $q\le p\le q+\eta_0$,
\[\Big(\fint_{\frac1{C_0}Q_0}G^p\Big)^\frac1p\lesssim_{C_0,q,r}\fint_{Q_0}G+\Big(\fint_{Q_0}F^p\Big)^\frac1p.\qedhere\]
\end{lem}

\subsection{Proof of Theorem~\ref{th:meyers-ann}}
Starting point is the following deterministic perturbative result, for which an argument is postponed to Section~\ref{sec:proof-Meyers}.

\begin{prop}\label{prop:Meyers}
Given a deterministic inclusion set $\Ic$ satisfying the hardcore and regularity conditions in~\ref{Hd},
there exists a constant $C_0\simeq1$ such that the following hold.
\begin{enumerate}[(i)]
\item \emph{Meyers-type $\Ld^p$ estimate:}\\
Given $g\in C^\infty_c(\R^d)^{d\times d}$, the solution $(\nabla u_g,P_g)$ of the steady Stokes problem~\eqref{eq:test-v0}
satisfies for all $2\le p\le2+\frac1{C_0}$,
\[\qquad\|[\nabla u]_2\|_{\Ld^p(\R^d)}\,\lesssim\,\|[g]_2\|_{\Ld^p(\R^d)}.\]
\item \emph{Reverse Jensen's inequality:}\\
For any ball $D\subset\R^d$, if $(w,Q)$ satisfies the following equations in $D$,
\begin{equation*}
\qquad\left\{\begin{array}{ll}
-\triangle w+\nabla Q=0,&\text{in $D\setminus\Ic$},\\
\Div(w)=0,&\text{in $D\setminus\Ic$},\\
\D(w)=0,&\text{in $\Ic$},\\
\int_{\partial I_{n}}\sigma(w,Q)\nu=0,&\forall n:\,I_{n}\subset D,\\
\int_{\partial I_{n}}\Theta(x-x_n)\cdot\sigma(w,Q)\nu=0,&\forall n:\,I_{n}\subset D,\,\forall\Theta\in\Md^\Skew,
\end{array}\right.
\end{equation*}
then there holds for all $q\le p$ with $|p-2|,|q-2|\le\frac1{C_0}$,
\[\qquad\Big(\fint_{\frac1CD}[\nabla w]_2^p\Big)^\frac1p\,\lesssim\,\Big(\fint_{D}[\nabla w]_2^q\Big)^\frac1q.\qedhere\]
\end{enumerate}
\end{prop}

We may now proceed with the proof of Theorem~\ref{th:meyers-ann}, which follows from the above together with Shen's dual version of the Calder\'on--Zygmund lemma, cf.~Lemma~\ref{lem:Shen}.

\begin{proof}[Proof of Theorem~\ref{th:meyers-ann}]
We split the proof into three steps.
We start with estimates outside the particles: first for $2\le q<p$, and then for $p<q\le2$ by a duality argument, so that the full range of exponents is finally reached by interpolation. Next, we extend the estimates inside the particles. 
Let $C_0\ge1$ be fixed as in the statement of Proposition~\ref{prop:Meyers}.

\medskip
\step1 Proof that for all $2\le q<p< 2+\frac1{C_0}$,
\begin{equation}\label{eq:CZ-0}
\|[\nabla u_g]_2\|_{\Ld^p(\R^d;\Ld^q(\Omega))}\,\lesssim\,\|[g]_2\|_{\Ld^p(\R^d;\Ld^q(\Omega))}.
\end{equation}
Let $2\le p_0\le p_1\le2+\frac1{C_0}$ be fixed. For balls $D\subset\R^d$, we decompose
\[\nabla u_g=\nabla u_{D,0}+\nabla u_{D,1},\]
where $\nabla u_{D,0}\in \Ld^\infty(\Omega;\Ld^2(\R^d)^{d\times d})$ denotes the unique solution of
\begin{equation*}
\left\{\begin{array}{ll}
-\triangle u_{D,0}+\nabla P_{D,0}=\Div (g\mathds1_{D}),&\text{in $\R^d\setminus\Ic$},\\
\Div (u_{D,0})=0,&\text{in $\R^d\setminus\Ic$},\\
\D(u_{D,0})=0,&\text{in $\Ic$},\\
\int_{\partial I_{n}}\big(g\mathds1_D+\sigma(u_{D,0},P_{D,0})\big)\nu=0,&\forall n,\\
\int_{\partial I_{n}}\Theta(x-x_n)\cdot\big(g\mathds1_D+\sigma(u_{D,0},P_{D,0})\big)\nu=0,&\forall n,\,\forall\Theta\in\Md^\Skew.
\end{array}\right.
\end{equation*}
On the one hand, for balls $D$ with radius $r_D>1$, Proposition~\ref{prop:Meyers}(i) applied to the above equation yields
\[\int_{D}\expec{[\nabla u_{D,0}]_2^{p_0}}\,\le\,\E\bigg[{\int_{\R^d}[\nabla u_{D,0}]_2^{p_0}}\bigg]\,\lesssim\,\E\bigg[{\int_{\R^d}[g\mathds1_D]_2^{p_0}}\bigg]\,\le\,\int_{2D}\expec{[g]_2^{p_0}},\]
while for balls $D$ with radius $r_D<1$ we appeal to the plain energy inequality~\eqref{eq:energy} in form of
\[\int_{D}\expec{[\nabla u_{D,0}]_2^{p_0}}\,\le\,|D|\,\E\bigg[\Big(\int_{\R^d}|\nabla u_{D,0}|^2\Big)^\frac{p_0}2\bigg]\,\lesssim\,|D|\,\E\bigg[\Big(\int_D|g|^2\Big)^\frac{p_0}2\bigg]\,\lesssim\,\int_{2D}\expec{[g]_2^{p_0}}.\]
On the other hand, noting that $\nabla u_{D,1}=\nabla u_g-\nabla u_{D,0}$ satisfies \begin{equation*}
\left\{\begin{array}{ll}
-\triangle u_{D,1}+\nabla P_{D,1}=0,&\text{in $D\setminus\Ic$},\\
\Div (u_{D,1})=0,&\text{in $D\setminus\Ic$},\\
\D(u_{D,1})=0,&\text{in $\Ic$},\\
\int_{\partial I_{n}}\sigma(u_{D,1},P_{D,1})\nu=0,&\forall n:\,I_{n}\subset D,\\
\int_{\partial I_{n}}\Theta(x-x_n)\cdot\sigma(u_{D,1},P_{D,1})\nu=0,&\forall n:\,I_{n}\subset D,\,\forall\Theta\in\Md^\Skew,
\end{array}\right.
\end{equation*}
it follows from the triangle inequality in $\Ld^{\frac{p_1}{p_0}}$ and from Proposition~\ref{prop:Meyers}(ii) that
\begin{eqnarray*}
\Big(\fint_{\frac1CD}\expec{[\nabla u_{D,1}]_2^{p_0}}^{\frac{p_1}{p_0}}\Big)^\frac1{p_1}&\le&\E\bigg[{\Big(\fint_{\frac1CD}[\nabla u_{D,1}]_2^{p_1}\Big)^\frac{p_0}{p_1}}\bigg]^\frac1{p_0}\\
&\lesssim&\E\bigg[{\fint_{D}[\nabla u_{D,1}]_2^{p_0}}\bigg]^\frac1{p_0}=\Big(\fint_{D}\expec{[\nabla u_{D,1}]_2^{p_0}}\Big)^\frac1{p_0}.
\end{eqnarray*}
In view of these estimates, appealing to Lemma~\ref{lem:Shen} with
\begin{gather*}
F:=\expecm{[\nabla u_g]_2^{p_0}}^{\frac1{p_0}},\quad G:=\expec{[g]_2^{p_0}}^\frac1{p_0},\\
F_{D,0}:=\expecm{[\nabla u_{D,0}]_2^{p_0}}^{\frac1{p_0}},\quad F_{D,1}:=\expecm{[\nabla u_{D,1}]_2^{p_0}}^{\frac1{p_0}},
\end{gather*}
we deduce for all $p_0<p<p_1$,
\[\Big(\int_{\R^d}\expecm{[\nabla u_g]_2^{p_0}}^{\frac{p}{p_0}}\Big)^\frac1q\,\lesssim\,\Big(\int_{\R^d}\expec{[g]_2^{p_0}}^\frac p{p_0}\Big)^\frac1q,\]
and the claim~\eqref{eq:CZ-0} follows.

\medskip
\step2 Duality and interpolation: proof that for all $2-\frac1{2C_0}<p<q\le2$,
\begin{equation}\label{eq:dual-CZ-0}
\|[\mathds1_{\R^d\setminus\Ic}\nabla u]_2\|_{\Ld^{p}(\R^d;\Ld^{q}(\Omega))}
\,\lesssim\,\|[g]_2\|_{\Ld^{p}(\R^d;\Ld^{q}(\Omega))}.
\end{equation}
Combining this with~\eqref{eq:CZ-0}, we then deduce by interpolation that the same estimate holds for all $p,q$ with $|p-2|,|q-2|<\frac1{8C_0}$.

\medskip\noindent
Given a test function $h\in C^\infty_c(\R^d;\Ld^\infty(\Omega)^{d\times d})$, we consider the solution $(\nabla u_h,P_h)$ of the steady Stokes problem~\eqref{eq:test-v0} with $g$ replaced by $h$.
In view of~\eqref{e:up-whole0}, there holds in the weak sense in $\R^d$,
\begin{eqnarray*}
-\triangle u_g+\nabla(P_g\mathds1_{\R^d\setminus\Ic})&=&\nabla\cdot (g\mathds1_{\R^d\setminus\Ic})-\sum_n\delta_{\partial I_{n}}\big(g+\sigma(u_g,P_g)\big)\nu,\\
-\triangle u_h+\nabla(P_h\mathds1_{\R^d\setminus\Ic})&=&\nabla\cdot (h\mathds1_{\R^d\setminus\Ic})-\sum_n\delta_{\partial I_{n}}\big(h+\sigma(u_h,P_h)\big)\nu.
\end{eqnarray*}
Testing the equation for $u_h$ with $u_g$, and vice versa, and noting that the boundary terms all vanish in view of the respective boundary conditions, we find
\begin{equation*}
\int_{\R^d\setminus\Ic}h:\nabla u_g=-\int_{\R^d}\nabla u_h:\nabla u_g=\int_{\R^d\setminus\Ic}g:\nabla u_h.
\end{equation*}
Combined with a duality argument, this identity yields
\begin{eqnarray*}
\lefteqn{\|[\mathds1_{\R^d\setminus\Ic}\nabla u_g]_2\|_{\Ld^{p}(\R^d;\Ld^{q}(\Omega))}}\\
&\lesssim&\sup\bigg\{\E\bigg[{\int_{\R^d\setminus\Ic}h:\nabla u_g}\bigg]~:~\|[h]_2\|_{\Ld^{p'}(\R^d;\Ld^{q'}(\Omega))}=1\bigg\}\\
&=&\sup\bigg\{\E\bigg[{\int_{\R^d\setminus\Ic}g:\nabla u_h}\bigg]~:~\|[h]_2\|_{\Ld^{p'}(\R^d;\Ld^{q'}(\Omega))}=1\bigg\}\\
&\le&\|[g]_2\|_{\Ld^{q}(\R^d;\Ld^{p}(\Omega))}\sup\bigg\{\|[\nabla u_h]_2\|_{\Ld^{p'}(\R^d;\Ld^{q'}(\Omega))}~:~\|[h]_2\|_{\Ld^{p'}(\R^d;\Ld^{q'}(\Omega))}=1\bigg\}.
\end{eqnarray*}
Given $2-\frac1{2C_0}<p<q\le2$, we may appeal to~\eqref{eq:CZ-0} with $2\le q'<p'<2+\frac1{C_0}$, and the claim~\eqref{eq:dual-CZ-0} follows.

\medskip
\step3 Conclusion.\\
In view of Step~2, it remains to show that for all $p,q\ge1$,
\begin{equation}\label{eq:CZ-inballs}
\|[\mathds1_{\Ic}\nabla u_g]_2\|_{\Ld^p(\R^d;\Ld^q(\Omega))}\,\lesssim\,\|[\mathds1_{\R^d\setminus\Ic}\nabla u_g]_2\|_{\Ld^p(\R^d;\Ld^q(\Omega))}.
\end{equation}
For all $n$, since $u$ is affine in $I_{n}$, we can write for any constant $c_{n}\in\R^d$,
\[\|\nabla u_g\|_{\Ld^\infty(I_{n})}\,\lesssim\,\|u_g-c_{n}\|_{\Ld^1(\partial I_{n})}.\]
By a trace estimate and by Poincar\'e's inequality with the choice $c_{n}:=\fint_{(I_{n}+\delta B)\setminus I_{n}}u_g$, we deduce
\[\|\nabla u_g\|_{\Ld^\infty(I_{n})}\,\lesssim\,\|u_g-c_{n}\|_{W^{1,1}((I_{n}+\delta B)\setminus I_{n})}\,\lesssim\,\|\nabla u_g\|_{\Ld^1((I_{n}+\delta B)\setminus I_{n})}.\]
We may then estimate pointwise,
\[\mathds1_{\Ic}|\nabla u_g|\,\lesssim\,\sum_n\mathds1_{I_{n}}\|\nabla u_g\|_{\Ld^1((I_{n}+\delta B)\setminus I_{n})},\]
and the claim~\eqref{eq:CZ-inballs} now follows from the hardcore condition in~\ref{Hd}.
\end{proof}

\subsection{Proof of Proposition~\ref{prop:Meyers}}\label{sec:proof-Meyers}
We split the proof into two steps. We start with a Meyers-type perturbative argument based on Caccioppoli's inequality and Gehring's lemma, and we conclude in the second step.

\medskip
\step1 Meyers-type perturbative argument: there exists $C_0\ge1$ (only depending on $d,\delta$) such that for all balls $D\subset\R^d$ and $2\le p\le 2+\frac1{C_0}$,
\begin{equation}\label{eq:impr-integr}
\Big(\fint_{D}[\nabla u_g]_2^p\Big)^\frac1p
\,\lesssim\, \Big(\fint_{CD}[\nabla u_g]_2^\frac{2d}{d+2}\Big)^\frac{d+2}{2d}+\Big(\fint_{CD}[g]_2^p\Big)^\frac1p.
\end{equation}
Given a ball $D\subset\R^d$ with radius $r_D\ge3$, choose a cut-off function $\chi_D$ with $\chi_D|_D\equiv 1$, $\chi_D|_{\R^d\setminus2D}\equiv 0$, and $|\nabla\chi_D|\lesssim\frac1{r_D}$, such that $\chi_D$ is constant in $I_{n}$ for all $n$.
Given arbitrary constants $c_{D}\in\R^d$ and $c_{D}'\in\R$, testing the equation~\eqref{e:up-whole0} for $u_g$ with $\chi_D^2(u_g-c_{D})$,
noting that the boundary terms all vanish, and recalling that $\Div (u_g)=0$, we obtain the following Caccioppoli-type inequality,
\begin{multline*}
\int_{D}|\nabla u_g|^2
\,\lesssim\,\frac{1}{r_D^2}\int_{2D}|u_g-c_{D}|^2+\int_{2D}|g|^2\\
+\Big(\frac{1}{r_D^2}\int_{2D}|u_g-c_{D}|^2\Big)^\frac12\Big(\int_{2D}|P_g-c_{D}'|^2\mathds1_{\R^d\setminus\Ic}\Big)^\frac12.
\end{multline*}
Hence, for all $K\ge1$,
\begin{equation*}
\int_{D}|\nabla u_g|^2
\,\lesssim\,\frac{K^2}{r_D^2}\int_{2D}|u_g-c_{D}|^2+\int_{2D}|g|^2+\frac1{K^2}\int_{2D}|P_g-c_{D}'|^2\mathds1_{\R^d\setminus\Ic}.
\end{equation*}
Using the the Poincar\'e--Sobolev inequality to estimate the first right-hand side term, with the choice $c_{D}:=\fint_{2D}u_g$, and using the localized pressure estimate of Lemma~\ref{lem:pres} to estimate the last right-hand side term, with the choice $c_{D}':=\fint_{2D\setminus\Ic}P$, we deduce
\begin{equation}\label{eq:Cacc-Sob}
\Big(\fint_{D}|\nabla u_g|^2\Big)^\frac12
\,\lesssim\, K\Big(\fint_{2D}|\nabla u_g|^\frac{2d}{d+2}\Big)^\frac{d+2}{2d}+\Big(\fint_{2D}|g|^2\Big)^\frac12+\frac1K\Big(\fint_{2D}|\nabla u_g|^2\Big)^\frac12.
\end{equation}
While this is proven here for all balls $D$ with radius $r_D\ge3$, taking local quadratic averages allows us to infer for all balls $D$ (with any radius $r_D>0$) and all $K\ge1$,
\begin{equation*}
\Big(\fint_{D}[\nabla u_g]_2^2\Big)^\frac12
\,\lesssim\, K\Big(\fint_{2D}[\nabla u_g]_2^\frac{2d}{d+2}\Big)^\frac{d+2}{2d}+\Big(\fint_{2D}[g]_2^2\Big)^\frac12+\frac1K\Big(\fint_{2D}[\nabla u_g]_2^2\Big)^\frac12.
\end{equation*}
Choosing $K$ large enough, the claim~\eqref{eq:impr-integr} now follows from Gehring's lemma in form of Lemma~\ref{lem:Gehring}.

\medskip
\step2 Conclusion.\\
We start with the proof of~(i).
Applying~\eqref{eq:impr-integr} together with Jensen's inequality and with the energy inequality~\eqref{eq:energy}, we find for all $2\le p\le2+\frac1{C_0}$,
\begin{eqnarray*}
\Big(\int_{D}[\nabla u_g]_2^p\Big)^\frac1p
&\lesssim& |D|^{\frac1p-\frac12}\Big(\int_{CD}[\nabla u_g]_2^2\Big)^\frac{1}{2}+\Big(\int_{CD}[g]_2^p\Big)^\frac1p\\
&\lesssim& |D|^{\frac1p-\frac12}\Big(\int_{\R^d}|g|^2\Big)^\frac{1}{2}+\Big(\int_{CD}[g]_2^p\Big)^\frac1p,
\end{eqnarray*}
hence the conclusion~(i) follows for $D\uparrow\R^d$.
Next, item~(ii) is a consequence of~\eqref{eq:impr-integr} with~$g=0$ in~$CD$.
\qed

\section{Corrector estimates}
This section is devoted to the proof of Theorem~\ref{th:cor}.
Next to the corrector $\psi_E$, we further introduce an associated flux corrector~$\zeta_E$,
which is key to put the equation for two-scale expansion errors into a more favorable form, cf.~\eqref{eq:error-wPeps+}.
As in~\cite[Theorem~4]{D-20}, motivated by the work of Jikov on homogenization problems with stiff inclusions~\cite{Jikov-87,Jikov-90} (see also~\cite[Section~3.2]{JKO94}), we start by defining a divergence-free extension $J_E$ of the flux $\sigma(\psi_E+Ex,\Sigma_E)\mathds1_{\R^d\setminus\Ic}$.
Although this extension is not unique, we can choose it as in~\cite{D-20} to coincide with the flux in the corresponding incompressible linear elasticity problem in the limit of inclusions with diverging shear modulus.
The flux corrector $\zeta_E$ is then defined as a vector potential for this extended flux $J_E$; more precisely, equation~\eqref{eq:def-zeta} below amounts to choosing the Coulomb gauge. The construction is recalled for convenience in Section~\ref{sec:pr-ext-cor}.

\begin{lem}[Extended fluxes and flux correctors; \cite{D-20}]\label{lem:ext-cor}
Under Assumption~\ref{Hd}, for all $E\in\Md_0$, there is a stationary random $2$-tensor field $J_E:=\{J_{E;ij}\}_{1\le i,j\le d}$ with finite second moment such that almost surely,
\begin{equation}\label{eq:def-flux}
J_E\mathds1_{\R^d\setminus\Ic}\,=\,\sigma(\psi_E+Ex,\Sigma_E)\mathds1_{\R^d\setminus\Ic},\qquad \Div (J_E)=0.
\end{equation}
In these terms, there exists a unique random $3$-tensor field $\zeta_E=\{\zeta_{E;ijk}\}_{1\le i,j,k\le d}$ that satisfies the following infinite-volume problem:
\begin{enumerate}[\quad$\bullet$]
\item For all $i,j,k$, almost surely, $\zeta_{E;ijk}$ belongs to $H^1_\loc(\R^d)$ and satisfies in the weak sense,
\begin{equation}\label{eq:def-zeta}
-\triangle \zeta_{E,ijk}\,=\,\partial_jJ_{E,ik}-\partial_kJ_{E,ij}.
\end{equation}
\item The random field $\nabla\zeta_E$ is stationary, has vanishing expectation, has finite second moment, and $\zeta_E$ satisfies the anchoring condition $\fint_B\zeta_E=0$ almost surely.
\end{enumerate}
In addition, the following properties are automatically satisfied:
\begin{enumerate}[(i)]
\item $\zeta_E$ is skew-symmetric in its last two indices, that is, $\zeta_{E,ijk}\,=\,-\zeta_{E,ikj}$ for all $i,j,k$;
\smallskip\item $\zeta_E$ is a vector potential for $J_E$, that is,
\[\quad\Div (\zeta_{E,i})=J_{E,i}-\expec{J_{E,i}},\]
in terms of $\zeta_{E,i}=\{\zeta_{E,ijk}\}_{1\le j,k\le d}$ and $J_{E,i}=\{J_{E,ij}\}_{1\le j\le d}$;
\smallskip\item $\zeta_E$ is sublinear at infinity, that is, $\e\zeta_E(\frac\cdot\e)\cvf{}0$ in $H^{1}_\loc(\R^d)$ almost surely as $\e\downarrow0$;
\smallskip\item $\expec{J_E}=2\Bb E+(\bb:E)\Id$, where we recall that the effective constants $\Bb$ and $\bb$ are defined in~\eqref{eq:def-B} and~\eqref{eq:def-b}.
\qedhere
\end{enumerate}
\end{lem}

With the above definition, we shall establish the following version of Theorem~\ref{th:cor} for the extended corrector $(\psi_E,\zeta_E)$; the proof is postponed to Section~\ref{sec:fluc-str}.

\begin{theor}[Extended corrector estimate]\label{th:cor+}
Under Assumptions~\ref{Hd} and~\ref{Mix+}, for all $E\in\Md_0$ and $q<\infty$,
\begin{equation}\label{eq:bnd-grad-cor1}
\|[(\nabla\psi_E,\Sigma_E\mathds1_{\R^d\setminus\Ic},\nabla\zeta_E)]_2\|_{\Ld^q(\Omega)}\,\lesssim_{q}\,|E|,
\end{equation}
and
\begin{equation}\label{eq:bnd-cor1}
\|[(\psi_E,\zeta_E)]_2(x)\|_{\Ld^q(\Omega)}\,\lesssim_q\,|E|\,\mu_d(|x|),
\end{equation}
where we recall that $\mu_d$ is defined in~\eqref{eq:bnd-cor}.
\end{theor}

\subsection{Proof of Lemma~\ref{lem:ext-cor}}\label{sec:pr-ext-cor}
Let $E\in\Md_0$.
We split the proof into two main steps.

\medskip
\step1 Construction of the extended flux $J_E$.\\
Given a realization of the set of inclusions, we consider for all $n$ the weak solution $(\psi_E^n,\Sigma_E^n)$ in $H^1(I_n)^d\times\Ld^2(I_n)$ of the following Neumann problem in~$I_n$,
\begin{equation}\label{eq:sol-psiEn}
\left\{\begin{array}{ll}
-\triangle\psi_E^{n}+\nabla\Sigma_E^{n}=0,&\text{in $I_n$},\\
\Div(\psi_E^{n})=0,&\text{in $I_n$},\\
\sigma(\psi_E^{n},\Sigma_E^{n})\nu=\sigma(\psi_E+Ex,\Sigma_E)\nu,&\text{on $\partial I_n$}.
\end{array}\right.
\end{equation}
Note that $\psi_E^{n}$ is defined only up to a rigid motion, which is fixed by choosing $\fint_{I_n}\psi_E^{n}=0$ and $\fint_{I_n}\nabla\psi_E^{n}\in\Md^\Sym_0$, and we prove that $(\psi_E^n,\Sigma_E^n)$ satisfies
\begin{equation}\label{eq:bnd-psinSign}
\|(\nabla\psi_E^n,\Sigma_E^n)\|_{\Ld^2(I_n)} \, \lesssim \,\|\sigma(\psi_E+Ex,\Sigma_E)\|_{\Ld^2((I_n+\delta B)\setminus I_n)}.
\end{equation}

\medskip
\substep{1.1} Well-posedness of the Neumann problem~\eqref{eq:sol-psiEn} for $\psi_E^n$.\\
The weak formulation of~\eqref{eq:sol-psiEn} takes on the following guise: $\psi_E^n$ is divergence-free and satisfies for all divergence-free test functions $\phi\in H^1(I_n)^d$,
\begin{equation}\label{eq:funct-psin}
2\int_{I_n}\D(\phi):\D(\psi_E^{n})=\Lc_E(\phi),
\end{equation}
in terms of the linear functional
\[\Lc_E(\phi)\,:=\,\int_{\partial I_n}\phi\cdot\sigma(\psi_E+Ex,\Sigma_E)\nu.\]
In view of the boundary conditions for $\psi_E$, we can rewrite for any $V\in\R^d$ and $\Theta\in\Md^\Skew$,
\[\Lc_E(\phi)\,=\,\int_{\partial I_n}(\phi-V-\Theta(\cdot-x_n))\cdot\sigma(\psi_E+Ex,\Sigma_E)\nu.\]
Choose an extension map
\[T_n:\big\{\phi\in H^1(I_n)^d:\Div(\phi)=0\big\}\to\big\{\phi\in H^1_0(I_n+\delta B)^d:\Div(\phi)=0\big\},\]
such that $T_n[\phi]|_{I_n}=\phi|_{I_n}$ and
\[\|T_n[\phi]\|_{H^1(I_n+\delta B)}\,\lesssim\,\|\phi\|_{H^1(I_n)}.\]
In these terms, using Stokes' formula and recalling that $\sigma(\psi_E+Ex,\Sigma_E)$ is symmetric and divergence-free, we can further rewrite
\begin{eqnarray*}
\Lc_E(\phi)
&=&-\int_{(I_n+\delta B)\setminus I_n}\Div\Big(\sigma(\psi_E+Ex,\Sigma_E)\,T_n[\phi-V-\Theta(\cdot-x_n)]\Big)\nonumber\\
&=&
-\int_{(I_n+\delta B)\setminus I_n}\D(T_n[\phi-V-\Theta(\cdot-x_n)]):\sigma(\psi_E+Ex,\Sigma_E),
\end{eqnarray*}
and thus, since $\D(T_n[\phi-V-\Theta(\cdot-x_n)])$ is trace-free,
\begin{equation}\label{eq:extension-eqn}
\Lc_E(\phi)
\,=\,-2\int_{(I_n+\delta B)\setminus I_n}\D(T_n[\phi-V-\Theta(\cdot-x_n)]):(\D(\psi_E)+E).
\end{equation}
We deduce that $\phi\mapsto\Lc_E(\phi)$ is a continuous linear functional on $\{\phi\in H^1(I_n)^d:\Div(\phi)=0\}$. In addition, for all divergence-free $\phi\in H^1(I_n)^d$,
minimizing over $V,\Theta$ and appealing to Korn's inequality,
we find
\begin{eqnarray}\label{eq:bnd-LE}
|\Lc_E(\phi)|
&\lesssim&\inf_{V\in\R^d,\,\Theta\in\Md^\Skew}\|\phi-V-\Theta(\cdot-x_n)\|_{H^1(I_n)}\|\!\D(\psi_E)+E\|_{\Ld^2((I_n+\delta B)\setminus I_n)}\nonumber\\
&\lesssim&\|\!\D(\phi)\|_{\Ld^2(I_n)}\|\!\D(\psi_E)+E\|_{\Ld^2((I_n+\delta B)\setminus I_n)}.
\end{eqnarray}
By the Lax-Milgram theorem, we deduce that there exists a unique trace-free gradient-like solution $\D(\psi_E^{n})\in \Ld^2(I_n)^{d\times d}_\Sym$ of~\eqref{eq:funct-psin},
and it satisfies
\begin{equation*}
\|\!\D(\psi_E^{n})\|_{\Ld^2(I_n)}\,\lesssim\,\|\!\D(\psi_E)+E\|_{\Ld^2((I_n+\delta B)\setminus I_n)}.
\end{equation*}
The vector field $\psi_E^{n}$ is itself defined only up to a rigid motion and is fixed by choosing $\fint_{I_n}\psi_E^{n}=0$ and $\fint_{I_n}\nabla\psi_E^{n}\in\Md_0^\Sym$, in which case the above becomes by Korn's inequality,
\begin{equation}\label{eq:nabpsiEn}
\|\nabla\psi_E^{n}\|_{\Ld^2(I_n)}\,\lesssim\,\|\!\D(\psi_E)+E\|_{\Ld^2((I_n+\delta B)\setminus I_n)}.
\end{equation}

\medskip
\substep{1.2} Construction of the pressure.\\
Consider the extended deformation
\[q_E^n\,:=\,\D(\psi_E)+E+D(\psi_E^n)\mathds1_{I_n},\qquad\text{in $I_n+\delta B$}.\]
In view of~\eqref{eq:extension-eqn}, the weak formulation~\eqref{eq:funct-psin} yields for all divergence-free test functions $\phi\in C^\infty_c(I_n+\delta B)^d$,
\[2\int_{\R^d}\D(\phi):q_E^n\,=\,0.\]
Appealing e.g.~to~\cite[Proposition~12.10]{JKO94}, we deduce that there exists an associated pressure field $\Sigma_E^n\in\Ld^2_\loc(I_n+\delta B)$, which is unique up to an additive constant, such that for all test functions~$\phi\in C^\infty_c(I_n+\delta B)^d$,
\begin{equation}\label{eq:qtildeS}
\int_{\R^d}\D(\phi):(2q_E^n-\Sigma_E^n)\,=\,0.
\end{equation}
Since for all $\phi\in C^\infty_c((I_n+\delta B)\setminus I_n)^d$ we have
\[\int_{\R^d}\D(\phi):(2q_E^n -\Sigma_E^n)\,=\,\int_{\R^d}\D(\phi):\big(2(\D(\psi_E)+E)-\Sigma_E^n\big)\,=\,0,\]
we deduce that $\Sigma_E^n$ can be chosen uniquely to coincide with $\Sigma_E$ on $(I_n+\delta B)\setminus I_n$.
The pair $(\psi_E^n,\Sigma_E^n)$ is then the unique weak solution of the Neumann problem~\eqref{eq:sol-psiEn} with $\fint_{I_n}\psi_E^{n}=0$ and $\fint_{I_n}\nabla\psi_E^{n}\in\Md^\Sym_0$.

\medskip\noindent
It remains to prove~\eqref{eq:bnd-psinSign}. The estimation of $\nabla\psi_E^n$ follows from~\eqref{eq:nabpsiEn} and it remains to estimate the pressure $\Sigma_E^n$. For that purpose, using that $\Sigma_E^n$ coincides with $\Sigma_E$ on $(I_n+\delta B)\setminus I_n$, we split
\begin{eqnarray*}
\|\Sigma_E^n\|_{\Ld^2(I_n)}
&\le&\Big\|\Sigma_E^n-\fint_{I_n+\delta B}\Sigma_E^n\Big\|_{\Ld^2(I_n+\delta B)}+\Big|\fint_{I_n+\delta B}\Sigma_E^n\Big|\\
&\le&\Big\|\Sigma_E^n-\fint_{I_n+\delta B}\Sigma_E^n\Big\|_{\Ld^2(I_n+\delta B)}+\Big|\fint_{(I_n+\delta B)\setminus I_n}\Sigma_E\Big|\\
&&\hspace{4cm}+\Big|\fint_{(I_n+\delta B)\setminus I_n}\Big(\Sigma_E^n-\fint_{I_n+\delta B}\Sigma_E^n\Big)\Big|\\
&\lesssim&\Big\|\Sigma_E^n-\fint_{I_n+\delta B}\Sigma_E^n\Big\|_{\Ld^2(I_n+\delta B)}+\Big|\fint_{(I_n+\delta B)\setminus I_n}\Sigma_E\Big|
\end{eqnarray*}
Starting from~\eqref{eq:qtildeS}, a standard argument based on the Bogovskii operator yields
\begin{equation*}
\Big\|\Sigma_E^n-\fint_{I_n+\delta B}\Sigma_E^n\Big\|_{\Ld^2(I_n+\delta B)}\,\lesssim\,\|q_E^n\|_{\Ld^2(I_n+\delta B)},
\end{equation*}
so that the above becomes
\[\|\Sigma_E^n\|_{\Ld^2(I_n)}\,\lesssim\,\|q_E^n\|_{\Ld^2(I_n+\delta B)}+\|\Sigma_E\|_{\Ld^2((I_n+\delta B)\setminus I_n)},\]
and the claim~\eqref{eq:bnd-psinSign} follows from~\eqref{eq:nabpsiEn}.

\medskip\noindent
\substep{1.3} Construction of the extended flux.\\
We define the extended deformation and the extended pressure,
\[\tilde q_E\,:=\,\D(\psi_E)+E+\sum_n\D(\psi_E^n)\mathds1_{I_n},\qquad \tilde\Sigma_E\,:=\,\Sigma_E\mathds1_{\R^d\setminus\Ic}+\sum_n \Sigma_E^n\mathds1_{I_n},\]
as well as the corresponding extended flux
\begin{equation}\label{eq:def-J}
J_E\,:=\,2\tilde q_E-\tilde\Sigma_E\,=\,\sigma(\psi_E+Ex,\Sigma_E)\mathds1_{\R^d\setminus\Ic}+\sum_n\sigma(\psi_E^{n},\Sigma_E^{n})\mathds1_{I_n}.
\end{equation}
In view of~\eqref{eq:qtildeS}, together with~\eqref{eq:psiE-ref0}, the pair $(\tilde q_E,\tilde \Sigma_E)$ satisfies for all test functions $\phi\in C^\infty_c(\R^d)^d$,
\begin{equation}\label{eq:qtildeS2}
\int_{\R^d}\D(\phi):(2\tilde q_E-\tilde\Sigma_E)\,=\,0,
\end{equation}
that is, $J_E$ is divergence-free.
The uniqueness of the extensions ensures that $\tilde q_E$ and $\tilde\Sigma_E$ are both stationary, and we now prove that they have finite second moments.
Combining the definition of $\tilde q_E$ with the estimate~\eqref{eq:nabpsiEn} on $\psi_E^n$, we find for all $R>0$,
\[\|\tilde q_E\|_{\Ld^2(B_R)}\,\lesssim\,\|\!\D(\psi_E)+E\|_{\Ld^2(B_{R+3})},\]
and thus, by stationarity, letting $R\uparrow\infty$, and using the $\Ld^2$ estimate on $\psi_E$, cf.~Lemma~\ref{lem:cor},
\[\|\tilde q_E\|_{\Ld^2(\Omega)}\,\lesssim\,\|\!\D(\psi_E)+E\|_{\Ld^2(\Omega)}\,\lesssim\,|E|.\]
For the pressure, starting from~\eqref{eq:qtildeS2}, a standard argument based on the Bogovskii operator yields for all~$R>0$,
\begin{equation*}
\Big\|\tilde\Sigma_E-\fint_{B_R}\tilde\Sigma_E\Big\|_{\Ld^2(B_R)}\,\lesssim\,\|\tilde q_E\|_{\Ld^2(B_R)},
\end{equation*}
and thus, by stationarity, letting $R\uparrow\infty$ and using the above $\Ld^2$ estimate on $\tilde q_E$,
\begin{equation*}
\|\tilde\Sigma_E-\expecmm{\tilde\Sigma_E}\|_{\Ld^2(\Omega)}\,\lesssim\,\|\tilde q_E\|_{\Ld^2(\Omega)}\,\lesssim\,|E|.
\end{equation*}
We conclude that $\|J_E\|_{\Ld^2(\Omega)}\lesssim|E|$.
The identity in item~(iv) for the expectation $\expec{J_E}$ follows from a direct computation, cf.~\cite[Lemma~4.2]{D-20}, and is not repeated here.

\medskip
\step2 Construction of the flux corrector $\zeta_E$.\\
In view of standard stationary calculus, e.g.~\cite[Section~7]{JKO94} (see also~\cite[Proof of Lemma~1]{GNO-reg}),
equation~\eqref{eq:def-zeta} admits a unique stationary gradient solution $\nabla\zeta_{E}\in\Ld^2_\loc(\R^d;\Ld^2(\Omega)^{d\times d})$ with vanishing expectation and with
\[\|\nabla\zeta_E\|_{\Ld^2(\Omega)}\,\lesssim\,\|J_E\|_{\Ld^2(\Omega)}\,\lesssim\,|E|.\]
Items~(i) and~(ii) are easy consequences of the definition of $\zeta_E$. As in Lemma~\ref{lem:cor}, the additional sublinearity statement~(iii) is a standard result for random fields having a stationary gradient with vanishing expectation, cf.~e.g.~\cite[Section~7]{JKO94}.
\qed

\subsection{Proof of Theorem~\ref{th:cor+}}\label{sec:fluc-str}
We start with the following estimate on the optimal CLT decay for large-scale averages of the extended corrector gradient $(\nabla\psi_E,\nabla\zeta_E)$ and of the pressure~$\Sigma_E$.
Due to the nonlinearity of the corrector equation with respect to randomness, local norms of $(\nabla\psi_E,\Sigma_E)$ also appear in the right-hand side of \eqref{eq:CLTscaling}, which is a common difficulty in stochastic homogenization; this will be subsequently absorbed by a buckling argument, taking advantage of the CLT scaling.
\begin{prop}[CLT scaling]\label{prop:CLTscaling}
Under Assumptions~\ref{Hd} and~\ref{Mix+}, for all $g\in C^\infty_c(\R^d)$, $E\in\Md_0$, $R,s\ge1$, and~$1\ll q<\infty$, we have
\begin{align}\label{eq:CLTscaling}
&\Big\|\int_{\R^d}g\,(\nabla \psi_E,\Sigma_E\mathds1_{\R^d\setminus\Ic},\nabla\zeta_E)\Big\|_{\Ld^{2q}(\Omega)}\\
&\hspace{3cm}\,\lesssim_q\,\|g\|_{\Ld^2(\R^d)}\bigg(|E|+\Big\|\Big(\int_{B_{R}}[(\nabla\psi_E,\Sigma_E\mathds1_{\R^d\setminus\Ic})]_2^{2s}\Big)^\frac1{2s}\Big\|_{\Ld^{2q}(\Omega)}\bigg).\nonumber\qedhere
\end{align}
\end{prop}

In order to get such a control on stochastic moments, we appeal to the following consequence of the multiscale variance inequality~\eqref{eq:SGL} in~\ref{Mix+}, cf.~\cite[Proposition~1.10(ii)]{DG1}.

\begin{lem}[Control of higher moments; \cite{DG1}]\label{lem:p}
If the inclusion process $\Ic$ satisfies the multiscale variance inequality~\eqref{eq:SGL} with some weight $\pi$, then we have for all $q<\infty$ and all $\sigma(\Ic)$-measurable random variables~$Y(\Ic)$
with $\expec{Y(\Ic)}=0$,
\begin{equation}\label{eq:SGL-p}
\|Y(\Ic)\|_{\Ld^{2q}(\Omega)}^2\,\lesssim \, q^2\,\E\bigg[{\int_0^\infty \bigg(\int_{\R^d}\Big(\partial^{\operatorname{osc}}_{\Ic,B_\ell(x)}Y(\Ic)\Big)^2dx\bigg)^q \langle\ell\rangle^{-dq}\,\pi(\ell)\,d\ell}\bigg]^\frac1{q}.
\qedhere
\end{equation}
\end{lem}

Next, in preparation for the buckling argument, we show how  to bound local norms of $(\nabla\psi_E,\Sigma_E\mathds1_{\R^d\setminus\Ic})$ as appearing in the right-hand side of~\eqref{eq:CLTscaling} by corresponding large-scale averages.
This statement is inspired by~\cite{JO} in the context of homogenization for divergence-form linear elliptic equations.

\begin{prop}\label{prop:interpol}
Choose \mbox{$\chi\in C^\infty_c(B)$} with $\int_B\chi=1$, and set $\chi_r(x):=r^{-d}\chi(\frac xr)$.
Under Assumption~\ref{Hd}, for all $E\in\Md_0$, $1\ll_\chi r\ll_\chi R$, and $q,s\ge1$ with $|s-1|\ll1$,
\[\bigg\|\Big(\fint_{B_{R}}[(\nabla\psi_E,\Sigma_E\mathds1_{\R^d\setminus\Ic})]_2^{2s}\Big)^\frac1{2s}\bigg\|_{\Ld^{2q}(\Omega)}\,\lesssim_\chi\, |E|+\Big\|\int_{\R^d}\chi_r\,(\nabla\psi_E,\Sigma_E\mathds1_{\R^d\setminus\Ic})\Big\|_{\Ld^{2q}(\Omega)}.\qedhere\]
\end{prop}

Based on the above two propositions, we are now in position to proceed with the buckling argument and the proof of Theorem~\ref{th:cor+}.

\begin{proof}[Proof of Theorem~\ref{th:cor+}]
Let $E\in\Md_0$ be fixed with $|E|=1$.
We split the proof into three steps:
after some preliminary estimate, we establish the moment bounds~\eqref{eq:bnd-grad-cor1} on $(\nabla\psi_E,\Sigma_E\mathds1_{\R^d\setminus\Ic},\nabla\zeta_E)$ by a buckling argument,
before deducing the corresponding moment bounds~\eqref{eq:bnd-cor1} on $(\psi_E,\zeta_E)$ by integration.

\medskip
\step1 Preliminary: proof that for all $R\ge1$,
\begin{multline}\label{eq:bnd-inter-1.2}
\|[(\nabla\psi_E,\Sigma_E\mathds1_{\R^d\setminus\Ic},\nabla\zeta_E)]_2\|_{\Ld^{2q}(\Omega)}\\
\,\lesssim\,(R^\frac d2)^{1-\frac1q}\bigg\|\Big(\fint_{B_{R}}|(\nabla\psi_E,\Sigma_E\mathds1_{\R^d\setminus\Ic},\nabla\zeta_E)|^2\Big)^\frac12\bigg\|_{\Ld^{2q}(\Omega)}.
\end{multline}
For $R,q\ge1$, in view of local quadratic averages, the discrete $\ell^2-\ell^{2q}$ inequality yields
\begin{eqnarray*}
\lefteqn{\Big(\fint_{B_R(x)}[(\nabla\psi_E,\Sigma_E\mathds1_{\R^d\setminus\Ic},\nabla\zeta_E)]_2^{2q}\Big)^\frac1{2q}}\\
&\lesssim&\bigg(R^{-d}\sum_{z\in B_{2R}(x)\cap\frac1C\Z^d}[(\nabla\psi_E,\Sigma_E\mathds1_{\R^d\setminus\Ic},\nabla\zeta_E)]_2(z)^{2q}\bigg)^\frac1{2q}\\
&\lesssim& (R^\frac d2)^{1-\frac1q}\Big(\fint_{B_{4R}(x)}|(\nabla\psi_E,\Sigma_E\mathds1_{\R^d\setminus\Ic},\nabla\zeta_E)|^2\Big)^\frac12.
\end{eqnarray*}
Taking the $\Ld^{2q}(\Omega)$ norm and using the stationarity of $(\nabla\psi_E,\Sigma_E\mathds1_{\R^d\setminus\Ic},\nabla\zeta_E)$, the claim follows.

\medskip
\step2 Moment bounds~\eqref{eq:bnd-grad-cor1}.\\
Combining the results of Propositions~\ref{prop:CLTscaling} and~\ref{prop:interpol}, 
we find for all $1\ll_\chi r\ll_\chi R$ and $q,s\ge1$ with $1\ll q<\infty$ and $|s-1|\ll1$,
\begin{eqnarray*}
\lefteqn{\bigg\|\Big(\fint_{B_{R}}[(\nabla \psi_E,\Sigma_E\mathds1_{\R^d\setminus\Ic},\nabla\zeta_E)]_2^{2s}\Big)^\frac1{2s}\bigg\|_{\Ld^{2q}(\Omega)}}\\
&\lesssim_\chi&1+\Big\|\int_{\R^d}\chi_r(\nabla\psi_E,\Sigma_E\mathds1_{\R^d\setminus\Ic})\Big\|_{\Ld^{2q}(\Omega)}\\
&\lesssim_{q,\chi}&1+\big(\tfrac Rr\big)^{\frac d2}R^{-\frac d{2s'}}\bigg\|\Big(1+\fint_{B_R}[(\nabla\psi_E,\Sigma_E\mathds1_{\R^d\setminus\Ic})]_2^{2s}\Big)^\frac1{2s}\bigg\|_{\Ld^{2q}(\Omega)}.
\end{eqnarray*}
Letting $1\ll_\chi r\ll_\chi R$ be fixed with $r\simeq_\chi R$, and choosing $R\gg_{q,s,\chi}1$, we deduce
\[\bigg\|\Big(\fint_{B_{R}}[(\nabla \psi_E,\Sigma_E\mathds1_{\R^d\setminus\Ic},\nabla\zeta_E)]_2^{2s}\Big)^\frac1{2s}\bigg\|_{\Ld^{2q}(\Omega)}\lesssim_{q,s} 1.\]
Inserting this into~\eqref{eq:bnd-inter-1.2} together with Jensen's inequality, the conclusion~\eqref{eq:bnd-grad-cor1} follows.

\medskip
\step3 Moment bounds~\eqref{eq:bnd-cor1}.\\
We focus on the bound on $\psi_E$, while the argument for~$\zeta_E$ is similar.
Poincar\'e's inequality in~$B(x)$ gives
\begin{equation}\label{eq:phi-bnd0}
\bigg\|\Big[\psi_E-\fint_{B}\psi_E\Big]_2(x)\bigg\|_{\Ld^{2q}(\Omega)}\,\lesssim\,\|[\nabla\psi_E]_2\|_{\Ld^{2q}(\Omega)}+\Big\|\fint_{B(x)}\psi_E-\fint_{B}\psi_E\Big\|_{\Ld^{2q}(\Omega)},
\end{equation}
and it remains to estimate the second right-hand side term.
For that purpose, we write
\[\fint_{B(x)}\psi_E-\fint_{B}\psi_E=\int_{\R^d}\nabla\psi_E\cdot\nabla h_{x},\]
where $h_x$ denotes the unique decaying solution in $\R^d$ of
\[-\triangle h_x=\tfrac{1}{|B|}(\mathds1_{B(x)}-\mathds1_B).\]
Appealing to Proposition~\ref{prop:CLTscaling} together with the moment bounds~\eqref{eq:bnd-grad-cor1}, we find for all \mbox{$q<\infty$},
\[\Big\|\int_{\R^d}\nabla\psi_E\cdot\nabla h_{x}\Big\|_{\Ld^{2q}(\Omega)}\,\lesssim_q\,\|\nabla h_{x}\|_{\Ld^2(\R^d)}.\]
A direct computation with Green's kernel gives
\[\|\nabla h_{x}\|_{\Ld^2(\R^d)}\,\lesssim\,\mu_d(|x|),\]
and thus
\[\Big\|\fint_{B(x)}\psi_E-\fint_{B}\psi_E\Big\|_{\Ld^{2q}(\Omega)}\,\lesssim_q\,\mu_d(|x|).\]
Inserting this into~\eqref{eq:phi-bnd0}, together with the moment bounds~\eqref{eq:bnd-grad-cor1},
the conclusion~\eqref{eq:bnd-cor1} for~$\psi_E$ follows.
\end{proof}

\subsection{Proof of Proposition~\ref{prop:CLTscaling}}
Let $E\in\Md_0$ be fixed with $|E|=1$.
Applying the version~\eqref{eq:SGL-p} of the multiscale variance inequality~\eqref{eq:SGL} to control higher moments, we find
\begin{align}\label{eq:SG-phi}
&\Big\|\int_{\R^d}g\,(\nabla \psi_E,\Sigma_E\mathds1_{\R^d\setminus\Ic},\nabla\zeta_E)\Big\|_{\Ld^{2q}(\Omega)}^{2}\\
&\quad\,\lesssim_q\,\E\bigg[\int_0^\infty\bigg(\int_{\R^d}\Big(\partial^{\operatorname{osc}}_{\Ic ,B_\ell(x)}\int_{\R^d}g\,(\nabla \psi_E,\Sigma_E\mathds1_{\R^d\setminus\Ic},\nabla\zeta_E)\Big)^2dx\bigg)^q\langle\ell\rangle^{-dq}\,\pi(\ell)\,d\ell\bigg]^\frac1q,\nonumber
\end{align}
and it remains to estimate the oscillation of $\int_{\R^d} g(\nabla \psi_E,\Sigma_E\mathds1_{\R^d\setminus\Ic},\nabla\zeta_E)$ with respect to the inclusion process $\Ic $ on any ball~$B_\ell(x)$.
Given $\ell\ge0$ and $x\in\R^d$, and given a realization of $\Ic$, let $\Ic '$ be a locally finite point set satisfying the hardcore and regularity conditions in~\ref{Hd}, with $\Ic '\cap(\R^d\setminus B_\ell(x))=\Ic \cap(\R^d\setminus B_\ell(x))$, and denote by~$(\nabla\psi'_E,\Sigma_E'\mathds1_{\R^d\setminus\Ic'},\nabla\zeta_E')$ the corresponding extended corrector with $\Ic $ replaced by $\Ic'$ (this is obviously well-defined in $\Ld^2_\loc(\R^d)$ as the perturbation is compactly supported).
We split the proof into nine steps.

\medskip
\step{1} Preliminary: dual test functions and annealed estimates.\\
As we shall abundantly appeal to duality arguments in the proof, this first step is devoted to the construction of a number of useful dual test functions and to the proof of corresponding annealed estimates:
\begin{enumerate}[\quad$\bullet$]
\item Given a test function $g\in \Ld^\infty(\Omega;C^\infty_c(\R^d)^{d\times d})$, we let $\nabla u_g\in \Ld^\infty(\Omega;\Ld^2(\R^d)^{d\times d})$ denote the unique solution of the steady Stokes problem~\eqref{eq:test-v0}, and we recall that 
Theorem~\ref{th:meyers-ann} yields for all $|q-2|\ll1$,
\begin{equation}\label{eq:test-V+}
\quad\|[\nabla u_g]\|_{\Ld^2(\R^d;\Ld^q(\Omega))}\,\lesssim\,\|[g]_2\|_{\Ld^2(\R^d;\Ld^q(\Omega))}.
\end{equation}
\item Given a test function $g\in C^\infty_c(\R^d;\Ld^\infty(\Omega)^{d})$, we let $\nabla v_g\in \Ld^\infty(\Omega;\Ld^2(\R^d)^{d})$ denote the unique solution of
\begin{equation}\label{eq:test-W}
\quad-\triangle v_g=\Div (g),\quad\text{in $\R^d$},
\end{equation}
which satisfies for all $1<q<\infty$,
\begin{equation}\label{eq:test-W+}
\quad\|[\nabla v_g]_2\|_{\Ld^2(\R^d;\Ld^q(\Omega))}\,\lesssim_q\,\|[g]_2\|_{\Ld^2(\R^d;\Ld^q(\Omega))}.
\end{equation}
\item Given $g\in C^\infty_c(\R^d;\Ld^\infty(\Omega))$, there exists a vector field $s_g\in\Ld^\infty(\Omega;\dot H^1(\R^d)^d)$ such that~$s_g|_{I_n}$ is constant for all $n$, and such that for all $1<q<\infty$,
\begin{gather}
\quad\Div (s_g)= g\mathds1_{\R^d\setminus\Ic},\quad\text{in $\R^d$},\label{eq:test-R}\\
\quad\|[\nabla s_g]_2\|_{\Ld^2(\R^d;\Ld^q(\Omega))}\,\lesssim\,\|[g]_2\|_{\Ld^2(\R^d;\Ld^q(\Omega))},\nonumber
\end{gather}
\item Given $g\in C^\infty_c(\R^d;\Ld^\infty(\Omega)^{d\times d})$, there exists a $2$-tensor field $h_g\in\Ld^\infty(\Omega;H^1(\R^d)^{d\times d})$ such that $h_g|_{I_n}=\fint_{I_n}g$ for all $n$, and such that for all $1<q<\infty$,
\begin{gather}
\Div (h_g)=0,\quad\text{in $\R^d$},\label{eq:test-H}\\
\|[ h_g]_2\|_{\Ld^2(\R^d;\Ld^q(\Omega))}\,\lesssim\,\|[g]_2\|_{\Ld^2(\R^d;\Ld^q(\Omega))}.\nonumber
\end{gather}
\end{enumerate}

\medskip\noindent
The existence and uniqueness of $\nabla v_g$ is clear, and the annealed bound~\eqref{eq:test-W+} follows from Banach-valued Fourier multiplier theorems, e.g.\@ in form of the extrapolation result in~\cite[Theorem~3.15]{Kunstmann-Weis-04}.

\medskip\noindent
We turn to the construction of $s_g$.
First denote by $s_g^\circ:=\nabla w_g\in\Ld^\infty(\Omega;\dot H^1(\R^d)^d)$ the solution of
\[\triangle w_g=g\mathds1_{\R^d\setminus\Ic},\quad\text{in $\R^d$}.\]
In view of~\eqref{eq:test-W+}, it satisfies for all $1<q<\infty$,
\begin{equation*}
\|[\nabla s_g^\circ]_2\|_{\Ld^2(\R^d;\Ld^q(\Omega))}\,\lesssim_q\,\|[g]_2\|_{\Ld^2(\R^d;\Ld^q(\Omega))}.
\end{equation*}
Next, as in~\eqref{eq:ext-Bogo}, by a standard use of the Bogovskii operator in form of~\cite[Theorem~III.3.1]{Galdi}, for all~$n$, we can construct a vector field $s_g^n\in H^1_0(I_n+\delta B)^d$ such that $s_g^n=-s_g^\circ+\fint_{I_n}s_g^\circ$ in~$I_n$, and
\begin{gather*}
\Div (s_g^n) = 0,\\
\|\nabla s_g^n\|_{\Ld^2((I_n+\delta B)\setminus I_n)}\,\lesssim\,\|\nabla s_g^\circ\|_{\Ld^2(I_n)}.
\end{gather*}
Since the fattened inclusions $\{I_n+\delta B\}_n$ are disjoint, cf.~\ref{Hd}, the vector field $s_g:=s_g^\circ+\sum_ns_g^n$ (where we implicitly extend $s_g^n$ by $0$ outside $I_n+\delta B$) is checked to satisfy the required properties.

\medskip\noindent
It remains to construct $h_g$.
As in~\eqref{eq:ext-Bogo}, using the Bogovskii operator in form of~\cite[Theorem~III.3.1]{Galdi}, for all $n$, we can construct a $2$-tensor field $h_g^n \in H^1_0(I_n+\delta B)^{d\times d}$ such that $h_g^n|_{I_n}=\fint_{I_n}g$, and
\begin{gather*}
\Div (h_g^n)=0,\\
\|\nabla h_g^n\|_{\Ld^2((I_n+\delta B)\setminus I_n)}\,\lesssim\,\|g\|_{\Ld^2(I_n)},
\end{gather*}
and the tensor field $h_g=\sum_nh_g^n$ then satisfies the required properties.

\medskip
\step2 Preliminary: trace estimate.\\
For later reference, we prove the following general trace estimate: given a symmetric $2$-tensor field $H\in C^\infty_b(\R^d)^{d\times d}$ such that
\[\left\{\begin{array}{ll}
\Div(H)=0,&\text{in $(I_n+\delta B)\setminus I_n$,}\\
\int_{\partial I_n}H\nu=0,&\\
\int_{\partial I_n}\Theta(x-x_n)\cdot H\nu=0,&\text{for all $\Theta\in\Md^\Skew$},
\end{array}\right.\]
we have for all $g\in C^\infty_b(\R^d)^d$,
\begin{equation}\label{eq:trace-estim}
\Big|\int_{\partial I_n}g\cdot H\nu\Big|\,\lesssim\,\Big(\int_{(I_n+\delta B)\setminus I_n}|\!\D(g)|^2\Big)^\frac12\Big(\int_{(I_n+\delta B)\setminus I_n}|H|^2\Big)^\frac12.
\end{equation}
We start by considering the following auxiliary Neumann problem,
\[\left\{\begin{array}{ll}
-\triangle z_n+\nabla R_n=0,&\text{in $I_n$},\\
\Div(z_n)=0,&\text{in $I_n$},\\
\sigma(z_n,R_n)\nu=H\nu,&\text{on $\partial I_n$}.
\end{array}\right.\]
Well-posedness for this problem is obtained as for~\eqref{eq:sol-psiEn} thanks to the assumptions on~$H$, and the solution satisfies
\[\|(\nabla z_n,R_n)\|_{\Ld^2(I_n)}\,\lesssim\,\|H\|_{\Ld^2((I_n+\delta B)\setminus I_n)}.\]
Since Stokes' formula yields
\[\int_{\partial I_n}g\cdot H\nu\,=\,\int_{\partial I_n}g\cdot \sigma(z_n,R_n)\nu\,=\,\int_{I_n}\D(g): \sigma(z_n,R_n),\]
the claim~\eqref{eq:trace-estim} follows.

\medskip
\step{3} Proof of
\begin{eqnarray}
\int_{B_{\ell+3}(x)}|\nabla\psi_E'|^2&\lesssim&\int_{B_{\ell+3}(x)}\big(1+|\nabla\psi_E|^2\big),\label{ant.1.30-}\\
\int_{B_{\ell+3}(x)}|\Sigma_E'|^2\mathds1_{\R^d\setminus\Ic'}&\lesssim&\int_{B_{\ell+3}(x)}\big(1+|\nabla\psi_E|^2+|\Sigma_E|^2\mathds1_{\R^d\setminus\Ic}\big).\label{ant.1.3}
\end{eqnarray}
Equation~\eqref{eq:psiE-ref0} for $\psi_E-\psi_E'$ takes the form
\begin{multline}\label{eq:diff-psi-def}
-\triangle(\psi_E-\psi_E')+\nabla(\Sigma_E\mathds1_{\R^d\setminus\Ic}-\Sigma_E'\mathds1_{\R^d\setminus\Ic'})\\
\,=\,-\sum_n\delta_{\partial I_n}\sigma(\psi_E+Ex,\Sigma_E)\nu+\sum_n\delta_{\partial I_n'}\sigma(\psi_E'+Ex,\Sigma_E')\nu.
\end{multline}
Testing this equation with $\psi_E-\psi_E'$, we find
\begin{multline*}
\int_{\R^d}|\nabla(\psi_E-\psi_E')|^2
\,=\,-\sum_n\int_{\partial I_n}(\psi_E-\psi_E')\cdot\sigma(\psi_E+Ex,\Sigma_E)\nu\\
+\sum_n\int_{\partial I_n'}(\psi_E-\psi_E')\cdot\sigma(\psi_E'+Ex,\Sigma_E')\nu,
\end{multline*}
which, by the boundary conditions, turns into
\begin{multline}\label{eq:eqn-psi-psi'0}
\int_{\R^d}|\nabla(\psi_E-\psi_E')|^2
\,=\,\sum_{n:I_n\cap B_\ell(x)\ne\varnothing}\int_{\partial I_n}\psi_E'\cdot\sigma(\psi_E+Ex,\Sigma_E)\nu\\
+\sum_{n:I_n'\cap B_\ell(x)\ne\varnothing}\int_{\partial I_n'}\psi_E\cdot\sigma(\psi_E'+Ex,\Sigma_E')\nu.
\end{multline}
Note that, by Stokes' formula, the constraints $\Div(\psi_E)=\Div(\psi_E')=0$ allow to replace the pressures $\Sigma_E$ and~$\Sigma_E'$ in this identity by $\Sigma_E-c$ and $\Sigma_E'-c'$, respectively, for any constants $c,c'\in\R$.
Appealing to the trace estimate~\eqref{eq:trace-estim}, we are led to
\begin{multline*}
\int_{\R^d}|\nabla(\psi_E-\psi_E')|^2
\,\lesssim\,\Big(\int_{B_{\ell+3}(x)}\big(1+|\nabla\psi_E|^2+|\Sigma_E-c|^2\mathds1_{\R^d\setminus\Ic}\big)\Big)^\frac12\\
\times\Big(\int_{B_{\ell+3}(x)}\big(1+|\nabla\psi_E'|^2+|\Sigma_E'-c'|^2\mathds1_{\R^d\setminus\Ic'}\big)\Big)^\frac12.
\end{multline*}
Choosing $c:=\fint_{B_{\ell+3}(x)\setminus\Ic }\Sigma_E$ and $c':=\fint_{B_{\ell+3}(x)\setminus\Ic'}\Sigma_E'$,
and using the pressure estimate of Lemma~\ref{lem:pres}, we deduce
\begin{equation}\label{eq:preant1.3}
\int_{\R^d}|\nabla(\psi_E-\psi_E')|^2
\,\lesssim\,\Big(\int_{B_{\ell+3}(x)}\big(1+|\nabla\psi_E|^2\big)\Big)^\frac12
\Big(\int_{B_{\ell+3}(x)}\big(1+|\nabla\psi_E'|^2\big)\Big)^\frac12,
\end{equation}
and the claim~\eqref{ant.1.30-}  follows from the triangle inequality.

\medskip\noindent
Next, we establish the corresponding bound~\eqref{ant.1.3} on the perturbed pressure.
Using the Bogovskii operator as in the construction of $s_g$ in Step~1, we can construct a vector field $S_E\in\dot H^1(\R^d)^d$ such that
$S_E|_{I_n'}$ is constant for all $n$ and such that
\begin{gather*}
\Div (S_E)=(\Sigma_E\mathds1_{\R^d\setminus\Ic}-\Sigma_E'\mathds1_{\R^d\setminus\Ic'})\mathds1_{\R^d\setminus\Ic'},\\
\|\nabla S_E\|_{\Ld^2(\R^d)}\,\lesssim\,\|\Sigma_E\mathds1_{\R^d\setminus\Ic}-\Sigma_E'\mathds1_{\R^d\setminus\Ic'}\|_{\Ld^2(\R^d\setminus\Ic')}.
\end{gather*}
Testing equation~\eqref{eq:diff-psi-def} with $S_E$ and using the boundary conditions, we find
\begin{multline*}
\int_{\R^d}\Div(S_E)\,(\Sigma_E\mathds1_{\R^d\setminus\Ic}-\Sigma_E'\mathds1_{\R^d\setminus\Ic'})=\int_{\R^d}\nabla S_E:\nabla(\psi_E-\psi_E')\\
+\sum_{n:I_n\cap B_\ell(x)\ne\varnothing}\int_{\partial I_n}S_E\cdot\sigma(\psi_E+Ex,\Sigma_E)\nu,
\end{multline*}
which yields, by inserting the value of $\Div (S_E)$ and using again the trace estimate~\eqref{eq:trace-estim},
\begin{multline*}
\int_{\R^d\setminus\Ic'}|\Sigma_E\mathds1_{\R^d\setminus\Ic}-\Sigma_E'\mathds1_{\R^d\setminus\Ic'}|^2\lesssim\Big(\int_{\R^d}|\nabla S_E|^2\Big)^\frac12\\
\times\Big(\int_{\R^d}|\nabla(\psi_E-\psi_E')|^2+\int_{B_{\ell+3}(x)}\big(1+|\nabla\psi_E|^2+|\Sigma_E|^2\mathds1_{\R^d\setminus\Ic}\big)\Big)^\frac12.
\end{multline*}
Appealing to the bound on the norm of $\nabla S_E$, this yields
\begin{equation*}
\int_{\R^d\setminus\Ic'}|\Sigma_E\mathds1_{\R^d\setminus\Ic}-\Sigma_E'\mathds1_{\R^d\setminus\Ic'}|^2\,\lesssim\,
\int_{\R^d}|\nabla(\psi_E-\psi_E')|^2+\int_{B_{\ell+3}(x)}\big(1+|\nabla\psi_E|^2+|\Sigma_E|^2\mathds1_{\R^d\setminus\Ic}\big).
\end{equation*}
Combining this with~\eqref{eq:preant1.3} and \eqref{ant.1.30-}, the claim~\eqref{ant.1.3} follows by the triangle inequality.

\medskip
\step{4} Sensitivity of the corrector gradient outside the inclusions: for all $g\in C^\infty_c(\R^d)^{d\times d}$,
\begin{multline}\label{ant.1.1}
\Big|\int_{\R^d\setminus\Ic }g:\nabla\psi_E-\int_{\R^d\setminus\Ic '}g:\nabla\psi_E'\Big|\\
\,\lesssim\, \Big(\int_{B_{\ell+3}(x)}\big(|g|^2+|\nabla u_g|^2\big)\Big)^\frac12
\Big(\int_{B_{\ell+3}(x)}\big(1+|\nabla\psi_E|^2\big)\Big)^\frac12.
\end{multline}
Decomposing $\int_{\R^d \setminus\Ic }-\int_{\R^d \setminus\Ic '}=\int_{\Ic '\setminus\Ic }-\int_{\Ic \setminus\Ic '}$ and noting that $(\Ic '\setminus\Ic )\cup(\Ic \setminus\Ic ')\subset B_{\ell}(x)$, we find
\begin{multline}\label{eq:decomp-delphi0}
\Big|\int_{\R^d\setminus\Ic }g:\nabla\psi_E-\int_{\R^d\setminus\Ic' }g:\nabla\psi_E'\Big|\\
\lesssim\Big|\int_{\R^d\setminus\Ic }g:\nabla(\psi_E-\psi_E')\Big|+\Big(\int_{B_{\ell}(x)}|g|^2\Big)^\frac12\Big(\int_{B_{\ell}(x)}|\nabla\psi_E'|^2\Big)^\frac12.
\end{multline}
It remains to examine the first right-hand side term, for which we appeal to a duality argument, in terms of the solution $\nabla u_g$ of~\eqref{eq:test-v0}.
Testing with $\psi_E-\psi_E'$ the equation~\eqref{e:up-whole0} for $\nabla u_g$, and subtracting an arbitrary constant $c_1\in\R$ to the pressure $P_g$, we obtain
\[\int_{\R^d\setminus\Ic }g:\nabla(\psi_E-\psi_E')\,=\,-\int_{\R^d}\nabla u_g:\nabla(\psi_E-\psi_E')-\sum_n\int_{\partial I_n}(\psi_E-\psi_E')\cdot\big(g+\sigma(u_g,P_g-c_1)\big)\nu,\]
which, in view of the boundary conditions, turns into
\begin{multline}\label{ant.1.1.1}
\int_{\R^d\setminus\Ic }g:\nabla(\psi_E-\psi_E')\,=\,-\int_{\R^d}\nabla u_g:\nabla(\psi_E-\psi_E')\\
+\sum_{n:I_n\cap B_\ell(x)\ne\varnothing}\int_{\partial I_n}\psi_E'\cdot\big(g+\sigma(u_g,P_g-c_1)\big)\nu.
\end{multline}
Likewise, testing with $u_g$ the equation~\eqref{eq:diff-psi-def} for $\psi_E-\psi'_E$, we get for any constant~$c_2\in\R$,
\begin{equation*}
\int_{\R^d}\nabla u_g :\nabla(\psi_E-\psi_E')\,=\,
-\sum_n\int_{\partial I_{n}} u_g \cdot\sigma(\psi_E+Ex ,\Sigma_E)\nu
+\sum_n\int_{\partial I_{n}'} u_g \cdot\sigma(\psi_E'+Ex,\Sigma_E'-c_2)\nu,
\end{equation*}
which, in view of the boundary conditions, takes the form
\begin{equation*}
\int_{\R^d}\nabla u_g :\nabla(\psi_E-\psi_E')\,=\,
\sum_{n:I_n'\cap B_\ell(x)\ne\varnothing}\int_{\partial I_{n}'} u_g\cdot\sigma(\psi_E'+Ex,\Sigma_E'-c_2)\nu.
\end{equation*}
Combining this with \eqref{ant.1.1.1}, we obtain
\begin{multline*}
\int_{\R^d\setminus\Ic }g:\nabla(\psi_E-\psi_E')\,=\,
\sum_{n:I_n\cap B_\ell(x)\ne\varnothing}\int_{\partial I_n}\psi_E'\cdot\big(g+\sigma(u_g,P_g-c_1)\big)\nu\\
-\sum_{n:I_n'\cap B_\ell(x)\ne\varnothing}\int_{\partial I_{n}'} u_g \cdot\sigma(\psi_E'+Ex,\Sigma_E'-c_2)\nu.
\end{multline*}
Appealing to the trace estimate~\eqref{eq:trace-estim}, we deduce
\begin{multline*}
\Big|\int_{\R^d\setminus\Ic }g:\nabla(\psi_E-\psi_E')\Big|
\,\lesssim\, \Big(\int_{B_{\ell+3}(x)}\big(|g|^2+|\nabla u_g|^2+|P_g -c_1|^2\mathds1_{\R^d\setminus\Ic }\big)\Big)^\frac12\\
\times\Big(\int_{B_{\ell+3}(x)}\big(1+|\nabla\psi_E'|^2+|\Sigma_E'-c_2|^2\mathds1_{\R^d\setminus\Ic '}\big)\Big)^\frac12.
\end{multline*}
Choosing $c_1:=\fint_{B_{\ell+3}(x)\setminus\Ic }P_g$ and $c_2:=\fint_{B_{\ell+3}(x)\setminus\Ic' }\Sigma_E'$, and appealing to the pressure estimate of Lemma~\ref{lem:pres},
we deduce
\begin{equation}\label{eq:preant1.1}
\Big|\int_{\R^d\setminus\Ic }g:\nabla(\psi_E-\psi_E')\Big|
\,\lesssim\, \Big(\int_{B_{\ell+3}(x)}\big(|g|^2+|\nabla u_g|^2\big)\Big)^\frac12
\Big(\int_{B_{\ell+3}(x)}\big(1+|\nabla\psi_E'|^2\big)\Big)^\frac12.
\end{equation}
Combined with~\eqref{eq:decomp-delphi0} and with the result~\eqref{ant.1.30-} of Step~3, this yields the claim \eqref{ant.1.1}.

\medskip
\step{5} Sensitivity of the corrector gradient inside the inclusions: for all $g\in C^\infty_c(\R^d)^{d\times d}$,
\begin{multline}\label{eq:diff-psi-ins}
\Big|\int_{\Ic }g:\nabla\psi_E-\int_{\Ic '}g:\nabla\psi_E'\Big|\,\lesssim\,\Big|\int_{\R^d\setminus\Ic }h_g :\nabla\psi_E-\int_{\R^d\setminus\Ic }h_g :\nabla\psi_E'\Big|\\
+\Big(\int_{B_{\ell+3}(x)}\big(|g|^2+|h_g|^2\big)\Big)^\frac12
\Big(\int_{B_{\ell+3}(x)}\big(1+|\nabla\psi_E|^2\big)\Big)^\frac12.
\end{multline}
First decompose
\begin{multline*}
{\Big|\int_{\Ic }g:\nabla\psi_E-\int_{\Ic '}g:\nabla\psi_E'\Big|}
\,\le\, \bigg| \sum_{n:I_n\cap B_\ell(x)=\varnothing} \int_{I_{n}} g: \nabla (\psi_E - \psi_E ') \bigg|
\\
+\sum_{n:I_n\cap B_\ell(x)\ne\varnothing}\Big| \int_{I_{n}} g: \nabla \psi_E\Big| +\sum_{n:I_n'\cap B_\ell(x)\ne\varnothing} \Big|\int_{I_{n}'} g:  \nabla \psi_E'\Big|.
\end{multline*}
Since $\psi_E$ and $\psi_E'$ are both affine inside inclusions $I_{n}$'s with $I_{n}\cap B_\ell(x)=\varnothing$, we can rewrite
\begin{multline*}
\Big|\int_{\Ic }g:\nabla\psi_E-\int_{\Ic '}g:\nabla\psi_E'\Big|
\,\lesssim\,\bigg|\sum_{n}\Big(\fint_{I_{n}}g\Big):\int_{I_{n}}\nabla(\psi_E-\psi_E')\bigg|\\
+\Big(\int_{B_{\ell+2}(x)}|g|^2\Big)^\frac12\Big(\int_{B_{\ell+2}(x)}\big(|\nabla\psi_E|^2+|\nabla\psi_E'|^2\big)\Big)^\frac12,
\end{multline*}
and it remains to analyze the first right-hand side term.
In terms of the $2$-tensor field $h_g$ defined in~\eqref{eq:test-H}, we can write
by means of Stokes' formula,
\begin{eqnarray*}
\sum_{n}\Big(\fint_{I_{n}}g\Big):\int_{I_{n}}\nabla(\psi_E-\psi_E')&=&\sum_{n}\Big(\fint_{I_{n}}g\Big):\int_{\partial I_{n}}(\psi_E-\psi_E')\otimes\nu\\
&=&\sum_{n}\int_{\partial I_{n}}h_g :(\psi_E-\psi_E')\otimes\nu\\
&=&-\int_{\R^d\setminus\Ic }\partial_i\big(h_g :(\psi_E-\psi_E')\otimes\ee_i\big)\\
&=&-\int_{\R^d\setminus\Ic }h_g :\nabla(\psi_E-\psi_E'),
\end{eqnarray*}
where in the last identity we used that $\Div(h_g)=0$.
Combining with the above, and using the result~\eqref{ant.1.30-} of Step~3, the claim~\eqref{eq:diff-psi-ins} follows.

\medskip
\step{6} Sensitivity of the corrector pressure: for all $g\in C^\infty_c(\R^d)$,
\begin{multline}\label{eq:bnd-Sig-diff}
\Big|\int_{\R^d\setminus\Ic}g\,\Sigma_E-\int_{\R^d\setminus\Ic'}g\,\Sigma_E'\Big|\,\lesssim\,
\Big|\int_{\R^d\setminus\Ic}\nabla s_g:\nabla\psi_E-\int_{\R^d\setminus\Ic'}\nabla s_g:\nabla\psi_E'\Big|\\
+\Big(\int_{B_{\ell+3}(x)}\big(|g|^2+|\nabla s_g|^2\big)\Big)^\frac12\Big(\int_{B_{\ell+3}(x)}\big(1+|\nabla\psi_E|^2+|\Sigma_E|^2\mathds1_{\R^d\setminus\Ic}\big)\Big)^\frac12.
\end{multline}
In terms of the vector field $s_g$ defined in~\eqref{eq:test-R}, we can write
\begin{eqnarray*}
\int_{\R^d\setminus\Ic}g\,\Sigma_E-\int_{\R^d\setminus\Ic'}g\,\Sigma_E'&=&\int_{\R^d\setminus\Ic}g\big(\Sigma_E\mathds1_{\R^d\setminus\Ic}-\Sigma_E'\mathds1_{\R^d\setminus\Ic'}\big)-\int_{\Ic\setminus\Ic'}g\,\Sigma_E'\\
&=&\int_{\R^d}\Div(s_g)\,(\Sigma_E\mathds1_{\R^d\setminus\Ic}-\Sigma_E'\mathds1_{\R^d\setminus\Ic'}) -\int_{\Ic\setminus\Ic'}g\,\Sigma_E',
\end{eqnarray*}
and thus, using the equation~\eqref{eq:diff-psi-def} for $\psi_E-\psi_E'$, the boundary conditions, and the fact that $s_g$ is constant on the inclusion $I_n$
\begin{multline}\label{eq:pre-bnd-Sig-diff}
\int_{\R^d\setminus\Ic}g\,\Sigma_E-\int_{\R^d\setminus\Ic'}g\,\Sigma_E'\,=\,
\int_{\R^d}\nabla s_g:\nabla(\psi_E-\psi_E')\\
-\sum_{n:I_n'\cap B_\ell(x)\ne\varnothing}\int_{\partial I_n'}s_g\cdot\sigma(\psi_E'+Ex,\Sigma_E')\nu
-\int_{\Ic\setminus\Ic'}g\,\Sigma_E'.
\end{multline}
As $s_g|_{I_n}$ is constant for all $n$,  $\nabla s_g=0$ in $\Ic$,
and since $\Ic \setminus \Ic' \subset B_{\ell}(x)$ the first right-hand side term satisfies
\begin{multline}\label{eq:rgcst-psidiff}
\bigg|\int_{\R^d}\nabla s_g:\nabla(\psi_E-\psi_E')-\Big(\int_{\R^d\setminus\Ic}\nabla s_g:\nabla\psi_E-\int_{\R^d\setminus\Ic'}\nabla s_g:\nabla\psi_E'\Big)\bigg|\\
\,\le\,\Big(\int_{B_{\ell}(x)}|\nabla s_g|^2\Big)^\frac12\Big(\int_{B_{\ell}(x)}|\nabla\psi_E'|^2\Big)^\frac12.
\end{multline}
Combining this with~\eqref{eq:pre-bnd-Sig-diff}, appealing to the trace estimate~\eqref{eq:trace-estim}, and using~\eqref{ant.1.30-}--\eqref{ant.1.3} in Step~3, the claim~\eqref{eq:bnd-Sig-diff}  follows.

\medskip
\step{7} Sensitivity of the extended flux: for all $g\in C^\infty_c(\R^d)^{d\times d}_\Sym$,
\begin{multline}\label{eq:bnd-J-diff}
\Big|\int_{\R^d}g:(J_E-J_E')\Big|\\
\,\lesssim\,
\Big|\int_{\R^d\setminus\Ic}g:\nabla\psi_E-\int_{\R^d\setminus\Ic'}g:\nabla\psi_E'\Big|
+\Big|\int_{\R^d\setminus\Ic}\Tr(g)\,\Sigma_E-\int_{\R^d\setminus\Ic'}\Tr(g)\,\Sigma_E'\Big|\\
+\Big(\int_{B_{\ell+3}(x)}|g|^2\Big)^\frac12
\Big(\int_{B_{\ell+3}(x)}\big(1+|\nabla\psi_E|^2+|\Sigma_E|^2\mathds1_{\R^d\setminus\Ic}\big)\Big)^\frac12.
\end{multline}
The definition~\eqref{eq:def-J} of $J_E$ yields
\begin{multline*}
\int_{\R^d}g:(J_E-J_E')\,=\,2\Big(\int_{\R^d\setminus\Ic}g:(\nabla \psi_E+E)-\int_{\R^d\setminus\Ic'}g:(\nabla\psi_E'+E)\Big)\\
-\Big(\int_{\R^d\setminus\Ic}\Tr(g)\,\Sigma_E-\int_{\R^d\setminus\Ic'}\Tr(g)\,\Sigma_E'\Big)\\
+\sum_{n:I_n\cap B_\ell(x)\ne\varnothing}\int_{I_n}g:\sigma(\psi_E^{n},\Sigma_E^{n})-\sum_{n:I_n'\cap B_\ell(x)\ne\varnothing}\int_{I_n'}g:\sigma(\psi_E^{\prime n},\Sigma_E^{\prime n}),
\end{multline*}
and the claim~\eqref{eq:bnd-J-diff} then follows by using~\eqref{eq:bnd-psinSign} to estimate the last two right-hand side terms.

\medskip
\step{8} Sensitivity of the flux corrector: for all $g\in C^\infty_c(\R^d)^{d}$,
\begin{equation}\label{eq:bnd-zeta-diff}
\Big|\int_{\R^d}g\cdot\nabla(\zeta_{E;ijk}-\zeta_{E;ijk}')\Big|\,\lesssim\,
\Big|\int_{\R^d}\nabla v_g\otimes(J_E-J_E')\Big|.
\end{equation}
In terms of the auxiliary field $\nabla v_g$ defined in~\eqref{eq:test-W},
we can write
\begin{equation*}
\int_{\R^d}g\cdot\nabla\zeta_{E;ijk}-\int_{\R^d}g\cdot\nabla\zeta_{E;ijk}'\,=\,-\int_{\R^d}\nabla v_g\cdot\nabla\zeta_{E;ijk}+\int_{\R^d}\nabla v_g\cdot\nabla\zeta_{E;ijk}',
\end{equation*}
which, in view of the equation~\eqref{eq:def-zeta} for $\zeta_E$, takes the form
\begin{equation*}
\int_{\R^d}g\cdot\nabla\zeta_{E;ijk}-\int_{\R^d}g\cdot\nabla\zeta_{E;ijk}'\,=\,\int_{\R^d}\partial_j v_g (J_{E;ik}-J_{E;ik}')-\int_{\R^d}\partial_k v_g(J_{E;ij}-J_{E;ij}'),
\end{equation*}
and the claim~\eqref{eq:bnd-zeta-diff} follows.

\medskip
\step{9} Conclusion.\\
Iteratively combining the results~\eqref{ant.1.1}, \eqref{eq:diff-psi-ins}, \eqref{eq:bnd-Sig-diff}, \eqref{eq:bnd-J-diff}, and \eqref{eq:bnd-zeta-diff} of Steps~4--8,
we obtain for all $g\in C^\infty_c(\R^d)$,
\begin{equation*}
\Big|\partial^{\operatorname{osc}}_{\Ic ,B_\ell(x)}\int_{\R^d}g\,(\nabla \psi_E,\Sigma_E\mathds1_{\R^d\setminus\Ic},\nabla\zeta_E)\Big|
\,\lesssim\,\ell^dM_\ell(x)\Big(\fint_{B_{\ell+3}(x)}\big(|A[g]|^2+|\nabla U [A[g]]|^2\big)\Big)^\frac12,
\end{equation*}
where we have set for abbreviation
\begin{eqnarray*}
M_\ell(x)&:=&\Big(\fint_{B_{\ell+3}(x)}\big(1+|\nabla\psi_E|^2+|\Sigma_E|^2\mathds1_{\R^d\setminus\Ic}\big)\Big)^\frac12,\\
A[g]&:=&\big(g,H[g],\nabla S[g],\nabla V[g],\nabla S[\nabla V[g]]\big),
\end{eqnarray*}
in terms of the following linear operators
\[\nabla U[g]:=\nabla u_g,\quad \nabla V[g]:=\nabla v_g,\quad \nabla S[g]:=\nabla s_g,\quad H[g]:=h_g,\]
as defined in Step~1.
Inserting this into~\eqref{eq:SG-phi}, we find for all $q<\infty$,
\begin{multline}\label{eq:pre-bnd-aver-phi}
\Big\|\int_{\R^d}g\,(\nabla \psi_E,\Sigma_E\mathds1_{\R^d\setminus\Ic},\nabla\zeta_E)\Big\|_{\Ld^{2q}(\Omega)}^{2}\\
\,\lesssim_q\,\,\E\bigg[\int_0^\infty\bigg(\int_{\R^d}M_\ell(x)^2\Big(\fint_{B_{\ell+3}(x)}\big(|A[g]|^2+|\nabla U [A[g]]|^2\big)\Big)dx\bigg)^q\langle\ell\rangle^{dq}\,\pi(\ell)\,d\ell\bigg]^\frac1q.
\end{multline}
Before we estimate the right-hand side of \eqref{eq:pre-bnd-aver-phi}, we smuggle in a spatial average at some arbitrary scale $R\ge1$: setting $|f|^2:=|A[g]|^2+|\nabla U [A[g]]|^2$ for shortness,
\begin{equation*}
\int_{\R^d}M_\ell(x)^2\Big(\fint_{B_{\ell+3}(x)} |f|^2\Big)dx
\,\lesssim\,\int_{\R^d}\Big(\sup_{B_R(y)}M_\ell^2\Big)\bigg(\fint_{B_{\ell+3}(y)}\Big(\fint_{B_{R}(x)}|f|^2\Big)dx\bigg)dy.
\end{equation*}
We then use a duality argument to compute the $\Ld^q(\Omega)$ norm of this expression,
\begin{multline*}
\E\bigg[\bigg(\int_{\R^d}M_\ell(x)^2\Big(\fint_{B_{\ell+3}(x)}|f|^2\Big)dx\bigg)^q\bigg]^\frac1q\\
\hspace{-6cm}\,\lesssim\,\sup_{\|X\|_{\Ld^{2q'}(\Omega)}=1}\,\E\bigg[\int_{\R^d}\Big(\sup_{ B_R(y)}M_\ell^2\Big)
\bigg(\fint_{B_{\ell+3}(y)}\Big(\fint_{B_{R}(x)}|Xf|^2\Big)\,dx\bigg)\,dy\bigg].
\end{multline*}
where the supremum runs over random variables $X$ independent of the space variable.
By H\"older's inequality and by stationarity of~$M_\ell$, we find
\begin{multline*}
\E\bigg[\bigg(\int_{\R^d}M_\ell(x)^2\Big(\fint_{B_{\ell+3}(x)} |f|^2\Big)\,dx\bigg)^q\bigg]^\frac1q\\
\,\lesssim\, \Big\|\sup_{B_R}M_\ell\Big\|_{\Ld^{2q}(\Omega)}^2~
\sup_{\|X\|_{\Ld^{2q'}(\Omega)}=1} \int_{\R^d}\E \bigg[\bigg(\fint_{B_{\ell+3}(y)} \Big(\fint_{B_{R}(x)}|Xf|^2\Big)\,dx\bigg)^{q'} \bigg]^\frac1{q'}dy,
\end{multline*}
which, by Jensen's inequality, yields
\begin{multline}\label{eq:bnd-Mell-subd}
\E\bigg[\bigg(\int_{\R^d}M_\ell(x)^2\Big(\fint_{B_{\ell+3}(x)}|f|^2\Big)\,dx\bigg)^q\bigg]^\frac1q\\
\,\lesssim\, \Big\|\sup_{B_R}M_\ell\Big\|_{\Ld^{2q}(\Omega)}^2~
\sup_{\|X\|_{\Ld^{2q'}(\Omega)}=1}\|[Xf]_2\|_{\Ld^2(\R^d;\Ld^{2q'}(\Omega))}^2.
\end{multline}
Appealing to the annealed estimate in~\eqref{eq:test-V+}, we find for $q\gg1$ (hence $|2q'-2|\ll1$),
\begin{eqnarray*}
\|[X\nabla V[A[g]]]_2\|_{\Ld^2(\R^d;\Ld^{2q'}(\Omega))}&=&\|[\nabla V[A[Xg]]]_2\|_{\Ld^2(\R^d;\Ld^{2q'}(\Omega))}\\
&\lesssim&\|[A[Xg]]_2\|_{\Ld^2(\R^d;\Ld^{2q'}(\Omega))},
\end{eqnarray*}
while the annealed estimates in~\eqref{eq:test-W+}, \eqref{eq:test-R}, and~\eqref{eq:test-H} yield for $q>1$,
\begin{eqnarray*}
\|[A[Xg]]_2\|_{\Ld^2(\R^d;\Ld^{2q'}(\Omega))}
\,\lesssim\,\|[Xg]_2\|_{\Ld^2(\R^d;\Ld^{2q'}(\Omega))}\,=\,\|X\|_{\Ld^{2q'}(\Omega)}\|g\|_{\Ld^2(\R^d)}.
\end{eqnarray*}
Using these bounds in combination with~\eqref{eq:pre-bnd-aver-phi} and~\eqref{eq:bnd-Mell-subd}, together with the superalgebraic decay of the weight $\pi$ in form of Jensen's inequality, cf.~Assumption~\ref{Mix+}, we obtain for all $1\ll q<\infty$,
\begin{equation*}
\Big\|\int_{\R^d}g\,(\nabla \psi_E,\Sigma_E\mathds1_{\R^d\setminus\Ic},\nabla\zeta_E)\Big\|_{\Ld^{2q}(\Omega)}^{2}
\,\lesssim_q\,\Big\|\sup_{B_R}M_\ell\Big\|_{\Ld^{2q}(\Omega)}^2\| g\|_{\Ld^2(\R^d)}^2.
\end{equation*}
Finally, by stationarity and by the $\ell^{2s}-\ell^\infty$ inequality, the supremum of $M_\ell$ can be estimated as follows, for all $s\ge1$,
\begin{equation*}
\Big\|\sup_{B_R}M_\ell\Big\|_{\Ld^{2q}(\Omega)}\,\lesssim\,\bigg\|\Big(1+\int_{B_{R}}[(\nabla\psi_E,\Sigma_E\mathds1_{\R^d\setminus\Ic})]_2^{2s}\Big)^\frac1{2s}\bigg\|_{\Ld^{2q}(\Omega)},
\end{equation*}
and the conclusion~\eqref{eq:CLTscaling} follows.
\qed

\subsection{Proof of Proposition~\ref{prop:interpol}}
Let $E\in\Md_0$ be fixed with $|E|=1$.
We split the proof into three steps.

\medskip
\step1 Meyers-type perturbative argument: for all $s\ge1$ with $|s-1|\ll1$, for all $R,K\ge1$ and $c_R\in\R^d$,
\begin{equation}\label{gl.2}
\Big(\fint_{B_R}[\nabla \psi_E]_2^{2s}\Big)^\frac1{s}
\lesssim K^2\Big(1+\frac{1}{R^2}\fint_{B_{CR}}|\psi_E-c_R|^2\Big)+\frac1{K^2}\fint_{B_{CR}}|\nabla \psi_E|^2.
\end{equation}
Arguing as in~\eqref{eq:Cacc-Sob}, with $u_g$ replaced by $\psi_E+Ex$ and with $g=0$, we obtain the following Caccioppoli-type inequality: for all balls $D\subset\R^d$ with radius $r_D\ge3$, for all $K\ge1$ and $c_D\in\R^d$,
\begin{equation}\label{eq:preconcl-buckle}
\fint_{D}|\nabla \psi_E|^2
\,\lesssim\, K^2\Big(1+\frac{1}{r_D^2}\fint_{2D}|\psi_E-c_D|^2\Big)
+\frac1{K^2}\fint_{2D}|\nabla\psi_E|^2.
\end{equation}
Using the Poincar\'e-Sobolev inequality to estimate the first right-hand side term, with the choice $c_D:=\fint_{2D}\psi_E$, we deduce 
\begin{equation*}
\Big(\fint_{D}|\nabla \psi_E|^2\Big)^\frac12
\,\lesssim \,K\Big(1+\fint_{2D}|\nabla\psi_E|^\frac{2d}{d+2}\Big)^\frac{d+2}{2d}
+\frac1K\Big(\fint_{2D}|\nabla\psi_E|^2\Big)^\frac12.
\end{equation*}
While this is proven for all balls $D$ with radius $r_D\ge3$, smuggling in local quadratic averages at scale~1 allows to infer that for all balls $D$ (with any radius $r_D>0$) and $K\ge1$, 
\begin{equation*}
\Big(\fint_{D}[\nabla \psi_E]_2^2\Big)^\frac12
\,\lesssim\, K\Big(1+\fint_{2D}[\nabla\psi_E]_2^\frac{2d}{d+2}\Big)^\frac{d+2}{2d}
+\frac1K\Big(\fint_{2D}[\nabla\psi_E]_2^2\Big)^\frac12.
\end{equation*}
Choosing $K$ large enough and applying Gehring's lemma in form of Lemma~\ref{lem:Gehring}, we deduce the following Meyers-type estimate: for all $s\ge1$ with $|s-1|\ll1$, and all~$R>0$,
\begin{equation*}
\Big(\fint_{B_R}[\nabla \psi_E]_2^{2s}\Big)^\frac1{s}
\,\lesssim\, 1+\fint_{B_{CR}}[\nabla\psi_E]_2^2.
\end{equation*}
Combining this with~\eqref{eq:preconcl-buckle}, the claim~\eqref{gl.2} follows.

\medskip
\step2 Conclusion on $\nabla\psi_E$: for all $1\le r\ll_\chi R$ and $q,s\ge1$ with $|s-1|\ll1$,
\begin{equation*}
\bigg\|\Big(\fint_{B_{R}}[\nabla \psi_E]_2^{2s}\Big)^\frac1{2s}\bigg\|_{\Ld^{2q}(\Omega)}\,\lesssim_\chi\,1+\Big\|\int_{\R^d}\chi_r\nabla\psi_E\Big\|_{\Ld^{2q}(\Omega)}.
\end{equation*}
For $1\le r\le R$, choosing $c_R:=\fint_{B_{CR}}\chi_r\ast\psi_E$, Poincar\'e's inequality yields
\begin{eqnarray*}
\fint_{B_{CR}}|\psi_E-c_R|^2&\lesssim &\fint_{B_{CR}}|\psi_E-\chi_r\ast\psi_E|^2+\fint_{B_{CR}}|\chi_r\ast\psi_E-c_R|^2\\
&\lesssim_\chi&r^2\fint_{B_{CR}}|\nabla\psi_E|^2+R^2\fint_{B_{CR}}|\chi_r\ast\nabla\psi_E|^2.
\end{eqnarray*}
Inserting this into~\eqref{gl.2}, we find
\begin{equation*}
\Big(\fint_{B_R}[\nabla \psi_E]_2^{2s}\Big)^\frac1s
\,\lesssim\,K^2+ \Big(K^2\frac{r^2}{R^2}+\frac1{K^2}\Big)\fint_{B_{CR}}|\nabla\psi_E|^2+K^2\fint_{B_{CR}}|\chi_r\ast\nabla\psi_E|^2.
\end{equation*}
Taking the $\Ld^{q}(\Omega)$ norm, and using that stationarity and Jensen's inequality yield
\[\Big\|\fint_{B_{CR}}|\nabla \psi_E|^2\Big\|_{\Ld^q(\Omega)}\,\lesssim\,\Big\|\fint_{B_{R}}[\nabla \psi_E]_2^2\Big\|_{\Ld^q(\Omega)}\,\lesssim\,\bigg\|\Big(\fint_{B_{R}}[\nabla \psi_E]_2^{2s}\Big)^\frac1{2s}\bigg\|_{\Ld^{2q}(\Omega)}^2,\]
and
\[\Big\|\fint_{B_{CR}} |\chi_r\ast\nabla \psi_E |^2\Big\|_{\Ld^q(\Omega)}\, \le\, \|\chi_r \ast \nabla \psi_E \|_{\Ld^{2q}(\Omega)}^2,\]
we deduce
\begin{multline*}
\bigg\|\Big(\fint_{B_{R}}[\nabla \psi_E]_2^{2s}\Big)^\frac1{2s}\bigg\|_{\Ld^{2q}(\Omega)}\\
\,\lesssim_\chi\,K+\Big(K\frac{r}{R}+\frac1{K}\Big)\bigg\|\Big(\fint_{B_{R}}[\nabla \psi_E]_2^{2s}\Big)^\frac1{2s}\bigg\|_{\Ld^{2q}(\Omega)}
+K\Big\|\int_{\R^d}\chi_r\nabla\psi_E\Big\|_{\Ld^{2q}(\Omega)}.
\end{multline*}
Choosing $K\gg1$ and $R\gg_{K,\chi} r$, the second right-hand side term can be absorbed into the left-hand side and the claim follows.

\medskip
\step3 Conclusion on the pressure $\Sigma_E$.\\
For all $R,s\ge1$,
we decompose
\begin{equation*} 
\Big(\fint_{B_R}[\Sigma_E\mathds1_{\R^d\setminus\Ic}]_2^{2s}\Big)^\frac1s\,\lesssim\,
\bigg(\fint_{B_R}\Big[\Big(\Sigma_E-\fint_{B_R\setminus\Ic}\Sigma_E\Big)\mathds1_{\R^d\setminus\Ic}\Big]_2^{2s}\bigg)^\frac1s+\Big|\fint_{B_R\setminus\Ic}\Sigma_E\Big|^2.
\end{equation*}
Appealing to the pressure estimate of Lemma~\ref{lem:pres} to estimate the first right-hand side term,
and further decomposing the second term, we obtain for all $1\le r\le R$, assuming that $\int_{\R^d\setminus\Ic}\chi_r\simeq\int_{\R^d}\chi_r=1$ (which holds automatically provided $r\gg_\chi1$ in view of the hardcore assumption, cf.~\ref{Hd}),
\begin{multline*} 
\Big(\fint_{B_R}[\Sigma_E\mathds1_{\R^d\setminus\Ic}]_2^{2s}\Big)^\frac1s\,\lesssim\,1+
\Big(\fint_{B_R}[\nabla\psi_E]_2^{2s}\Big)^\frac1s+\Big|\int_{\R^d}\chi_r\Sigma_E\mathds1_{\R^d\setminus\Ic}\Big|^2\\
+\bigg|\int_{\R^d}\chi_r\Big(\Sigma_E-\fint_{B_R\setminus\Ic}\Sigma_E\Big)\mathds1_{\R^d\setminus\Ic}\bigg|^2.
\end{multline*}
It remains to estimate the last right-hand side term. By the Cauchy--Schwarz inequality, for~$r\ll_\chi R$ such that $\chi_r$ is supported in $B_R$, using again the pressure estimate of Lemma~\ref{lem:pres}, we find
\begin{eqnarray*} 
\bigg|\int_{\R^d}\chi_r\Big(\Sigma_E-\fint_{B_R\setminus\Ic}\Sigma_E\Big)\mathds1_{\R^d\setminus\Ic}\bigg|^2&\lesssim&\Big(R^d\int_{\R^d}|\chi_r|^2\Big)\fint_{B_R}\Big|\Sigma_E-\fint_{B_R\setminus\Ic}\Sigma_E\Big|^2\mathds1_{\R^d\setminus\Ic}\\
&\lesssim&\Big(R^d\int_{\R^d}|\chi_r|^2\Big)\Big(1+\fint_{B_R}|\nabla\psi_E|^2\Big).
\end{eqnarray*}
Since for $r\le R$ we have $R^d\int_{\R^d}|\chi_r|^2\le\|\chi\|_{\Ld^\infty(\R^d)}\|\chi\|_{\Ld^1(\R^d)}$, we conclude
\begin{equation*} 
\Big(\fint_{B_R}[\Sigma_E\mathds1_{\R^d\setminus\Ic}]_2^{2s}\Big)^\frac1s\,\lesssim_\chi\,1+
\Big(\fint_{B_{R}}[\nabla\psi_E]_2^{2s}\Big)^\frac1s+\Big|\int_{\R^d}\chi_r\Sigma_E\mathds1_{\R^d\setminus\Ic}\Big|^2.
\end{equation*}
Combined with the results on $\nabla\psi_E$ in Step~2, the conclusion follows.\qed

\section{Large-scale regularity}\label{sec:reg}

This section is devoted to the development of a large-scale regularity theory for the steady Stokes problem~\eqref{eq:test-v0}, and to the proof of Theorems~\ref{th:schauder}, \ref{th:lp-reg}, and~\ref{th:ann-reg}.
We take inspiration from the theory developed recently in the model setting of divergence-form linear elliptic equations with random coefficients~\cite{AS,AM-16,AKM1,AKM2,AKM-book,GNO-reg,DO1,JO}, and we focus more precisely on the formulation in~\cite{GNO-reg,DO1}.

\subsection{Structure of the argument}
In the formulation of~\cite{GNO-reg}, for divergence-form linear elliptic equations, the key ingredient to large-scale regularity theory is encapsulated in a perturbative statement encoding an improvement of flatness similar to what holds for harmonic functions, at the price of controlling the linear growth of an extended corrector, cf.~\cite[Proposition~1]{GNO-reg}. 
In the heterogeneous setting, we recall that Euclidean coordinates are naturally corrected by correctors, and flatness is understood as closeness to gradients of such corrected coordinates.
The following proposition is the extension of such a result in the context of the steady Stokes problem~\eqref{eq:test-v0};
the proof is postponed to Section~\ref{sec:pr-p1}.

\begin{prop}[Perturbative improvement of flatness]\label{p1}
There exists an exponent $\e\simeq1$ such that the following holds:
For all $R\gg1$, if $\nabla u$ is a solution of the following free steady Stokes problem in~$B_R$,
\begin{equation}\label{eq:Stokes-zero}
\left\{\begin{array}{ll}
-\triangle u+\nabla P=0,&\text{in $B_R\setminus\Ic$},\\
\Div(u)=0,&\text{in $B_R$},\\
\D(u)=0,&\text{in $\Ic\cap B_R$},\\
\int_{\partial I_n}\sigma(u,P)\nu=0,&\forall n:I_n \subset B_R,\\
\int_{\partial I_n}\Theta(x-x_n)\cdot\sigma(u,P)\nu=0,&\forall n:I_n \subset B_R,\,\forall\Theta\in\Md^\Skew,
\end{array}\right.
\end{equation}
then there exists a matrix $E_0\in\Md_0$
such that for all $4\le r\le R$,
\begin{equation}\label{e.p1.1}
\fint_{B_r}|\nabla u-(\nabla\psi_{E_0}+E_0)|^2
\,\lesssim\,\Big((\tfrac{r}{R})^2+(\tfrac{R}{r})^{d+2}(1\wedge\gamma_R)^{2\e}\Big)\fint_{B_R}|\!\D(u)|^2,
\end{equation}
where we have set for abbreviation,
\begin{equation}\label{e.p1.2}
  \gamma_R \, :=\, \sup_{L\ge R}\,\frac{1}{L}\bigg(1+\fint_{B_L}\Big|(\psi,\zeta)-\fint_{B_L}(\psi,\zeta)\Big|^2\bigg)^\frac12.
\end{equation}
Moreover, the following non-degeneracy property holds for all $E\in\Md_0$,
\begin{equation}\label{e.p1.3}
(1-C\gamma_R) |E|\,\lesssim\, \Big(\fint_{B_{R/2}}|\nabla\psi_E+E|^2\Big)^\frac12\, \lesssim\, (1+\gamma_R)|E|.
\qedhere
\end{equation}
\end{prop}
Although the proof of  Proposition~\ref{p1} follows the main steps as the proof of~\cite[Proposition~1]{GNO-reg}, it differs in two 
significant respects.
First, the natural two-scale expansion is not rigid inside the inclusions, which makes energy estimates more involved and requires some
local surgery. Second, and more importantly, a suitable control is needed on the pressure of the two-scale expansion error, which is made particularly subtle due to the crucial use of weighted norms.
Weighted pressure estimates are obtained based on the following weighted version of Bogovskii's standard construction; the proof is postponed to Section~\ref{sec:pr-p1+}.

\begin{lem}[Weighted Bogovskii construction]\label{lem:muck}
Given a domain $D\subset B_R$ that is star-shaped with respect to every point in $B_{R_0}$, for some $0<R_0 \le R$,
consider a weight $\mu\in C^\infty(\R^d;[0,1])$ such that $\mu^{-1}$ belongs to the Muckenhoupt class~$A_2$ and $\mu^{-1}\in\Ld^{d/2}(D)$ (or $\mu^{-1}\in\Ld^{s}(D)$ for some $s>1$ in the case $d=2$).
Then, for all $F\in\Ld^2(D)$ with $\int_D\mu F=0$, there exists $S\in H^1_0(D)^d$ such that
\begin{gather*}
\Div (S)\,=\,\mu F,\quad\text{in $D$},\\
\int_D\mu^{-1}|\nabla S|^2\,\lesssim\,\int_D\mu|F|^2,
\end{gather*}
where the multiplicative constant only depends on $d$, on $R/R_0$, on the $A_2$-norm of~$\mu^{-1}$, and on $\fint_D\mu^{-d/2}$ (or on $s>1$ and on $\fint_D\mu^{-s}$ in the case $d=2$).
\end{lem}

With Proposition~\ref{p1} at hand, we may now turn to the proof of Theorems~\ref{th:schauder}--\ref{th:ann-reg}, for which we heavily lean on~\cite{GNO-reg,DO1}.
First, following~\cite{GNO-reg}, we encapsulate a quantitative (averaged) control on the sublinear growth of the extended corrector by considering the minimal radius~$R$ such that $\gamma_R$ in~\eqref{e.p1.2} is small enough:
more precisely, given a constant~$C_0\ge1$ (to be fixed large enough),
we define the \emph{minimal radius} $r_*$ as the following random field,
\begin{equation}\label{e.def-r*} 
r_*(x)\,:=\,\inf\bigg\{R>0~:~ \frac1{\ell^2} \fint_{B_\ell(x)} \Big|(\psi,\zeta)-\fint_{B_\ell(x)} (\psi,\zeta)\Big|^2 \le \frac1{C_0},~~\forall \ell\ge R \bigg\}.
\end{equation}
Stationarity of~$r_*$ follows from stationarity of $(\nabla\psi,\nabla\zeta)$.
Almost sure finiteness of~$r_*$ follows from the sublinearity of $(\psi,\zeta)$ at infinity, cf.~Lemmas~\ref{lem:cor} and~\ref{lem:ext-cor}(iii). Under Assumption~\ref{Mix+}, moment bounds on $r_*$ are a direct consequence of corrector estimates of Theorem~\ref{th:cor} together with a union bound; we omit the details.

\medskip
Next, still following~\cite{GNO-reg}, in order to quantify the improvement of flatness for the solution of the steady Stokes problem, we consider the excess \eqref{eq:def-exc} of a trace-free $2$-tensor field $h$ on a ball~$D$, that is,
\begin{equation*}
\Exc(h;D)\,:=\,\inf_{E\in \Md_0 }\fint_{D} |h-(\nabla\psi_E+E)|^2,
\end{equation*}
which measures the deviation of $h$ from gradients of corrected coordinates.
In these terms, we establish the following consequence of Proposition~\ref{p1}, which quantifies the decay of the excess for solutions of the free steady Stokes problem~\eqref{eq:Stokes-zero}, proving a quantitative improvement of flatness on smaller balls.
The proof relies on Proposition~\ref{p1} together with a standard Campanato iteration;
in particular, since it is oblivious of the underlying PDE, we refer the reader to the proof of~\cite[Theorem~1]{GNO-reg} in the context of divergence-form linear elliptic equations, which applies without changing a iota.

\begin{theor}[Excess-decay estimate]\label{th:excess-decay}
Under Assumption~\ref{Hd}, for any H\"older exponent $\alpha \in (0,1)$, there exists a constant $C_\alpha\simeq_\alpha1$ such that the following holds: Let $r_*$ be defined in~\eqref{e.def-r*} with constant~$C_0$ replaced by $C_\alpha$.
For all $R\ge r_*(0)$, if $\nabla u$ is a solution of the free steady Stokes problem~\eqref{eq:Stokes-zero} in $B_R$,
then the following large-scale Lipschitz estimate holds for all $r_*(0)\le r\le R$,
\begin{equation}\label{e.exc-decay-3}
\fint_{B_r} |\nabla u|^2 \,\le\, C_\alpha \fint_{B_R}|\nabla u|^2,
\end{equation}
as well as the following large-scale $C^{1,\alpha}$ estimate for all $r_*(0)\le r\le R$,
\begin{equation*}
\Exc(\nabla u;B_r)\,\le\, C_\alpha (\tfrac r R)^{2\alpha} \Exc(\nabla u;B_R). 
\end{equation*}
In addition, the correctors enjoy the following non-degeneracy property for all $r\ge r_*(0)$ and $E\in\Md_0$,
\begin{equation*}
\frac1{C_\alpha} |E|^2 \,\le\, \fint_{B_r} |\nabla \psi_E+E|^2 \,\le\, C_\alpha |E|^2.\qedhere
\end{equation*}
\end{theor}

As a direct consequence, we may deduce a corresponding result for solutions of the steady Stokes problem~\eqref{eq:Stokes-zero} with a nontrivial right-hand side, cf.~\eqref{eq:test-v0}, as stated in~Theorem~\ref{th:schauder}.
The proof, which is identical to that of~\cite[Corollary~3]{GNO-reg}, is omitted as it only relies on Theorem~\ref{th:excess-decay} together with an energy estimate.

\medskip
Next, as a second consequence of the above, we may further deduce quenched large-scale $\Ld^p$ regularity estimates as stated  in Theorem~\ref{th:lp-reg}.
This is obtained for instance by combining the large-scale Lipschitz estimate~\eqref{e.exc-decay-3} together with Shen's dual Calder\'on--Zygmund lemma, cf.~\cite[Theorem~2.1]{Shen-07}, as done in~\cite[Section~6.1]{DO1} in the context of divergence-form linear elliptic equations. This proof further requires to replace the minimal radius $r_*$ in the above by the largest $\frac18$-Lipschitz lower bound $\underline r_*$, cf.~\cite[Section~3.7]{GNO-reg}; both satisfy the same boundedness properties and we use the same notation ``$r_*$'' in the statement.
Since this approach does not rely on the specific PDE at hand, the same proof applies without changing a iota and we do not reproduce it here.

\medskip
Finally, making a further use of Shen's dual Calder\'on--Zygmund lemma, cf.~\cite[Theorem~2.1]{Shen-07}, together with the quenched large-scale $\Ld^p$ regularity theory of Theorem~\ref{th:lp-reg} and  with  the large-scale Lipschitz estimate~\eqref{e.exc-decay-3}, the annealed regularity estimate of Theorem~\ref{th:ann-reg} easily follows as in~\cite{DO1} for $2\le q\le p<\infty$.
A duality argument yields the corresponding conclusion for $1<p\le q\le2$, and an interpolation argument allows to conclude for all $1<p,q<\infty$. The additional perturbative statement in Theorem~\ref{th:ann-reg} is already established in Theorem~\ref{th:meyers-ann}.

\subsection{Proof of Proposition~\ref{p1}}\label{sec:pr-p1}
Let $R\gg1$ be large enough and fixed. To ease notation, we assume without loss of generality $\fint_{B_R} (\psi,\zeta,\Sigma\mathds1_{\R^d\setminus\Ic})=0$.
Set $\Nc_R:=\{n:I_n^++\delta B\subset B_{R}\}$ and $\Nc_R^\circ:=\{n:(I_n^++\delta B)\cap \partial B_{R}\ne\varnothing\}$, where we recall that $I_n^+$ stands for the convex hull of $I_n$, and define
\[D_R:=\bigg(B_{R-\frac\delta2}\setminus\bigcup_{n\in\Nc_R^\circ}(I_n^++\delta B)\bigg)~+~\tfrac\delta2B.\]
In view of Assumption~\ref{Hd}, we note that
\begin{enumerate}[\quad$\bullet$]
\item $D_R$ is a $C^2$ domain (uniformly in $R$);
\item any inclusion that intersects $D_R$ is contained in $D_R$ and is at distance at least $\delta$ from~$\partial D_R$;
\item $B_{R-2-\delta}\subset D_R\subset B_R$.
\end{enumerate}
Given $4\le\rho\le\frac R4$ (the choice of which will be optimized later), we choose a smooth cut-off function $\eta_R\in C^\infty_c(\R^d;[0,1])$ such that $\eta_R=1$ in $B_{R-2\rho}$, $\eta_R=0$ outside $B_{R-\rho}$, and $|\nabla\eta_R|\lesssim\rho^{-1}$, and we further choose $\eta_R$ to be constant in the fattened inclusions $\{I_n+\frac\delta2B\}_{n\in\Nc_R}$.
Note in particular that $\eta_R$ is supported inside $D_R$.
We split the proof into five main steps.

\medskip
\step1 Two-scale expansion and representation of the error.\\
We split the proof into two further substeps.

\medskip\noindent
\substep{1.1} Construction of two-scale expansions.\\
Given a weak solution $(u,P)$ to~\eqref{eq:Stokes-zero},
let $(\hat u,\hat P)$ denote the unique weak solution of the following corresponding homogenized equation with Dirichlet data on $D_R$,
\begin{equation}\label{e.p1.4}
\quad\left\{\begin{array}{ll}
-\Div(2\Bb \D( \hat u))+\nabla\hat P=0 ,&\text{in $D_R$},\\
\Div(\hat u)=0,&\text{in $D_R$},\\
\hat u=u,&\text{on $\partial D_R$},
\end{array}\right.
\end{equation}
where we recall that the effective viscosity $\Bb$ is defined in~\eqref{eq:def-B}.
For definiteness, the pressures $P$ and $\hat P$ are chosen with $\int_{D_R} P\mathds1_{\R^d\setminus\Ic}=\int_{D_R}\hat P=0$.
Reformulating this homogenized equation as
\[-\Div(2\Bb \D( \hat u-u))+\nabla\hat P=\Div(2\Bb \D(u)),\qquad\text{in $D_R$},\]
testing with $\hat u-u\in H^1_0(D_R)^d$, and combining an energy estimate with the triangle inequality, we obtain
\begin{equation*}
\int_{D_R} |\!\D( \hat u)|^2 \,\lesssim\, \int_{D_R}|\!\D( u)|^2,
\end{equation*}
and, further using that $\Div(\hat u-u)=0$ implies $\int_{D_R}|\nabla(\hat u-u)|^2=2\int_{D_R}|\!\D(\hat u-u)|^2$,
\begin{equation}\label{e.p1.12+}
\int_{D_R} |\nabla\hat u|^2 \,\lesssim\, \int_{D_R}|\nabla u|^2.
\end{equation}
We now compare $u$ and $P$ to their respective two-scale expansions,
\[u\,\leadsto\,\hat u + \eta_R \psi_E \partial_E \hat u,\qquad P\,\leadsto\,\hat P+\eta_R\bb:\D(\hat u)+\eta_R \Sigma_E\mathds1_{\R^d\setminus\Ic} \partial_E \hat u,\]
where we use Einstein's convention of implicit summation on repeated indices and where the index $E$ runs here over an orthonormal basis $\Ec$ of $\Md_0^\Sym$.
Recall that the pressure~$P$ is only defined up to a global arbitrary constant on ${\R^d\setminus\Ic}$,
so that we may choose an arbitrary constant $P_*\in\R$ and consider the pressure $P'=P+P_*$ on ${\R^d\setminus\Ic}$.
In addition we choose arbitrary constants $\{P_n\}_n\subset\R$ and extend the pressure inside the inclusions by setting $P'|_{I_n}=P_n$. We thus define in the whole domain $D_R$,
\begin{equation}\label{eq:cst-press}
P'\,:=\,(P+P_*)\mathds1_{{\R^d\setminus\Ic}}+\sum_{n\in\Nc_R}P_n \mathds 1_{I_n},
\end{equation}
where the constants $P_*$ and $\{P_n\}_n$ will be suitably chosen later.
We then consider the following two-scale expansion errors in $D_R$,
\begin{equation}\label{eq:def-wQ}
w\,:=\,u-\hat u - \eta_R \psi_E \partial_E \hat u,\qquad
Q\,:=\,P'-\hat P-\eta_R\bb:\D(\hat u)-\eta_R \Sigma_E\mathds1_{\R^d\setminus\Ic} \partial_E \hat u.
\end{equation}

\medskip
\substep{1.2}
Proof that $(w,Q)$ satisfies  in the weak sense in $D_R$
\begin{align}\label{e.p1.5}
&-\triangle w+\nabla Q\,=\,
-\sum_{n\in\Nc_R}\delta_{\partial I_n}\sigma(u,P+P_*-P_n)\nu-\Div\big((\eta_R\partial_E\hat u)J_E\mathds1_\Ic\big)\\
&\quad+\Div\Big(2(1-\eta_R)(\Id-\Bb)\D(\hat u)+(2\psi_E\otimes_s-\zeta_E)\nabla(\eta_R\partial_E\hat u)-\Id(\psi_{E}\cdot\nabla)(\eta_R\partial_E\hat u)\Big).\nonumber
\end{align}
By definition of $w,Q$, expanding the gradient and reorganizing the terms, we find
\begin{multline*}
-\triangle w+\nabla Q\,=\,-\triangle u+\nabla P'+\triangle \hat u-\nabla\hat P-\nabla\big(\eta_R\bb:\D(\hat u)\big)
+\Div\big(\psi_E\otimes\nabla(\eta_R\partial_E\hat u)\big)\\
+(\eta_R\partial_E\hat u)\,\Div\big(\nabla\psi_E-\Sigma_E\mathds1_{\R^d\setminus\Ic}\Id\big)+\big(\nabla\psi_E-\Sigma_E\mathds1_{\R^d\setminus\Ic}\Id\big)\nabla(\eta_R\partial_E\hat u).
\end{multline*}
Further using that $\Div(\psi_E)=0$, and using Leibniz' rule, this can be rewritten as
\begin{multline*}
-\triangle w+\nabla Q\,=\,-\triangle u+\nabla P'+\triangle \hat u-\nabla\hat P-\nabla\big(\eta_R\bb:\D(\hat u)\big)\\
+\Div\big(2\psi_E\otimes_s\nabla(\eta_R\partial_E\hat u)\big)-\nabla\big(\psi_{E}\cdot\nabla(\eta_R\partial_E\hat u)\big)\\
+\Div\Big((\eta_R\partial_E\hat u)\big(2\D(\psi_E)-\Sigma_E\mathds1_{\R^d\setminus\Ic}\Id\big)\Big).
\end{multline*}
Since $\Div(\hat u)=0$, we may decompose
\[\triangle\hat u=\Div(2\D(\hat u))=\Div\big(2(1-\eta_R)\D(\hat u)\big)+\Div\big((\eta_R\partial_E\hat u)2E\big).\]
Inserting this into the above, and writing $2(\D(\psi_E)+E)-\Sigma_E\mathds1_{\R^d\setminus\Ic}=J_E\mathds1_{\R^d\setminus\Ic}$ in terms of the extended flux $J_E$, cf.~Lemma~\ref{lem:ext-cor}, we obtain
\begin{multline}\label{eq:pre-wQ-eqn}
-\triangle w+\nabla Q\,=\,-\triangle u+\nabla P'+\Div\big(2(1-\eta_R)\D(\hat u)\big)-\nabla\hat P-\nabla\big(\eta_R\bb:\D(\hat u)\big)\\
+\Div\big(2\psi_E\otimes_s\nabla(\eta_R\partial_E\hat u)\big)-\nabla\big(\psi_{E}\cdot\nabla(\eta_R\partial_E\hat u)\big)
+\Div\big((\eta_R\partial_E\hat u)J_E\mathds1_{\R^d\setminus\Ic}\big).
\end{multline}
Since $\Div(J_E)=0$, we have
\[\Div\big((\eta_R\partial_E\hat u)J_E\mathds1_{\R^d\setminus\Ic}\big)\,=\,J_E\nabla(\eta_R\partial_E\hat u)-\Div\big((\eta_E\partial_E\hat u)J_E\mathds1_\Ic\big),\]
and thus, further recalling $\expec{J_E}=2\Bb E+(\bb:E)\Id$, writing $J_E-\expec{J_E}=\Div(\zeta_E)$, and using the skew-symmetry of $\zeta_E$, cf.~Lemma~\ref{lem:ext-cor}, we find
\begin{multline*}
\Div\big((\eta_R\partial_E\hat u)J_E\mathds1_{\R^d\setminus\Ic}\big)\,=\,\Div(2\eta_R\Bb\D(\hat u))+\nabla(\eta_R\bb:\D(\hat u))\\
-\Div(\zeta_E\nabla(\eta_R\partial_E\hat u))-\Div\big((\eta_R\partial_E\hat u)J_E\mathds1_\Ic\big).
\end{multline*}
Inserting this into~\eqref{eq:pre-wQ-eqn}, and recalling that equation~\eqref{e.p1.4} yields $-\Div(2\Bb\D(\hat u))+\nabla\hat P=0$, we deduce
\begin{multline*}
-\triangle w+\nabla Q\,=\,-\triangle u+\nabla P'+\Div\big(2(1-\eta_R)(\Id-\Bb)\D(\hat u)\big)\\
+\Div\big((2\psi_E\otimes_s-\zeta_E)\nabla(\eta_R\partial_E\hat u)\big)-\nabla\big(\psi_{E}\cdot\nabla(\eta_R\partial_E\hat u)\big)
-\Div\big((\eta_R\partial_E\hat u)J_E\mathds1_\Ic\big).
\end{multline*}
Finally, since equation~\eqref{e:up-whole0} for $(u,P)$ implies of $(u,P')$ on $D_R$
\[-\triangle u+\nabla P'=-\sum_{n\in\Nc_R}\delta_{\partial I_n}\sigma(u,P+P_*-P_n)\nu,\]
 the claim~\eqref{e.p1.5} follows.

\medskip
\step2 Weighted energy estimate for the two-scale expansion error:
considering the following weight function as in~\cite{GNO-reg},
\begin{equation}\label{eq:def-muReps}
\mu_{R,\e}:~B_R \to [0,1]:~ x\mapsto(1-\tfrac{|x|}{R})^{\frac\e2},
\end{equation}
we prove, for all $K\gg1$ and $\e\ll K^{-1/2}$,
\begin{multline}\label{e.p1.8}
\int_{D_R} \mu_{R,\e}^2|\nabla w|^2
\,\lesssim\,\frac1K\int_{D_R}\mu_{R,\e}^2Q^2+K\int_{D_R}(1-\eta_R)^2\mu_{R,\e}^2|\nabla\hat u|^2\\
+K\Big(\sup_{D_R}|\nabla(\eta_R\nabla\hat u)|^2\Big)\int_{D_R}(1+|(\psi,\zeta,\nabla\psi,\Sigma\mathds1_{\R^d\setminus\Ic})|^2).
\end{multline}
The main difficulty is that neither $\hat u$ nor~$\mu_{R,\e}$ is constant inside the inclusions, which prohibits us from easily taking advantage of the boundary conditions for~$u$ and~$\psi_E$ in the estimate.
To circumvent this issue, we use the following truncation maps $T_0,T_1$: for all $g\in C^\infty_b(D_R)$,
\begin{eqnarray}\label{eq:def-T0T1}
T_0[g](x)&:=&(1-\chi(x))g(x)+\sum_{n\in\Nc_R}\chi_n(x)\Big(\fint_{I_n+\frac\delta2B}g\Big),\\
T_1[g](x)&:=&(1-\chi(x))g(x)+\sum_{n\in\Nc_R}\chi_n(x)\bigg(\Big(\fint_{I_n+\frac\delta2B}g\Big)+\Big(\fint_{I_n+\frac\delta2B}\nabla g\Big)(x-x_n)\bigg),\nonumber
\end{eqnarray}
where for all $n$ we have chosen a cut-off function $\chi_n\in C^\infty_c(\R^d;[0,1])$ with
\[\chi_n|_{I_n+\frac\delta4B}= 1, \qquad\chi_n|_{\R^d\setminus(I_n+\frac\delta2 B)}=0,\qquad|\nabla \chi_n|+|\nabla^2\chi_n|\lesssim 1,\]
and where we have set for abbreviation $\chi:=\sum_{n\in\Nc_R}\chi_n$.
In these terms,
we consider the following modification of the weight $\mu_{R,\e}$ and of the two-scale expansion error $(w,Q)$,
\begin{eqnarray}
\tilde \mu_{R,\e}&:=&T_0[\mu_{R,\e}],\nonumber\\
\tilde w&:=&u-T_1[\hat u]-\eta_R\psi_ET_0[\partial_E\hat u],\nonumber\\
\tilde Q&:=&P'-T_0[\hat P]-\eta_R\bb:T_0[\D(\hat u)]-\eta_R\Sigma_E\mathds1_{\R^d\setminus\Ic}T_0[\partial_E\hat u].\label{eq:trunc-wQ}
\end{eqnarray}
Note that $T_1[\hat u]=\hat u=u$ on $\partial D_R$, and thus $\tilde w\in H^1_0(D_R)^d$.
Testing equation~\eqref{e.p1.5} for $w$ with the test function $\tilde\mu_{R,\e}^2\tilde w\in H^1_0(D_R)^d$, we find
\begin{equation}\label{eq:J0123}
J_0\,=\,J_1+J_2+J_3,
\end{equation}
in terms of
\begin{eqnarray*}
J_0 &:=&\int_{D_R}\nabla(\tilde\mu_{R,\e}^2\tilde w):(\nabla w- Q\Id),\\
J_1&:=&-\sum_{n\in\Nc_R}\int_{\partial I_n}\tilde\mu_{R,\e}^2\tilde w\cdot\sigma(u,P+P_*-P_n)\nu+\sum_{n\in\Nc_R}\int_{I_n}(\eta_R\partial_E\hat u)\,\nabla(\tilde\mu_{R,\e}^2\tilde w):J_E,\\
J_2&:=&-2\int_{D_R}(1-\eta_R)\nabla(\tilde\mu_{R,\e}^2\tilde w):(\Id-\Bb)\D(\hat u),\\
J_3&:=&-\int_{D_R}\nabla(\tilde\mu_{R,\e}^2\tilde w):\Big((2\psi_E\otimes_s-\zeta_E)\nabla(\eta_R\partial_E\hat u)-\Id(\psi_E\cdot\nabla)(\eta_R\partial_E\hat u)\Big).
\end{eqnarray*}
It remains to estimate these  terms, and we split the proof of~\eqref{e.p1.8} into four further substeps.

\medskip
\substep{2.1} Lower bound on $J_0$: for all $K\gg1$ and $0<\e\ll K^{-1/2}$,
\begin{multline}\label{e.p1.6}
J_0\,\ge\, \frac12 \int_{D_R} \tilde \mu_{R,\e}^2|\nabla w|^2\\
-\frac1K \int_{D_R} \tilde \mu_{R,\e}^2 Q^2
-K\int_{D_R} \tilde \mu_{R,\e}^2|\nabla(w-\tilde w)|^2
-K\int_{D_R} \tilde \mu_{R,\e}^2\,\Div(\tilde w)^2.
\end{multline}
Expanding the gradient in the definition of $J_0$ yields
\begin{multline*}
J_0\,=\,\int_{D_R} \tilde \mu_{R,\e}^2 \nabla \tilde w:\nabla w+\int_{D_R}2\tilde \mu_{R,\e}(\tilde w\otimes \nabla \tilde \mu_{R,\e}) : \nabla w\\
-\int_{D_R}\Big( \tilde \mu_{R,\e}^2\Div (\tilde w)+2\tilde \mu_{R,\e}\tilde w\cdot \nabla \tilde \mu_{R,\e}\Big)Q.
\end{multline*}
Adding and subtracting $\nabla w$ to $\nabla \tilde w$, we deduce by Young's inequality, for all $K\ge1$,
\begin{multline}\label{e.p1.6-00}
J_0
\,\ge\, \Big(1-\frac{1}{K}\Big)\int_{D_R} \tilde \mu_{R,\e}^2|\nabla  w|^2
-\frac1{K}\int_{D_R}\tilde \mu_{R,\e}^2Q^2\\
-4K \int_{D_R} |\nabla \tilde \mu_{R,\e}|^2|\tilde w|^2
-\frac K2\int_{D_R} \tilde \mu_{R,\e}^2|\nabla(w-\tilde w)|^2
-\frac K2\int_{D_R} \tilde \mu_{R,\e}^2\,\Div(\tilde w)^2.
\end{multline}
Since $\tilde \mu_{R,\e}$ satisfies for all $x\in B_{R}$,
$$
\tilde \mu_{R,\e}(x) \simeq \mu_{R,\e}(x), \qquad |\nabla \tilde \mu_{R,\e}(x)| \lesssim |\nabla \mu_{R,\e}(x)| \simeq \tfrac{\e}{R} \big(1-\tfrac{|x|}{R}\big)^{\frac\e2-1},
$$
the following estimate follows from Hardy's inequality in form of e.g.~\cite[Estimate~(88)]{GNO-reg}: given $0<\e\le\frac12$, there holds for all $g\in H^1_0(B_{R})$,
\begin{equation}\label{hardy}
\int_{B_{R}} |\nabla \tilde \mu_{R,\e}|^2 |g|^2 \,\lesssim\, {\e^2} \int_{B_{R}}\tilde \mu_{R,\e}^2 |\nabla g|^2.
\end{equation}
Extending $\tilde w$ by $0$ outside $D_R$ and applying this inequality, we find
\[\int_{D_R} |\nabla \tilde \mu_{R,\e}|^2|\tilde w|^2~\lesssim~ \e^2\int_{D_R}\tilde\mu_{R,\e}^2|\nabla\tilde w|^2.\]
Inserting this into~\eqref{e.p1.6-00}, the claim~\eqref{e.p1.6} follows for $K\ge3$ and $K\e^2\ll1$.

\medskip
\substep{2.2} Upper bound on $J_1$: for all $K\ge1$,
\begin{multline}\label{eq:bnd-J1}
|J_1|\,\lesssim\,
\frac1K\int_{D_R}\tilde\mu_{R,\e}^2|\nabla\tilde w|^2
+\frac1K\int_{D_R}\tilde\mu_{R,\e}^2|\tilde Q|^2\\
+K\Big(\sup_{D_R}|\nabla(\eta_R\nabla\hat u)|^2\Big)\int_{D_R}(1+|(\nabla\psi,\Sigma\mathds1_{\R^d\setminus\Ic})|^2)
+K\int_{D_R}(1-\eta_R)^2\tilde\mu_{R,\e}^2|\nabla\hat u|^2.
\end{multline}
We examine separately the two terms in the definition of $J_1=J_{1,1}+J_{1,2}$,
\begin{eqnarray*}
J_{1,1}&:=&-\sum_{n\in\Nc_R}\int_{\partial I_n}\tilde\mu_{R,\e}^2\tilde w\cdot\sigma(u,P+P_*-P_n)\nu,\\
J_{1,2}&:=&\sum_{n\in\Nc_R}\int_{I_n}(\eta_R\partial_E\hat u)\,\D(\tilde\mu_{R,\e}^2\tilde w):J_E,
\end{eqnarray*}
and we start with $J_{1,1}$.
Since $\tilde \mu_{R,\e}$ and $\eta_R$ are constant in the inclusions, and since for all $n\in\Nc_R$ we have
\begin{equation}\label{eq:Dtildew-incl}
\D(\tilde w)=-(1-\eta_R)\Big(\fint_{I_n+\frac\delta2B}\D(\hat u)\Big),\qquad\text{in $I_n$},
\end{equation}
we may use the boundary conditions for $u$ to the effect of
\begin{equation*}
J_{1,1}\,=\,\sum_{n\in\Nc_R}\big((1-\eta_R)\tilde\mu_{R,\e}^2\big)(x_n)\Big(\fint_{I_n+\frac\delta2B}\D(\hat u)\Big):\int_{\partial I_n}\sigma(u,P+P_*-P_n)\nu\otimes(x-x_n).
\end{equation*}
Using Stokes' formula in the form $\int_{\partial I_n}\nu\otimes(x-x_n)=|I_n|\Id$, together with the constraint $\Div(\hat u)=0$ that we use in the form
$(\fint_{I_n+\frac \delta 2 B} \D(\hat u)):\Id=0$, we can subtract any constant to the pressure in the above expression, so that in particular
\begin{multline}\label{eq:formJ11}
J_{1,1}\,=\,\sum_{n\in\Nc_R}\big((1-\eta_R)\tilde\mu_{R,\e}^2\big)(x_n)\Big(\fint_{I_n+\frac\delta2B}\D(\hat u)\Big)\\
:\int_{\partial I_n}\sigma\big(u,P+P_*-T_0[\hat P]-\eta_R\bb:T_0[\D(\hat u)]\big)\nu\otimes(x-x_n).
\end{multline}
We turn to $J_{1,2}$.
Decomposing $\partial_E\hat u=(\partial_E\hat u-T_0[\partial_E\hat u])+T_0[\partial_E\hat u]$, using that~$T_0[\partial_E\hat u]$, $\tilde \mu_{R,\e}$, and $\eta_R$ are constant in the inclusions, that $\tilde w$ is affine in the inclusions, and using~\eqref{eq:Dtildew-incl} again,
we find
\begin{multline*}
J_{1,2}\,=\,\sum_{n\in\Nc_R}\int_{I_n}\eta_R\tilde\mu_{R,\e}^2\big(\partial_E\hat u-T_0[\partial_E\hat u]\big)\D(\tilde w):J_E\\
-\sum_{n\in\Nc_R}\big((1-\eta_R)\eta_R\tilde\mu_{R,\e}^2T_0[\partial_E\hat u]\big)(x_n)\Big(\fint_{I_n+\frac\delta2B}\D(\hat u)\Big):\int_{I_n}J_E.
\end{multline*}
Writing $J_E|_{I_n}=\sigma(\psi_E^n,\Sigma_E^n)$ with $(\psi_E^n,\Sigma_E^n)$ defined in~\eqref{eq:sol-psiEn}, cf.~\eqref{eq:def-J}, using Stokes' formula,
and recalling that $\sigma(\psi_E^n,\Sigma_E^n)\nu=\sigma(\psi_E+Ex,\Sigma_E)\nu$ on $\partial I_n$, cf.~\eqref{eq:sol-psiEn},
we deduce
\begin{multline*}
J_{1,2}\,=\,\sum_{n\in\Nc_R}\int_{I_n}\eta_R\tilde\mu_{R,\e}^2\big(\partial_E\hat u-T_0[\partial_E\hat u]\big)\D(\tilde w):\sigma(\psi_E^n,\Sigma_E^n)\\
-\sum_{n\in\Nc_R}\big((1-\eta_R)\eta_R\tilde\mu_{R,\e}^2T_0[\partial_E\hat u]\big)(x_n)\Big(\fint_{I_n+\frac\delta2B}\D(\hat u)\Big):\int_{\partial I_n}\sigma(\psi_E+Ex,\Sigma_E)\nu\otimes(x-x_n).
\end{multline*}
Combining this with~\eqref{eq:formJ11}, and reorganizing the terms, we obtain
\begin{equation*}
J_1\,=\,J'_{1,1}+J'_{1,2},
\end{equation*}
in terms of
\begin{eqnarray*}
J'_{1,1}&=&\sum_{n\in\Nc_R}\int_{I_n}\eta_R\tilde\mu_{R,\e}^2\big(\partial_E\hat u-T_0[\partial_E\hat u]\big)\D(\tilde w):\sigma(\psi_E^n,\Sigma_E^n),\\
J'_{1,2}&=&\sum_{n\in\Nc_R}\big((1-\eta_R)\tilde\mu_{R,\e}^2\big)(x_n)\Big(\fint_{I_n+\frac\delta2B}\D(\hat u)\Big)\\
&&\hspace{2cm}:\int_{\partial I_n}\Big(\sigma\big(u,P+P_*-T_0[\hat P]-\eta_R\bb:T_0[\D(\hat u)]\big)\\
&&\hspace{5 cm}-\eta_R T_0[\partial_E\hat u]\sigma(\psi_E+Ex,\Sigma_E)\Big)\nu\otimes(x-x_n).
\end{eqnarray*}
We separately estimate $J_{1,1}'$ and $J_{1,2}'$, and we start with the former.
Using~\eqref{eq:bnd-psinSign} and noting that $|\nabla\hat u-T_0[\nabla\hat u]|=|\nabla\hat u-\fint_{I_n+\frac\delta2B}\nabla\hat u|\lesssim\sup_{I_n+\frac\delta2B}|\nabla^2\hat u|$ on $I_n$ and that $\eta_R$ is constant in $I_n$, we find
\begin{equation}\label{eq:bnd-L1}
|J'_{1,1}|\,\lesssim\,\Big(\sup_{D_R}|\nabla(\eta_R\nabla\hat u)|\Big)\Big(\int_{D_R}\tilde\mu_{R,\e}^2|\nabla\tilde w|^2\Big)^\frac12\Big(\int_{D_R}(1+|(\nabla\psi,\Sigma\mathds1_{\R^d\setminus\Ic})|^2)\Big)^\frac12.
\end{equation}
We turn to $J'_{1,2}$. Writing for abbreviation
\[H:=\sigma\big(u,P'-T_0[\hat P]-\eta_R\bb:T_0[\D(\hat u)]\big)-\eta_R T_0[\partial_E\hat u]\sigma(\psi_E+Ex,\Sigma_E\big),\]
and noting that $\Div(H)=0$ in $(I_n+\frac\delta4B)\setminus I_n$, $\int_{\partial I_n}H\nu=0$, and $\int_{\partial I_n}\Theta(x-x_n)\cdot H\nu=0$ for all $n\in\Nc_R$ and $\Theta\in\Md^\Skew$, the trace estimate~\eqref{eq:trace-estim} leads to
\begin{equation}\label{eq:trace-est}
|J'_{1,2}|\,\lesssim\,\sum_{n\in\Nc_R}\big((1-\eta_R)\tilde\mu_{R,\e}^2\big)(x_n)\Big(\fint_{I_n+\frac\delta2B}|\!\D(\hat u)|^2\Big)^\frac12\Big(\int_{(I_n+\frac\delta4B)\setminus I_n}|H|^2\Big)^\frac12.
\end{equation}
For all $n\in\Nc_R$, we can write in the annulus $(I_n+\frac\delta4B)\setminus I_n$ (where $P'=P+P_*$), recalling the definition~\eqref{eq:trunc-wQ} of the modified two-scale expansion error~$(\tilde w,\tilde Q)$ and the definition of truncations,
\begin{eqnarray*}
H&=&2\D\!\big(u-\eta_RT_1[\hat u]-\eta_R\psi_ET_0[\partial_E\hat u]\big)\\
&&\hspace{2cm}-\big(P'-T_0[\hat P]-\eta_R\bb:T_0[\D(\hat u)]-\eta_R\Sigma_ET_0[\partial_E\hat u]\big)\Id\\
&=&\sigma(\tilde w,\tilde Q)+2(1-\eta_R)T_0[\D(\hat u)].
\end{eqnarray*}
Inserting this into~\eqref{eq:trace-est}, using that $\sup_{B(x)}\tilde\mu_{R,\e}\simeq\inf_{B(x)}\tilde\mu_{R,\e}$ holds for all $x\in D_R$, and using that $\eta_R$ is constant in fattened inclusions,
we deduce
\begin{multline*}
|J'_{1,2}|\,\lesssim\,\Big(\int_{D_R}(1-\eta_R)^2\tilde\mu_{R,\e}^2|\!\D(\hat u)|^2\Big)^\frac12\\
\times\Big(\int_{D_R}\tilde\mu_{R,\e}^2(|\!\D(\tilde w)|^2+|\tilde Q|^2)+(1-\eta_R)^2\tilde\mu_{R,\e}^2|\!\D(\hat u)|^2\Big)^\frac12.
\end{multline*}
Combined with the bound~\eqref{eq:bnd-L1} on $J'_{1,1}$, the claim~\eqref{eq:bnd-J1} follows by Young's inequality.

\medskip
\substep{2.3}
Upper bound on $J_2,J_3$: for all $K\ge1$,
\begin{multline}\label{e.p1.7}
|J_2|+|J_3|
\,\lesssim\,
\frac1K\int_{D_R}\tilde\mu_{R,\e}^2|\nabla\tilde w|^2
+K\int_{D_R}(1-\eta_R)^2\tilde\mu_{R,\e}^2|\nabla\hat u|^2\\
+K\Big(\sup_{D_R}|\nabla(\eta_R\nabla\hat u)|^2\Big)\int_{D_R}|(\psi,\zeta)|^2.
\end{multline}
Expanding the gradients and using Young's inequality, we find for all $K\ge1$,
\begin{eqnarray*}
|J_2|&\lesssim&K\int_{D_R}(1-\eta_R)^2\tilde\mu_{R,\e}^2|\nabla\hat u|^2+\frac1K\int_{D_R}\tilde\mu_{R,\e}^2|\nabla\tilde w|^2+\frac1K\int_{D_R}|\nabla\tilde\mu_{R,\e}|^2|\tilde w|^2,\\
|J_3|&\lesssim&K\int_{D_R}\tilde\mu_{R,\e}^2|(\psi,\zeta)|^2|\nabla(\eta_R\nabla\hat u)|^2+\frac1K\int_{D_R}\tilde\mu_{R,\e}^2|\nabla\tilde w|^2+\frac1K\int_{D_R}|\nabla\tilde\mu_{R,\e}|^2|\tilde w|^2,
\end{eqnarray*}
and Hardy's inequality~\eqref{hardy} yields the claim~\eqref{e.p1.7}.

\medskip
\substep{2.4}
Control of truncation errors:
\begin{eqnarray}
\hspace{-1cm}\int_{D_R} \tilde \mu_{R,\e}^2 |\nabla (w-\tilde w)|^2 &\lesssim& \int_{D_R}(1-\eta_R)^2 \tilde \mu_{R,\e}^2 |\nabla\hat u|^2\nonumber\\
&&\hspace{1.5cm}+\Big(\sup_{D_R}|\nabla(\eta_R \nabla\hat u)|^2\Big)\int_{D_R}\big(1+ |(\psi,\nabla \psi)|^2\big),\label{eq:est-w-bar}\\
\hspace{-1cm}\int_{D_R} \tilde \mu_{R,\e}^2 (Q-\tilde Q)^2&\lesssim&\int_{D_R} (1-\eta_R)^2\tilde \mu_{R,\e}^2|\nabla\hat u|^2\nonumber\\
&&\hspace{1.5cm}+\Big(\sup_{D_R}|\nabla(\eta_R\nabla \hat u)|^2\Big)\int_{D_R}\big(1+\Sigma^2\mathds1_{\R^d\setminus\Ic}\big).\label{eq:est-Q-bar}
\end{eqnarray}
We start with the proof of~\eqref{eq:est-w-bar}.
The definition~\eqref{eq:trunc-wQ} of $\tilde w$ yields
\begin{multline*}
\nabla (w-\tilde w)\,=\,-\nabla (\hat u- T_1[\hat u])-\eta_R(\partial_E\hat u-T_0[\partial_E\hat u]) \nabla\psi_E
- \psi_E \otimes \nabla \big(\eta_R  (\partial_E\hat u-T_0[\partial_E\hat u])\big),
\end{multline*}
and thus
\begin{multline}\label{e.p1.11+0}
\int_{D_R} \tilde \mu_{R,\e}^2 |\nabla (w-\tilde w)|^2 \,\lesssim\, \int_{D_R} \tilde \mu_{R,\e}^2 |\nabla (\hat u-T_1[\hat u])|^2\\
+\Big(\sup_{D_R}|\eta_R(\nabla\hat u-T_0[\nabla\hat u])|^2+\sup_{D_R}|\nabla (\eta_R (\nabla\hat u-T_0[\nabla\hat u]))|^2\Big)\int_{D_R} |(\psi,\nabla \psi)|^2.
\end{multline}
The definition~\eqref{eq:def-T0T1} of the truncation maps $T_0,T_1$ gives
\begin{eqnarray*}
\nabla\hat u-T_0[\nabla\hat u]&=&\sum_{n\in\Nc_R}\chi_n\Big(\nabla\hat u-\fint_{I_n+\frac\delta2B}\nabla\hat u\Big)\\
\nabla(\hat u-T_1[\hat u])&=&\sum_{n\in\Nc_R}\chi_n\Big(\nabla \hat u-\fint_{I_n+\frac\delta2B}\nabla\hat u\Big)\\
&&+\sum_{n\in\Nc_R}\nabla\chi_n\bigg(\hat u-\Big(\fint_{I_n+\frac\delta2B}\hat u\Big)-\Big(\fint_{I_n+\frac\delta2B}\nabla\hat u\Big)(x-x_n)\bigg).
\end{eqnarray*}
Using the properties of $\tilde\mu_{R,\e}$, $\eta_R$, and of the cut-off functions $\{\chi_n\}_n$, and appealing to Poincaré's inequality on the fattened inclusions (on which we recall that $\eta_R$ is constant), we find
\begin{eqnarray}
\lefteqn{\int_{D_R} \tilde \mu_{R,\e}^2|\nabla (\hat u-T_1[\hat u])|^2}\nonumber\\
&\lesssim&\sum_{n\in\Nc_R}\Big(\sup_{I_n+\frac\delta2B}\tilde\mu_{R,\e}^2\Big)\int_{I_n+\frac\delta2B}\Big( \eta_R^2 |\nabla (\hat u-T_1[\hat u])|^2+   (1-\eta_R)^2 |\nabla (\hat u-T_1[\hat u])|^2\Big)\nonumber\\
&\lesssim&\sum_{n\in\Nc_R}\Big(\sup_{I_n+\frac\delta2B}\tilde\mu_{R,\e}^2\Big)\int_{I_n+\frac\delta2B}\Big( \eta_R^2 |\nabla^2\hat u|^2+   (1-\eta_R)^2 |\nabla\hat u|^2\Big)\nonumber\\
&\lesssim& \int_{D_R} |\nabla(\eta_R\nabla \hat u)|^2+ \int_{D_R} (1-\eta_R)^2\tilde \mu_{R,\e}^2 |\nabla \hat u|^2,\label{eq:bnd-uhatT1}
\end{eqnarray}
and similarly,
\begin{multline}\label{eq:bnd-ubarutil}
\sup_{D_R}|\eta_R(\nabla\hat u-T_0[\nabla\hat u])|+\sup_{D_R}|\nabla (\eta_R (\nabla\hat u-T_0[\nabla\hat u]))|\\
\,\lesssim\,\sup_{n\in\Nc_R}\Big(\eta_R(x_n)\sup_{I_n+\frac\delta2B}|\nabla^2\hat u|\Big)
\,\lesssim\,\sup_{D_R}|\nabla(\eta_R\nabla \hat u)|.
\end{multline}
Inserting these bounds into~\eqref{e.p1.11+0}, the claim~\eqref{eq:est-w-bar} follows.

\medskip\noindent
We turn to the proof of~\eqref{eq:est-Q-bar}.
The definition~\eqref{eq:trunc-wQ} of $\tilde Q$ yields
\begin{equation*}
Q-\tilde Q\,=\,-(\hat P-T_0[\hat P])-\eta_R\bb:(\D(\hat u)- T_0[\D(\hat u)])-\eta_R\Sigma_E\mathds1_{\R^d\setminus\Ic}(\partial_E\hat u-T_0[\partial_E\hat u]),
\end{equation*}
and thus
\begin{multline}\label{e.p1.11+}
\int_{D_R}\tilde\mu_{R,\e}^2(Q-\tilde Q)^2\,\lesssim\,\int_{D_R}\tilde\mu_{R,\e}^2(\hat P-T_0[\hat P])^2\\
+\Big(\sup_{D_R}|\eta_R(\nabla\hat u- T_0[\nabla\hat u])|^2\Big)\int_{D_R}(1+\Sigma^2\mathds1_{\R^d\setminus\Ic}).
\end{multline}
We start by analyzing the first right-hand side term. By definition of $T_0$, using the properties of $\tilde\mu_{R,\e}$ and appealing to Poincaré's inequality on the fattened inclusions (on which we recall that $\eta_R$ is constant), we find
\begin{multline}\label{eq:bnd-hatPdiffT0}
\int_{D_R}\tilde\mu_{R,\e}^2(\hat P-T_0[\hat P])^2\,\lesssim\,\sum_{n\in\Nc_R}\Big(\sup_{I_n+\frac\delta2B}\tilde\mu_{R,\e}^2\Big)\int_{I_n+\frac\delta2B}\Big(\hat P-\fint_{I_n+\frac\delta2B}\hat P\Big)^2\\
\,\lesssim\,\sum_{n\in\Nc_R}\Big(\sup_{I_n+\frac\delta2B}\tilde\mu_{R,\e}^2\Big)\int_{I_n+\frac\delta2B}\bigg(\eta_R^2|\nabla\hat P|^2+(1-\eta_R)^2\Big(\hat P-\fint_{I_n+\frac\delta2B}\hat P\Big)^2\bigg).
\end{multline}
We now appeal to a classical pressure estimates on $\hat P$.
On the one hand, since $(\hat u,\hat P)$ satisfies a steady Stokes equation~\eqref{e.p1.4} without forcing in $D_R$, a direct use of the Bogovskii operator in form of e.g.~\cite[Theorem III.3.1]{Galdi} yields for all $n\in\Nc_R$,
\begin{equation}\label{eq:bnd-hatPdiff}
\int_{I_n+\frac\delta2B}\Big(\hat P-\fint_{I_n+\frac\delta2B} \hat P\Big)^2\,\lesssim\, \int_{I_n+\frac\delta2B}|\nabla \hat u|^2.
\end{equation}
On the other hand, since $(\partial_i\hat u,\partial_i\hat P)$ satisfies the same equation in $D_R$, the same argument yields
\begin{equation*}
\int_{I_n+\frac\delta2B}\Big|\nabla\hat P-\fint_{I_n+\frac\delta2B}\nabla \hat P\Big|^2\,\lesssim\, \int_{I_n+\frac\delta2B}|\nabla^2 \hat u|^2,
\end{equation*}
Further noting that equation~\eqref{e.p1.4} yields
\begin{eqnarray*}
\int_{I_n+\frac\delta2B}\nabla \hat P&=&\int_{I_n+\frac\delta2B}\Div(2\Bb\D(\hat u))\\
&=&2\int_{\partial(I_n+\frac\delta2B)}(\Bb\D(\hat u))\nu\\
&=&2\int_{\partial(I_n+\frac\delta2B)}\Big(\Bb\D(\hat u)-\fint_{I_n+\frac\delta2B}\Bb\D(\hat u)\Big)\nu,
\end{eqnarray*}
and thus
\[\Big|\int_{I_n+\frac\delta2B}\nabla \hat P\Big|\,\lesssim\,\sup_{I_n+\frac\delta2B}|\nabla^2\hat u|,\]
we deduce
\begin{equation*}
\int_{I_n+\frac\delta2B}|\nabla\hat P|^2\,\lesssim\,\sup_{I_n+\frac\delta2B}|\nabla^2\hat u|^2.
\end{equation*}
Inserting this together with~\eqref{eq:bnd-hatPdiff} into~\eqref{eq:bnd-hatPdiffT0}, we obtain
\begin{equation*}
\int_{D_R}\tilde\mu_{R,\e}^2(\hat P-T_0[\hat P])^2
\,\lesssim\,|D_R|\Big(\sup_{D_R}|\nabla(\eta_R\nabla\hat u)|^2\Big)+\int_{D_R}(1-\eta_R)^2\tilde\mu_{R,\e}^2|\nabla\hat u|^2.
\end{equation*}
Combining this with~\eqref{e.p1.11+} and~\eqref{eq:bnd-ubarutil},
the claim~\eqref{eq:est-Q-bar} follows.

\medskip
\substep{2.5} Control of the divergence:
\begin{equation}\label{eq:est-div-w}
\int_{D_R}\tilde\mu_{R,\e}^2\,\Div(\tilde w)^2\,\lesssim\,\int_{D_R}(1-\eta_R)^2\tilde\mu_{R,\e}^2|\nabla\hat u|^2+\Big(\sup_{D_R}|\nabla(\eta_R\nabla\hat u)|^2\Big)\int_{D_R}\big(1+|\psi|^2\big).
\end{equation}
As $\Div (u)=\Div (\hat u)=\Div(\psi_E)=0$, the definition~\eqref{eq:trunc-wQ} of $\tilde w$ yields
\begin{equation*}
\Div (\tilde w)\,=\,\Div(\hat u-T_1[\hat u])-\psi_E\cdot\nabla(\eta_RT_0[\partial_E\hat u]),
\end{equation*}
and the claim~\eqref{eq:est-div-w} follows from the estimates~\eqref{eq:bnd-uhatT1} and~\eqref{eq:bnd-ubarutil}.

\medskip
\substep{2.6} Proof of~\eqref{e.p1.8}.\\
Combining~\eqref{eq:J0123}, \eqref{e.p1.6}, \eqref{eq:bnd-J1}, and~\eqref{e.p1.7}, we obtain for all $K\gg1$ and $0<\e\ll K^{-1/2}$,
\begin{multline*}
\int_{D_R} \tilde \mu_{R,\e}^2|\nabla w|^2
\,\lesssim\,\frac1K\int_{D_R}\tilde\mu_{R,\e}^2|\nabla\tilde w|^2
+\frac1K\int_{D_R}\tilde\mu_{R,\e}^2( Q^2+\tilde Q^2)\\
+K\Big(\sup_{D_R}|\nabla(\eta_R\nabla\hat u)|^2\Big)\int_{D_R}(1+|(\psi,\zeta,\nabla\psi,\Sigma\mathds1_{\R^d\setminus\Ic})|^2)
+K\int_{D_R}(1-\eta_R)^2\tilde\mu_{R,\e}^2|\nabla\hat u|^2\\
+K\int_{D_R} \tilde \mu_{R,\e}^2|\nabla(w-\tilde w)|^2
+K\int_{D_R} \tilde \mu_{R,\e}^2\,\Div(\tilde w)^2.
\end{multline*}
Decomposing $\nabla\tilde w=\nabla w+\nabla(\tilde w-w)$ and $\tilde Q=Q+(\tilde Q-Q)$, using the bounds~\eqref{eq:est-w-bar} and~\eqref{eq:est-Q-bar} on the truncation errors $\nabla(w-\tilde w)$ and $Q-\tilde Q$, and using the bound~\eqref{eq:est-div-w} on~$\Div(\tilde w)$, we find
\begin{multline*}
\int_{D_R} \tilde \mu_{R,\e}^2|\nabla w|^2
\,\lesssim\,\frac1K\int_{D_R}\tilde\mu_{R,\e}^2|\nabla w|^2
+\frac1K\int_{D_R}\tilde\mu_{R,\e}^2Q^2\\
+K\Big(\sup_{D_R}|\nabla(\eta_R\nabla\hat u)|^2\Big)\int_{D_R}(1+|(\psi,\zeta,\nabla\psi,\Sigma\mathds1_{\R^d\setminus\Ic})|^2)
+K\int_{D_R}(1-\eta_R)^2\tilde\mu_{R,\e}^2|\nabla\hat u|^2.
\end{multline*}
Choosing $K\gg1$ large enough to absorb the first right-hand side term, and noting that~$\tilde\mu_{R,\e}\simeq\mu_{R,\e}$ on $D_R$, the conclusion~\eqref{e.p1.8} follows.

\medskip
\step3 Weighted pressure estimate for the two-scale expansion error: for all $0<\e\ll1$,
\begin{multline}\label{e.p1.9}
\int_{D_R}\mu_{R,\e}^2 Q^2
\,\lesssim\,\int_{D_R}\mu_{R,\e}^2|\nabla w|^2
+\int_{D_R}(1-\eta_R)^2\mu_{R,\e}^2|\nabla\hat u|^2\\
+\Big(\sup_{D_R}|\nabla(\eta_R\nabla\hat u)|^2\Big)~\int_{D_R}\big(1+|(\psi,\zeta,\Sigma\mathds1_{\R^d\setminus\Ic})|^2\big).
\end{multline}
Combining this with the bound~\eqref{e.p1.8} on $\nabla w$, and choosing $K\gg1$ large enough, we deduce for all $0<\e\ll1$,
\begin{multline}\label{e.p1.9concl}
\int_{D_R} \mu_{R,\e}^2(|\nabla w|^2+Q^2)
\,\lesssim\,\int_{D_R}(1-\eta_R)^2\mu_{R,\e}^2|\nabla\hat u|^2\\
+\Big(\sup_{D_R}|\nabla(\eta_R\nabla\hat u)|^2\Big)\int_{D_R}(1+|(\psi,\zeta,\nabla\psi,\Sigma\mathds1_{\R^d\setminus\Ic})|^2).
\end{multline}
We turn to the proof of~\eqref{e.p1.9}. For that purpose, we shall again appeal to the truncated version $\tilde Q$ of~$Q$ as in Step~2, cf.~\eqref{eq:trunc-wQ}.
We also recall the notation~\eqref{eq:cst-press} for $P'$, where we choose the constants $P_*$ and $\{P_n\}_{n}$ such that
\[P_n=\fint_{I_n+\frac\delta2B} \hat P+\eta_R(x_n)\,\bb : \fint_{I_n+\frac\delta2B}\D(\hat u),\qquad \int_{D_R}\tilde \mu_{R,\e}^2 \tilde Q=0.\]
Note that this choice entails in particular $\tilde Q=0$ inside inclusions $\{I_n\}_{n\in\Nc_R}$.
With these definitions, we may turn to the proof of~\eqref{e.p1.9}, which we split into three further substeps.

\medskip
\substep{3.1} Weighted Bogovskii construction: given $0<\e<\frac1{2d}$, there exists a vector field $S \in H^1_0(D_R)^d$ such that $S|_{I_n}$ is constant for all $n\in \Nc_R$
and such that
\begin{gather}
\Div( S)=-\tilde \mu_{R,\e}^2 \tilde Q,\qquad\text{in $D_R$},\nonumber\\
\int_{D_R}\tilde \mu_{R,\e}^{-2} |\nabla S|^2 \lesssim \int_{D_R} \tilde \mu_{R,\e}^2 \tilde Q^2.\label{e:estim-weight-pres+}
\end{gather}
Since $\int_{D_R}\tilde\mu_{R,\e}^2\tilde Q=0$, and since the weight $\tilde\mu_{R,\e}^{-2}\simeq\mu_{R,\e}^{-2}$ on $D_R$ can be extended to \mbox{$x\mapsto|1-\frac{|x|}{R}|^{-\e}$} on $\R^d$, which belongs to the Muckenhoupt class $A_2$ uniformly in $R$ provided $\e<1$, and which satisfies \mbox{$\fint_{D_R}\mu_{R,\e}^{-2d}\lesssim\int_{B}(1-|x|)^{-\e d}\lesssim1$} provided $\e<\frac1{2d}$,
we may appeal to the weighted Bogovskii construction in form of Lemma~\ref{lem:muck}. Note that by definition the set~$D_R$ is star-shaped with respect to every point in $B_{R/2}$ as soon as $R \gg 1$. Hence,  there exists a vector field $S^\circ\in H^1_0(D_R)^d$ such that
\begin{gather*}
\Div (S^\circ)=-\tilde \mu_{R,\e}^2 \tilde Q,\qquad\text{in $D_R$},\nonumber\\
\int_{D_R}\tilde \mu_{R,\e}^{-2} |\nabla S^\circ|^2 \lesssim \int_{D_R} \tilde \mu_{R,\e}^2 \tilde Q^2.
\end{gather*}
It remains to modify $S^\circ$ to make it constant inside the inclusions $\{I_n\}_{n\in\Nc_R}$ without changing its divergence and the bound on its norm.
For that purpose, we essentially follow the argument of \cite[Proof of Proposition~2.1]{DG-19}; see also the proof of Lemma~\ref{lem:pres}.
More precisely, for all $n\in \Nc_R$, recalling that $\dist(I_n,\partial D_R)\ge \delta$ and that $\tilde Q=0$ in $I_n$, a standard use of the Bogovskii operator allows to construct as in~\eqref{eq:ext-Bogo} a vector field $S^n\in H^1_0(I_n+\frac\delta2B)^d$ such that $S^n=-S^\circ+\fint_{I_n}S^\circ$ in $I_n$ and
\begin{gather*}
\Div (S^n)=0,\qquad\text{in $I_n+\tfrac\delta2B$},\\
\|\nabla S^n\|_{\Ld^2((I_n+\frac\delta2B)\setminus I_n)} \,\lesssim\, \|\nabla S^\circ\|_{\Ld^2(I_n)}.
\end{gather*}
Smuggling in the weight $\tilde\mu_{R,\e}^{-1}$ (which is constant on the fattened inclusions), this yields
\begin{equation*}
\|\tilde\mu_{R,\e}^{-1}\nabla S^n\|_{\Ld^2((I_n+\frac\delta2B)\setminus I_n)} \,\lesssim\,\|\tilde\mu_{R,\e}^{-1}\nabla S^\circ\|_{\Ld^2(I_n)}.
\end{equation*}
Since the fattened inclusions are all disjoint, cf.~\ref{Hd}, extending $S^n$ by $0$ in $D_R\setminus(I_n+\frac\delta2 B)$ for all $n\in\Nc_R$, the vector field $S:=S^\circ+\sum_{n\in\Nc_R}S^n$ satisfies all the required properties.

\medskip
\substep{3.2} Proof of~\eqref{e.p1.9}.\\
Testing equation~\eqref{e.p1.5} with the test function $S\in H^1_0(D_R)^d$ constructed in the previous substep yields
\[L_0\,=\,L_1+L_2+L_3,\]
in terms of
\begin{eqnarray*}
L_0&:=&\int_{D_R} \nabla S : \nabla w -\int_{D_R} \Div(S)\, Q,\\
L_1&:=&-\sum_{n\in\Nc_R}\int_{\partial I_n}S\cdot\sigma(u,P+P_*-P_n)\nu+\sum_{n\in\Nc_R}\int_{I_n}(\eta_R\partial_E\hat u)\nabla S:J_E,\\
L_2&:=&-2\int_{D_R} (1-\eta_R)\nabla S : (\Id-\Bb )\D(\hat u),\\
L_3&:=&-\int_{D_R}\nabla S:\Big((2\psi_E\otimes_s-\zeta_E)\nabla(\eta_R\partial_E\hat u)-\Id(\psi_E\cdot\nabla)(\eta_R\partial_E\hat u)\Big).
\end{eqnarray*}
We start by giving a  lower bound on $L_0$.
Using the defining property~\eqref{e:estim-weight-pres+} of the test function $S$ in form of
\[-\int_{D_R} \Div (S)\, Q\,=\,\int_{D_R} \tilde \mu_{R,\e}^2 Q \tilde Q\,\ge\,\int_{D_R} \tilde \mu_{R,\e}^2 \tilde Q^2-\Big(\int_{D_R} \tilde \mu_{R,\e}^2 \tilde Q^2\Big)^\frac12\Big(\int_{D_R} \tilde \mu_{R,\e}^2 (Q-\tilde Q)^2\Big)^\frac12,\]
and using the bound~\eqref{e:estim-weight-pres+} on the weighted norm of $\nabla S$ in form of
\begin{eqnarray*}
\Big|\int_{D_R} \nabla S : \nabla w\Big|&\le& \Big(\int_{D_R} \tilde \mu_{R,\e}^{-2}  |\nabla S|^2\Big)^\frac12\Big( \int_{D_R} \tilde \mu_{R,\e}^{2}  |\nabla w|^2\Big)^\frac12\\
&\lesssim& \Big( \int_{D_R} \tilde \mu_{R,\e}^2 \tilde Q^2\Big)^\frac12\Big(\int_{D_R} \tilde \mu_{R,\e}^{2}  |\nabla w|^2\Big)^\frac12,
\end{eqnarray*}
we deduce for all $K\ge1$,
\begin{equation}\label{eq:lowerbnd-L0}
L_0\,\ge\,\int_{D_R}\tilde\mu_{R,\e}^2\tilde Q^2-\frac1K\int_{D_R}\tilde\mu_{R,\e}^2\tilde Q^2-K\int_{D_R}\tilde\mu_{R,\e}^2(Q-\tilde Q)^2-CK\int_{D_R}\tilde\mu_{R,\e}^2|\nabla w|^2.
\end{equation}
Next, recalling that $S|_{I_n}$ is constant for all $n\in\Nc_R$, and using the boundary conditions for~$u$, we find $L_1=0$.
It remains to estimate $L_2$ and $L_3$.
Smuggling in the weight $\tilde\mu_{R,\e}$, we find for all $K\ge1$,
\begin{multline*}
L_2+L_3\,\lesssim\,\frac1K\int_{D_R}\tilde\mu_{R,\e}^{-2}|\nabla S|^2
+K\int_{D_R}(1-\eta_R)^2\tilde\mu_{R,\e}^2|\nabla\hat u|^2\\
+K\Big(\sup_{D_R}|\nabla(\eta_R\nabla\hat u)|^2\Big)~\int_{D_R}\big(1+|(\psi,\zeta)|^2\big),
\end{multline*}
Using the weighted estimate~\eqref{e:estim-weight-pres+} on $\nabla S$ to estimate the first right-hand side term, and combining with the lower bound~\eqref{eq:lowerbnd-L0} on $L_0$, we deduce for all $K\ge1$,
\begin{multline*}
\int_{D_R}\tilde\mu_{R,\e}^2\tilde Q^2
\,\lesssim\,\frac1K\int_{D_R}\tilde\mu_{R,\e}^{2}\tilde Q^2
+K\int_{D_R}\tilde\mu_{R,\e}^2(Q-\tilde Q)^2
+K\int_{D_R}\tilde\mu_{R,\e}^2|\nabla w|^2\\
+K\Big(\sup_{D_R}|\nabla(\eta_R\nabla\hat u)|^2\Big)~\int_{D_R}\big(1+|(\psi,\zeta)|^2\big)
+K\int_{D_R}(1-\eta_R)^2\tilde\mu_{R,\e}^2|\nabla\hat u|^2,
\end{multline*}
Choosing $K\gg1$ large enough to absorb the first right-hand side term, and decomposing $\tilde Q=Q+(\tilde Q-Q)$, we obtain
\begin{multline*}
\int_{D_R}\tilde\mu_{R,\e}^2\, Q^2
\,\lesssim\,
\int_{D_R}\tilde\mu_{R,\e}^2(Q-\tilde Q)^2
+\int_{D_R}\tilde\mu_{R,\e}^2|\nabla w|^2\\
+\Big(\sup_{D_R}|\nabla(\eta_R\nabla\hat u)|^2\Big)~\int_{D_R}\big(1+|(\psi,\zeta)|^2\big)
+\int_{D_R}(1-\eta_R)^2\tilde\mu_{R,\e}^2|\nabla\hat u|^2.
\end{multline*}
Using the bound~\eqref{eq:est-Q-bar} on the truncation error $Q-\tilde Q$, and recalling that $\tilde\mu_{R,\e}\simeq\mu_{R,\e}$ on~$D_R$, the conclusion~\eqref{e.p1.9} follows.

\medskip
\step4 Conclusion: proof of~\eqref{e.p1.1}.\\
We split the proof into five further substeps.

\medskip
\substep{4.1} Caccioppoli-type inequality for homogeneous steady Stokes equation:
given a solution $(\bar v,\bar T)$ of
\begin{equation}\label{eq:vR0}
-\Div(2\Bb\D(\bar v))+\nabla\bar T=0,\qquad\Div(\bar v)=0,\qquad\text{in $B_R$},
\end{equation}
we have for all $0<r<R$ and $K\ge1$,
\begin{equation}\label{eq:Cacc}
\int_{B_r}|\nabla\bar v|^2\,\lesssim\,\frac1K\int_{B_R}|\nabla\bar v|^2+\frac{K}{(R-r)^2}\int_{B_R}|\bar v|^2.
\end{equation}
Consider a cut-off function $\chi_{r,R}\in C^\infty_c(\R^d)$ such that $\chi_{r,R}|_{B_r}=1$, $\chi_{r,R}|_{\R^d\setminus B_R}=0$, and $|\nabla\chi_{r,R}|\lesssim (R-r)^{-1}$.
Testing the equation~\eqref{eq:vR0} with the test function $\chi_{r,R}^2\bar v$, we find
\[\int_{\R^d}\chi_{r,R}^2\D(\bar v):2\Bb\D(\bar v)\,=\,-2\int_{\R^d}\chi_{r,R}\,\bar v\otimes\nabla\chi_{r,R}:(2\Bb\D(\bar v)-\bar T),\]
and thus
\begin{equation}\label{eq:bnd-Cac-Dv}
\int_{\R^d}\chi_{r,R}^2|\!\D(\bar v)|^2\,\lesssim\,\frac1{R-r}\Big(\int_{B_R}(|\!\D(\bar v)|^2+\bar T^2)\Big)^\frac12\Big(\int_{B_{R}}|\bar v|^2\Big)^\frac12.
\end{equation}
Since $\Div(\bar v)=0$, integration by parts yields
\begin{eqnarray*}
\int_{\R^d}\chi_{r,R}^2|\nabla\bar v|^2&=&2\int_{\R^d}\chi_{r,R}^2|\!\D(\bar v)|^2-\int_{\R^d}\chi_{r,R}^2\partial_i\bar v_j\partial_j\bar v_i\\
&=&2\int_{\R^d}\chi_{r,R}^2|\!\D(\bar v)|^2+2\int_{\R^d}\chi_{r,R}\nabla\chi_{r,R}\otimes\bar v:\nabla\bar v,
\end{eqnarray*}
and thus
\begin{equation*}
\int_{\R^d}\chi_{r,R}^2|\nabla\bar v|^2
\,\lesssim\,\int_{\R^d}\chi_{r,R}^2|\!\D(\bar v)|^2+\frac1{(R-r)^2}\int_{B_R}|\bar v|^2.
\end{equation*}
Combining this with~\eqref{eq:bnd-Cac-Dv}, we deduce for all $K\ge1$,
\begin{equation*}
\int_{B_r}|\nabla\bar v|^2\,\lesssim\,\frac1K\int_{B_R}(|\!\D(\bar v)|^2+\bar T^2)+\frac{K}{(R-r)^2}\int_{B_{R}}|\bar v|^2.
\end{equation*}
As the pressure $\bar T$ in~\eqref{eq:vR0} is only defined up to an additive constant, we may choose without loss of generality $\int_{B_R}\bar T=0$,
and we then appeal to a standard pressure estimate:
a standard use of the Bogovskii operator in form of e.g.~\cite[Theorem III.3.1]{Galdi} yields
\[\int_{B_{R}}\bar T^2\,\lesssim\,\int_{B_{R}}|\nabla\bar v|^2\]
and the claim~\eqref{eq:Cacc} follows.

\medskip
\substep{4.2} Interior regularity estimate for homogeneous steady Stokes equation~\eqref{e.p1.4}: for any boundary layer $4< \rho< R$,
\begin{eqnarray}\label{e.p1.14}
\rho^{2(n-1)}\sup_{B_{R-\rho}}|\nabla^n\hat u|^2\,\lesssim_n\,\big(\tfrac\rho R\big)^{-d}\fint_{D_R}|\nabla\hat u|^2.
\end{eqnarray} 
First consider a solution $(\bar v,\bar T)$ of the following homogeneous steady Stokes equation,
\begin{equation}\label{eq:vR}
-\Div(2\Bb\D(\bar v))+\nabla\bar T=0,\qquad\Div(\bar v)=0,\qquad\text{in $B$}.
\end{equation}
In view of the standard interior regularity theory for this equation, see~\cite[Theorem~IV.4.1]{Galdi}, we find for all $n\ge0$,
\[\int_{\frac12B}|\langle\nabla\rangle^n\nabla\bar v|^2\,\lesssim_n\,\int_{B}\big(|\nabla\bar v|^2+|\bar T|^2\big).\]
We then appeal to a pressure estimate for $\bar T$: assuming without loss of generality $\int_{B}\bar T=0$,
a standard use of the Bogovskii operator in form of e.g.~\cite[Theorem III.3.1]{Galdi} yields
\[\int_{\frac12B}|\langle\nabla\rangle^n\nabla\bar v|^2\,\lesssim_n\,\int_{B}|\nabla\bar v|^2.\]
By Sobolev's embedding, this entails for all $n\ge1$,
\[\sup_{\frac12B}|\nabla^n\bar v|^2\,\lesssim_n\,\int_{B}|\nabla\bar v|^2.\]
Upon rescaling and translation, this implies for all $\rho<1$, $x\in B_{1-\rho}$, and $n\ge1$,
\[\rho^{2(n-1)}|\nabla^n\bar v(x)|^2\,\lesssim_n\,\fint_{B_\rho(x)}|\nabla\bar v|^2,\]
hence, for all $n\ge1$,
\begin{eqnarray*}
\rho^{2(n-1)}\sup_{B_{1-\rho}}|\nabla^n\bar v|^2\,\lesssim_n\,\rho^{-d}\int_B|\nabla\bar v|^2.
\end{eqnarray*}
Turning back to equation~\eqref{e.p1.4} and recalling that $B_{R-3}\subset D_R$,
the claim~\eqref{e.p1.14} follows after rescaling.

\medskip
\substep{4.3} Reduction to the two-scale expansion error: for all $4\le r\le \frac14R$,
\begin{multline}\label{e.p1.13}
\fint_{B_r} \big|\nabla u-\nabla\hat u(0)-(\partial_E \hat u)(0)\nabla \psi_E\big|^2
\,\lesssim\,(\tfrac{r}{R})^{-d-2}\fint_{D_R} \mu_{R,\e}^2 |\nabla w|^2\\
+\bigg( (\tfrac rR)^2\fint_{B_R}\big(1+|\nabla\psi|^2\big)+(\tfrac{r}{R})^{-d}\frac1{R^2} \fint_{B_R} |\psi|^2 \bigg)\fint_{B_R} |\nabla u|^2.
\end{multline}
Consider the following local two-scale expansion error centered at the origin,
\[w_\circ:=u-\hat u(0)-\nabla\hat u(0)x-\psi_E\partial_E \hat u(0),\qquad Q_\circ:=P-\Sigma_E\partial_E \hat u(0),\]
and note that equations~\eqref{e:up-whole0} and~\eqref{eq:psiE-ref0} yield the following on $D_R$,
\[-\triangle w_\circ+\nabla (Q_\circ\mathds1_{\R^d\setminus\Ic})\,=\,-\sum_n\delta_{\partial I_n}\sigma(w_\circ,Q_\circ)\nu.\]
We appeal to a Caccioppoli-type argument: as in the proof of~\eqref{eq:Cacc}, choosing a cut-off function that is constant in the inclusions, and using the boundary conditions for~$u$ and~$\psi_E$, we find for all $4\le r\le\frac14R$ and $K\ge1$,
\begin{equation}\label{e.p1.15}
\fint_{B_{r}}|\nabla w_\circ|^2\,\lesssim\,\frac1K\fint_{B_{2r}}|\nabla w_\circ|^2+\frac K{r^2}\fint_{B_{2r}}| w_\circ|^2,
\end{equation}
and it remains to examine the last right-hand side term.
Comparing the local error $w_\circ$ to its global version $w=u-\hat u-\eta_R\psi_E\partial_E \hat u$, cf.~\eqref{eq:def-wQ},
and recalling that $\eta_R=1$ on $B_{R-2\rho}$,
we obtain from the triangle inequality, for all $r,\rho\le\frac R4$ (which entails $B_{2r}\subset B_{R-2\rho}$),
\begin{equation*}
\fint_{B_{2r}} |w_\circ|^2 \,\lesssim\, \fint_{B_{2r}} |w|^2+ \Big(\sup_{B_{2r}}|\hat u-\hat u(0)- \nabla\hat u(0) x|^2\Big)+\Big(\sup_{B_{2r}} |\nabla \hat u-\nabla \hat u(0)|^2\Big) \fint_{B_{2r}}|\psi|^2.
\end{equation*}
Using Taylor's formula, the interior regularity estimate~\eqref{e.p1.14} with $\rho=\frac R4$, and the energy estimate~\eqref{e.p1.12+}, we find for all $r\le\frac14R$,
\begin{multline*}
\sup_{B_{2r}}|\hat u-\hat u(0)- \nabla \hat u(0)x|^2+ r^2\sup_{B_{2r}}|\nabla \hat u-\nabla \hat u(0)|^2\\
\,\lesssim\, r^4 \sup_{B_{2r}}|\nabla^2\hat u|^2
\,\lesssim\, r^2(\tfrac rR)^2\fint_{D_R}|\nabla\hat u|^2
\,\lesssim\, r^2(\tfrac rR)^2\fint_{D_R}|\nabla u|^2,
\end{multline*}
so that the above becomes
\begin{equation}\label{e.p1.15b}
\fint_{B_{2r}} |w_\circ|^2 \,\lesssim\, \fint_{B_{2r}} |w|^2 + \Big(r^2(\tfrac rR)^2+(\tfrac{r}R)^{2-d}\fint_{B_{R}} |\psi|^2\Big) \fint_{D_R}|\nabla u|^2.
\end{equation}
It remains to analyze the first right-hand side term in this estimate.
By definition of the weight~$\mu_{R,\e}$ in~\eqref{eq:def-muReps}, appealing to Hardy's inequality~\eqref{hardy},
we find for all $r\le \frac14R$,
\begin{eqnarray*}
\fint_{B_{2r}} |w|^2 \,\lesssim\, \fint_{B_{2r}}  \big(1-\tfrac{|x|}{R}\big)^{\e-2} |w|^2
&\lesssim& (\tfrac rR)^{-d} \fint_{D_{R}} \big(1-\tfrac{|x|}{R}\big)^{\e-2}|w|^2\\
&\lesssim& \e^{-2} R^2(\tfrac rR)^{-d}  \fint_{D_{R}} |\nabla \mu_{R,\e}|^2 |w|^2\\
&\lesssim& R^2(\tfrac rR)^{-d}\fint_{D_R}  \mu_{R,\e}^2 |\nabla w|^2.
\end{eqnarray*}
Combined with~\eqref{e.p1.15} and~\eqref{e.p1.15b}, this yields the following, for all $4\le r\le \frac14R$ and $K\ge1$,
\begin{multline}\label{eq:bef-iter-abs}
\fint_{B_{r}}|\nabla w_\circ|^2\,\lesssim\,\frac1K\fint_{B_{2r}}|\nabla w_\circ|^2\\
+ K\Big((\tfrac rR)^2+(\tfrac{r}R)^{-d}\frac1{R^2}\fint_{B_{R}} |\psi|^2\Big) \fint_{D_R}|\nabla u|^2+K(\tfrac rR)^{-d-2}\fint_{D_R}  \mu_{R,\e}^2 |\nabla w|^2.
\end{multline}
In order to absorb the first right-hand side term, we proceed by iteration. Let us first rewrite~\eqref{eq:bef-iter-abs} as follows: for any $K\ge1$,
\[f(r)\le \frac1Kf(2r)+CKg(r),\qquad\text{for all $4\le r\le\tfrac14R$},\]
where we have set for abbreviation,
\begin{eqnarray*}
f(r)&:=&\fint_{B_r} |\nabla w_\circ|^2,\\
g(r)&:=&\Big( (\tfrac rR)^2+(\tfrac{r}{R})^{-d}\frac1{R^2} \fint_{D_R} |\psi|^2 \Big)\fint_{D_R} |\nabla u|^2
+(\tfrac{r}{R})^{-d-2}\fint_{D_R} \mu_{R,\e}^2 |\nabla w|^2.
\end{eqnarray*}
Iterating this estimate yields for all $r\ge4$ and $n\ge1$ with $2^nr\le\frac14R$,
\[f(r)\,\le\, CK\sum_{m=0}^{n-1} K^{-m}g(2^mr)+K^{-n}f(2^nr).\]
Noting that $g(2^mr)\le 4^mg(r)$ and choosing $K=8$, this entails
\[f(r)\,\lesssim\, g(r)+8^{-n}f(2^nr).\]
Choosing $n$ large enough such that $2^nr\simeq R$, with $2^nr\le\frac14R$, we deduce
\begin{equation}\label{eq:iterate-ag}
f(r)\,\lesssim\, g(r)+(\tfrac rR)^3 f(\tfrac14R).
\end{equation}
It remains to estimate the second right-hand side term. By definition of $f$ and of $w_\circ$, we find
\[f(\tfrac14R)\,\lesssim\,\fint_{D_R} |\nabla w_\circ|^2\,\lesssim\,\fint_{D_R}|\nabla u|^2+|\nabla\hat u(0)|^2\fint_{D_R}(1+|\nabla\psi|^2).\]
Using the interior regularity estimate~\eqref{e.p1.14} with $\rho\simeq R$ and using the energy estimate~\eqref{e.p1.12+}, we note that
\begin{equation}\label{eq:bnd-nabu0}
|\nabla\hat u(0)|^2\,\lesssim\,\fint_{D_R}|\nabla\hat u|^2\,\lesssim\,\fint_{D_R}|\nabla u|^2,
\end{equation}
so that the above becomes
\[f(\tfrac14R)\,\lesssim\,\Big(\fint_{D_R}\big(1+|\nabla\psi|^2\big)\Big)\fint_{D_R}|\nabla u|^2.\]
Combining this with~\eqref{eq:iterate-ag},
and inserting the definition of $f$, $g$, and $w_\circ$, the claim~\eqref{e.p1.13} follows.

\medskip
\substep{4.4} Estimate on the two-scale expansion error: for all $0<\e\ll1$,
\begin{multline}\label{eq:sub44}
\int_{D_R} \mu_{R,\e}^2|\nabla w|^2\\
\,\lesssim\,
\bigg((\tfrac{\rho}{R})^\e+(\tfrac\rho R)^{-d-2}\frac1{R^2}\fint_{D_R}\big(1+|(\psi,\zeta,\nabla\psi,\Sigma\mathds1_{\R^d\setminus\Ic})|^2\big)\bigg)\int_{D_R}|\nabla u|^2.
\end{multline}
Starting point is~\eqref{e.p1.9concl}: for all $0<\e\ll1$,
\begin{multline*}
\int_{D_R} \mu_{R,\e}^2|\nabla w|^2
\,\lesssim\,\int_{D_R}(1-\eta_R)^2\mu_{R,\e}^2|\nabla\hat u|^2\\
+\Big(\sup_{D_R}|\nabla(\eta_R\nabla\hat u)|^2\Big)\int_{D_R}\big(1+|(\psi,\zeta,\nabla\psi,\Sigma\mathds1_{\R^d\setminus\Ic})|^2\big).
\end{multline*}
Noting that the definition of $\eta_R$ and $\mu_{R,\e}$ entails $(1-\eta_R)^2\mu_{R,\e}^2\lesssim(\frac{\rho}{R})^\e$, recalling that $\eta_R$ is supported in $B_{R-\rho}$ and satisfies $|\nabla\eta_R|\lesssim\rho^{-1}$, using the interior regularity estimate~\eqref{e.p1.14}, and using the energy estimate~\eqref{e.p1.12+},
the claim~\eqref{eq:sub44} follows.

\medskip
\substep{4.5} Proof of~\eqref{e.p1.1}.\\
Inserting the error bound~\eqref{eq:sub44} into~\eqref{e.p1.13}, we find for all $4\le r,\rho\le\frac14R$,
\begin{multline}\label{eq:pre-concl}
\fint_{B_r} \big|\nabla u-\nabla\hat u(0)-(\partial_E \hat u)(0)\nabla \psi_E\big|^2
\,\lesssim\,\bigg((\tfrac rR)^2\fint_{B_R}\big(1+|\nabla\psi|^2\big)\\
+(\tfrac{r}{R})^{-d-2}\Big((\tfrac{\rho}{R})^\e+(\tfrac\rho R)^{-d-2}\frac1{R^2}\fint_{B_R}\big(1+|(\psi,\zeta,\nabla\psi,\Sigma\mathds1_{\R^d\setminus\Ic})|^2\big)\Big)\bigg)\fint_{B_R}|\nabla u|^2.
\end{multline}
Next, we slightly reformulate this estimate by removing the dependence on $\nabla\psi$.
For that purpose, we appeal to a Caccioppoli-type argument for $\psi$: arguing as in~\eqref{e.p1.15}, now starting from equation~\eqref{eq:psiE-ref0}, we find for all $K,R\ge1$,
\begin{equation*}
\fint_{B_{R}}|\nabla\psi|^2\,\lesssim\,\frac1K\fint_{B_{2R}}|\nabla\psi|^2+K\bigg(1+\frac{1}{R^2}\fint_{B_{2R}}\Big|\psi-\fint_{B_{2R}}\psi\Big|^2\bigg).
\end{equation*}
Iterating this estimate for some $K\gg1$ large enough, and recalling that the ergodic theorem yields $\fint_{B_R}|\nabla\psi|^2\to\expecmm{|\nabla\psi|^2}\lesssim1$ almost surely as $R\uparrow\infty$, we deduce for all $R\ge1$,
\begin{equation}\label{eq:bnditer-psiE}
\fint_{B_{R}}|\nabla\psi|^2
\,\lesssim\,1+\gamma_{R}^2,
\end{equation}
where we recall that $\gamma_R$ is defined in~\eqref{e.p1.2}.
Recalling the choice $\fint_{B_R}(\psi,\zeta,\Sigma\mathds1_{\R^d\setminus\Ic})=0$ in this proof, and appealing to the pressure estimate of Lemma~\ref{lem:pres} to further remove the dependence on $\Sigma$ in~\eqref{eq:pre-concl}, we obtain for all $4\le r,\rho\le\frac14R$,
\begin{equation*}
\fint_{B_r} \big|\nabla u-\nabla\hat u(0)-(\partial_E \hat u)(0)\nabla \psi_E\big|^2
\,\lesssim\,\bigg((\tfrac rR)^2
+(\tfrac{r}{R})^{-d-2}\Big((\tfrac{\rho}{R})^\e+(\tfrac\rho R)^{-d-2}\gamma_{R}^2)\Big)\bigg)\fint_{B_R}|\nabla u|^2.
\end{equation*}
It remains to optimize in $\rho$.
If $\gamma_{R}\le1$, the choice $(\frac\rho R)^{d+2+\e}\simeq\gamma_R^2$ yields the conclusion~\eqref{e.p1.1} with $M_0=\nabla\hat u(0)$ up to renaming $\e$. If $\gamma_R\ge1$ or if $\frac14R\le r\le R$, then the conclusion~\eqref{e.p1.1} trivially holds with $M_0=0$.

\medskip
\step5 Proof of the non-degeneracy property~\eqref{e.p1.3}.\\
The upper bound in~\eqref{e.p1.3} follows from the Caccioppoli-type inequality~\eqref{eq:bnditer-psiE}, and it remains to establish the lower bound.
Poincar\'e's inequality and the triangle inequality yield
\begin{eqnarray*}
\Big(\fint_{B_{R/2}}|\nabla \psi_E+E|^2\Big)^\frac12&\gtrsim &\frac1{R} \bigg(\fint_{B_{R/2}}\Big|(\psi_E+Ex)-\fint_{B_{R/2}}(\psi_E+Ex)\Big|^2\bigg)^\frac12
\\
&\ge&  \frac1{R} \Big(\fint_{B_{R/2}}|E x|^2\Big)^\frac12 - \frac1{R} \bigg(\fint_{B_{R/2}}\Big|  \psi_E-\fint_{B_{R/2}} \psi_E\Big|^2\bigg)^\frac12
\\
&\gtrsim&(1-C\gamma_R)|E|,
\end{eqnarray*}
and the conclusion~\eqref{e.p1.3} follows.\qed

\subsection{Proof of Lemma~\ref{lem:muck}: Weighted Bogovskii construction}\label{sec:pr-p1+}
By scaling, we may assume without loss of generality that the domain $D$ is contained in $\frac{R}{R_0}B$ and is star-shaped with respect to every point in~$B$.
In addition, it is enough to consider $F\in C^\infty_c(D)$.
A solution $S\in H^1_0(D)^d$ to the problem $\Div (S)=\mu F$ in $D$ can then be constructed via the Bogovskii formula, and its gradient can be represented as follows, cf.~\cite[(III.3.19)]{Galdi},
\[ \nabla S\,=\,F_1+F_2+F_3,\]
in terms of
\begin{eqnarray*}
F_{1,ij}(x)&=& \int_{D} K_{ij}(x,x-y)\mu(y)F(y)dy,\\
F_{2,ij}(x)&=& \int_{D} G_{ij}(x,y)\mu(y)F(y)dy,\\
F_{3,ij}(x)&=& \mu(x)F(x) \int_{D} \frac{(x_j-y_j)(x_i-y_i)}{|x-y|^2}\,\alpha(y)\,dy,
\end{eqnarray*}
where $K_{ij}$ is a singular Calder\'on--Zygmund kernel of the form
\[K_{ij}(x,x-y)=\frac{k_{ij}(x,\frac{x-y}{|x-y|})}{|x-y|^d},\]
where $G_{ij}$ satisfies the pointwise bound $|G_{ij}(x,y)|\lesssim  |x-y|^{1-d}$, and where $\alpha \in C^\infty_c(\R^d)$ is supported in $B$ and satisfies $\int_{B} \alpha=1$. Note that the integral defining $F_1$ is understood in the principal value sense.

\medskip\noindent
We turn to the estimate of $\int_D\mu^{-1}|\nabla S|^2$ and we separately analyze the contributions of~$F_1,F_2,F_3$.
First, as $K$ is a Calder\'on--Zygmund kernel and since $\mu^{-1}$ belongs to the Muckenhoupt class $A_2$, the weighted Calder\'on--Zygmund theory yields
\[\int_{\R^d} \mu^{-1}|F_1|^2 \,\lesssim\, \int_{D} \mu^{-1}|\mu F|^2 \,=\, \int_{D} \mu |F|^2.\]
Second, the Hardy--Littlewood--Sobolev inequality yields for $d>2$,
\begin{eqnarray*}
\int_{D} \mu^{-1}|F_2|^2
&\lesssim&\int_{D}\mu^{-1}(x) \Big(\int_{D} |x-y|^{1-d}|(\mu F)(y)|dy\Big)^2dx \\
&\lesssim&\|\mu^{-1}\|_{\Ld^{d/2}(D)}^2\|\mu F\|_{\Ld^2(D)}^2\\
&\lesssim&\|\mu^{-1}\|_{\Ld^{d/2}(D)}^2\int_D\mu|F|^2.
\end{eqnarray*}
For $d=2$, we may rather bound $|x-y|^{1-d}\lesssim_D|x-y|^{1-d-\eta}$ for $x,y\in D$, and a corresponding estimate is then deduced with a norm of $\mu^{-1}$ in $\Ld^{1+\eta}(D)$.
Finally, the properties of $\alpha$ yield
\[\int_D\mu^{-1}|F_3|^2\,\lesssim\int_D\mu|F|^2.\]
Summing these bounds on $F_1,F_2,F_3$ yields the desired conclusion.
\qed

\section{Quantitative homogenization}
This section is devoted to the proof of Theorem~\ref{th:hom-quant}.

\begin{proof}[Proof of Theorem~\ref{th:hom-quant}]
First consider a cut-off function $\eta_\e\in C^\infty_c(\R^d;[0,1])$ supported in $U$ such that $\eta_\e$ is constant inside the inclusions $\{\e I_n\}_n$, and $\eta_\e|_{\e I_n}=0$ for all $n\notin\Nc_\e(U)$.  In particular, $\Ic_\e(U)$ coincides with $\e\Ic$ in the support of $\eta_\e$. In addition, given $5\le R\le\frac1\e$ (to be later optimized depending on $\e$), we assume that  $\eta_\e=1$ in $U\setminus\partial_{\e R}U$ and $|\nabla\eta_\e|\lesssim(\e R)^{-1}$, where we use the notation \mbox{$\partial_{\e R}U:=\{x\in U:\dist(x,\partial U)<\e R\}$} for the fattened boundary.

\medskip
\step1 Two-scale expansion and representation of the error.\\
Let $(u_\e,P_\e)$ denote the solution of the heterogeneous Stokes equation~\eqref{eq:Stokes}, and let $(\bar u,\bar P)$ be the solution of the corresponding homogenized equation~\eqref{eq:Stokes-hom}.
The pressures $P_\e$ and~$\bar P$ are chosen such that $\int_{U}P_\e\mathds1_{\R^d\setminus\Ic_\e(U)}=\int_{U}\bar P=0$.
In terms of the corrector $(\psi,\Sigma)$, we consider the two-scale expansions
\[u_\e\leadsto\bar u+\e\eta_\e\psi_E(\tfrac\cdot\e)\partial_E\bar u,
\qquad P_\e\leadsto\bar P+\eta_\e\bb:\D(\bar u)+\eta_\e(\Sigma_E\mathds1_{\R^d\setminus\Ic})(\tfrac\cdot\e)\partial_E\bar u.\]
Given arbitrary constants $P_{\e,*}\in\R$ and $\{P_{\e,n}\}_n\subset\R$ (that will be made explicit later in the proof), we modify the pressure $P_\e$ into
\[P_\e':=(P_\e+P_{\e,*})\mathds1_{U\setminus\Ic_\e(U)}+\sum_{n\in\Nc_\e(U)}P_{\e,n}\mathds1_{\e I_n},\]
and we then consider the following two-scale expansion errors in $U$,
\begin{eqnarray*}
w_\e&:=&u_\e-\bar u-\e\eta_\e\psi_E(\tfrac\cdot\e)\partial_E\bar u,\\
Q_\e&:=&P_\e'-\bar P-\eta_\e\bb:\D(\bar u)-\eta_\e(\Sigma_E\mathds1_{\R^d\setminus\Ic})(\tfrac\cdot\e)\partial_E\bar u.
\end{eqnarray*}
Arguing as in Substep~1.2 of the proof of Proposition~\ref{p1}, cf.~\eqref{e.p1.5},
we find that $(w_\e,Q_\e)$ satisfies the following equation in the weak sense in $U$,
\begin{align}\label{eq:error-wPeps}
&\quad-\triangle w_\e+\nabla Q_\e\,=\,(\lambda-\mathds1_{\Ic_\e(U)})f\\
&\hspace{1cm}-\sum_{n\in\Nc_\e(U)}\delta_{\e\partial I_n}\sigma(u_\e,P_\e+P_{\e,*}-P_{\e,n})\nu
-\Div\big((\eta_\e\partial_E\bar u)J_E(\tfrac\cdot\e)\mathds1_{\e\Ic}\big)\nonumber\\
&+\Div\Big(2(1-\eta_\e)(\Id-\Bb)\D(\bar u)+2\e(\psi_E\otimes_s-\zeta_E)(\tfrac\cdot\e)\nabla(\eta_\e\partial_E\bar u)-\e\Id(\psi_E(\tfrac\cdot\e)\cdot\nabla)(\eta_\e\partial_E\hat u)\Big).\nonumber
\end{align}
In order to quantify the almost sure weak convergence $\mathds1_{\Ic_\e(U)}\cvf\lambda$ in $\Ld^2(U)$ in the first right-hand side term, we define a new corrector $\theta:=\nabla\gamma$ as the unique solution of the following infinite-volume problem:
\begin{enumerate}[\quad$\bullet$]
\item Almost surely, $\theta=\nabla\gamma$ belongs to $\Ld^2_\loc(\R^d)^d$ and satisfies
\[\Div(\theta)=\triangle\gamma=\mathds1_{\Ic}-\lambda,\qquad\text{in $\R^d$}.\]
\item The field $\nabla\theta=\nabla^2\gamma$ is stationary, has vanishing expectation, has finite second moment, and $\theta$ satisfies the anchoring condition $\fint_B\theta=0$ almost surely.
\end{enumerate}
Under the mixing condition~\ref{Mix+}, along the lines of the proof of Theorem~\ref{th:cor+} (but noting that no buckling is needed here as the corrector problem is linear with respect to randomness),
the following moment bounds are easily checked to hold for all $q<\infty$,
\begin{equation}\label{eq:bnd-gam}
\|\nabla\theta\|_{\Ld^q(\Omega)}\,\lesssim_q\,1,
\qquad
\|\theta(x)\|_{\Ld^q(\Omega)}\,\lesssim_q\,\mu_d(|x|).
\end{equation}
In terms of this corrector, recalling that $\Ic_\e(U)$ coincides with $\e\Ic$ in the support of~$\eta_\e$, the first right-hand side term in~\eqref{eq:error-wPeps} can be decomposed as
\begin{eqnarray*}
(\lambda-\mathds1_{\Ic_\e(U)})f&=&(\lambda-\mathds1_{\Ic_\e(U)})(1-\eta_\e)f+(\lambda-\mathds1_{\e\Ic})\eta_\e f\\
&=&(\lambda-\mathds1_{\Ic_\e(U)})(1-\eta_\e)f-\Div(\eta_\e f\otimes \e\theta(\tfrac\cdot\e))+\nabla(\eta_\e f)\,\e\theta(\tfrac\cdot\e).
\end{eqnarray*}
Inserting this into~\eqref{eq:error-wPeps}, we are led to the following equation for $(w_\e,Q_\e)$ on $U$,
\begin{multline}\label{eq:error-wPeps+}
-\triangle w_\e+\nabla Q_\e\,=\,(\lambda-\mathds1_{\Ic_\e(U)})(1-\eta_\e)f+\nabla(\eta_\e f)\,\e\theta(\tfrac\cdot\e)\\
-\sum_{n\in\Nc_\e(U)}\delta_{\e\partial I_n}\sigma(u_\e,P_\e+P_{\e,*}-P_{\e,n})\nu
-\Div\big((\eta_\e\partial_E\bar u)J_E(\tfrac\cdot\e)\mathds1_{\e\Ic}\big)\\
+\Div\Big(2(1-\eta_\e)(\Id-\Bb)\D(\bar u)-\eta_\e f\otimes \e\theta(\tfrac\cdot\e)
+2\e(\psi_E\otimes_s-\zeta_E)(\tfrac\cdot\e)\nabla(\eta_\e\partial_E\bar u)\\
-\e\Id(\psi_E(\tfrac\cdot\e)\cdot\nabla)(\eta_\e\partial_E\hat u)\Big).
\end{multline}

\step2 Conclusion.\\
We repeat the argument for~\eqref{e.p1.9concl} in Step~2 of the proof of Proposition~\ref{p1}, now without weight, starting from equation~\eqref{eq:error-wPeps+} instead of~\eqref{e.p1.5}. More precisely, we truncate $w_\e$ to make it affine in the inclusions, we test~\eqref{eq:error-wPeps+} with this truncated version of $w_\e$, we take advantage of boundary conditions, and we estimate the different terms.
Compared to equation~\eqref{e.p1.5}, the only new part here stems from the first two right-hand side terms in~\eqref{eq:error-wPeps+}, for which we simply appeal to Poincaré's inequality: as $w_\e\in H^1_0(U)^d$, we can estimate for any test function $g\in \Ld^2(U)^d$,
\[\Big|\int_Ug\cdot w_\e\Big|\,\le\,\Big(\int_U|g|^2\Big)^\frac12\Big(\int_U|w_\e|^2\Big)^\frac12\,\lesssim\,\Big(\int_U|g|^2\Big)^\frac12\Big(\int_U|\nabla w_\e|^2\Big)^\frac12.\]
In this way, for a suitable choice of the constants $P_{\e,*}$ and $\{P_{\e,n}\}_n$, we arrive at the following estimate,
\begin{multline*}
\int_{U}|(\nabla w_\e,Q_\e)|^2
\,\lesssim\,\int_{U}(1-\eta_\e)^2|(f,\nabla\bar u)|^2
+\e^2\int_U|\theta(\tfrac\cdot\e)|^2\big(|\nabla f|^2+|\nabla\eta_\e|^2|f|^2\big)\\
+\e^2\int_U\big(1+|(\psi,\zeta,\nabla\psi,\Sigma\mathds1_{\R^d\setminus\Ic})(\tfrac x\e)|^2\big)\Big(\sup_{B_{4\e}(x)}\big(|\nabla^2\bar u|^2+|\nabla\eta_\e|^2|\nabla\bar u|^2\big)\Big)\,dx.
\end{multline*}
Taking the $\Ld^q(\Omega)$ norm, using corrector estimates of Theorem~\ref{th:cor}, as well as~\eqref{eq:bnd-gam},
recalling that $1-\eta_\e$ and $\nabla\eta_\e$ are supported on the fattened boundary $\partial_{\e R}U$, noting that the latter has volume $|\partial_{\e R}U|\lesssim \e R$, and recalling that \mbox{$|\nabla\eta_\e|\lesssim(\e R)^{-1}$},
we deduce for all $q<\infty$,
\begin{equation*}
\expec{\Big(\int_{U}|(\nabla w_\e,Q_\e)|^2\Big)^q}^\frac1q
\,\lesssim_q\,
\big(\e R+\e^2 \mu_d(\tfrac1\e)^2\tfrac1{\e R}\big)\|(f,\nabla\bar u)\|_{W^{1,\infty}(U)}^2.
\end{equation*}
Next, decomposing
\begin{eqnarray*}
w_\e&:=&\big(u_\e-\bar u-\e\psi_E(\tfrac\cdot\e)\partial_E\bar u\big)+\e(1-\eta_\e)\psi_E(\tfrac\cdot\e)\partial_E\bar u,\\
Q_\e\mathds1_{U\setminus\Ic_\e(U)}&:=&\big(P_\e+P_{\e,*}-\bar P-\bb:\D(\bar u)-(\Sigma_E\mathds1_{\R^d\setminus\Ic})(\tfrac\cdot\e)\partial_E\bar u\big)\mathds1_{U\setminus\Ic_\e(U)}\\
&&\qquad+(1-\eta_\e)\big(\bb:\D(\bar u)+(\Sigma_E\mathds1_{\R^d\setminus\Ic})(\tfrac\cdot\e)\partial_E\bar u\big)\mathds1_{U\setminus\Ic_\e(U)},
\end{eqnarray*}
we deduce for all $q<\infty$,
\begin{multline*}
\big\|u_\e-\bar u-\e\psi_E(\tfrac\cdot\e)\partial_E\bar u\big\|_{\Ld^q(\Omega;H^1(U))}^2\\
+\inf_{\kappa\in\R}\big\|P_\e-\bar P-\bb:\D(\bar u)-(\Sigma_E\mathds1_{\R^d\setminus\Ic})(\tfrac\cdot\e)\partial_E\bar u-\kappa\big\|_{\Ld^q(\Omega;\Ld^2(U\setminus\Ic_\e(U)))}^2\\
\,\lesssim_q\,
\big(\e R+\e^2 \mu_d(\tfrac1\e)^2\tfrac1{\e R}\big)\|(f,\nabla\bar u)\|_{W^{1,\infty}(U)}^2.
\end{multline*}
Choosing $\e R=\e\mu_d(\tfrac1\e)$,
and using the regularity theory for the steady Stokes equation~\eqref{eq:Stokes-hom}, cf.~\cite[Section~IV]{Galdi},
this yields the conclusion~\eqref{eq:hom-quant0}.
\end{proof}

\section*{Acknowledgements}
The authors warmly thank Felix Otto for some enlightening comments on the structure of correctors in connection with Lemma~\ref{lem:ext-cor}.
MD acknowledges financial support from the CNRS-Momentum program,
and AG from the European Research Council (ERC) under the European Union's Horizon 2020 research and innovation programme (Grant Agreement n$^\circ$~864066).

\bibliographystyle{plain}
\bibliography{biblio}

\end{document}